\documentclass{article}
\usepackage{amsmath}
\usepackage{graphicx}     
\usepackage{amsfonts}
\usepackage{latexsym}
\pdfoutput=1
\usepackage{upgreek}
\usepackage{amssymb}
\usepackage{amsthm}
\usepackage{setspace}
\usepackage{fancyhdr}
\usepackage{mathrsfs}
\usepackage{framed}
\usepackage{bm}
\usepackage{bbm}
\usepackage{dsfont}
\usepackage[margin=1.25in]{geometry}

\newtheorem{theorem}{Theorem}[section]
\newtheorem{lemma}[theorem]{Lemma}
\newtheorem{remark}[theorem]{Remark}

\newtheorem{assumptions}[theorem]{Assumptions}
\newtheorem{assumption}[theorem]{Assumption}

\theoremstyle{definition}
\newtheorem{algor}{Algorithm}

\newcommand{\yn}{y}
\newcommand{\tr}{{\rm {Tr}}}
\newcommand{\Var}{\mbox{Var}}

\newcommand{\G}{\mathcal{N}}
\newcommand{\R}{\mathbb{R}}
\newcommand{\rp}{\mathbb{P}}

\newcommand{\D}{\mathcal{D}}

\newcommand{\hiddensection}[1]{
\stepcounter{section}
\section*{}}
\newcommand{\1}{\mathds{1}}

\newcommand{\N}{\mathbb{N}}

\newcommand{\A}{\mathcal{A}}
\newcommand{\bA}{\bm{\mathcal{A}}}
\newcommand{\C}{\mathcal{C}}
\newcommand{\bC}{\bm{\mathcal{C}}}
\newcommand{\bz}{\bm{z}}
\newcommand{\bI}{\bm{I}}
\newcommand{\bP}{{\bf P}}
\newcommand{\bG}{\bm{\mathcal{G}}}
\newcommand{\bxi}{\bm{\xi}}
\newcommand{\bv}{\bm{v}}
\newcommand{\bx}{\bm{x}}
\newcommand{\bw}{\bm{w}}
\newcommand{\bme}{\bm{m}}
\newcommand{\brho}{\bm{\rho}}
\newcommand{\bT}{\bm{T}}
\newcommand{\be}{\bm{e}}
\newcommand{\bSigma}{\bm{\Sigma}}
\newcommand{\bK}{\bm{K}}
\newcommand{\bmeta}{\bm{\eta}}

\newcommand{\X}{\mathcal{X}}

\newcommand{\Y}{\mathcal{Y}}

\newcommand{\norm}[1]{\big\|#1\big\|}
\newcommand{\smnorm}[1]{\|#1\|}
\newcommand{\sumj}{\sum_{j=1}^{\infty}}

\newcommand{\E}{\mathbb{E}}
\newcommand{\pr}[2]{\big\langle#1,#2\big\rangle}
\newcommand{\envelope}{(\raisebox{-.5pt}{\scalebox{1.45}{\Letter}}\kern-1.7pt)}

\newcommand{\hl}{\lambda}
\newcommand{\I}{{\rm I}}

\newcommand{\bu}{\bm{u}}
\newcommand{\by}{\bm{y}}

\newcommand{\alp}{\alpha}
\newcommand{\bet}{\beta}

\newcommand\smallO[1]{
  \mathchoice
    {{\scriptstyle\mathcal{O}}}
    {{\scriptstyle\mathcal{O}}}
    {{\scriptscriptstyle\mathcal{O}}}
    {\scalebox{.7}{$\scriptscriptstyle\mathcal{O}$}}
  {\left(#1\right)}}

\newcommand{\de}[1]{\delta^{(#1)}}
\newcommand{\si}[1]{\lambda}

\newcommand{\sumn}{\sum_{j=1}^{N}}

\newcommand{\Ga}{{\rm {Gamma}}}

\newcommand{\ad}{\upalpha_0}
\newcommand{\bd}{\upbeta_0}

\begin{document}
\title{\textbf{Analysis of the Gibbs sampler for hierarchical inverse problems}}
\author{Sergios Agapiou \\\small Mathematics Institute, University of
Warwick\\ \small Coventry CV4 7AL, United Kingdom\\\small S.Agapiou@warwick.ac.uk\\\\ Johnathan M. Bardsley \\\small Department of Mathematics Sciences, University of Montana\\\small Missoula, MT, 59812-0864 USA\\\small BardsleyJ@mso.umt.edu\\\\Omiros Papaspiliopoulos \\\small Department of Economics, Universitat Pompeu Fabra\\ \small Ramon Trias Fargas 25-27, 08005 Barcelona, Spain\\\small Omiros.Papaspiliopoulos@upf.edu\\\\Andrew M. Stuart\\\small Mathematics Institute, University of
Warwick\\\small Coventry CV4 7AL, United Kingdom\\\small A.M.Stuart@warwick.ac.uk
} 
\date{}
\maketitle

{\bf Abstract} Many inverse problems arising in applications come from continuum models
where the unknown parameter is a field. In practice the unknown field
is discretized resulting in a problem in $\R^N$, with an understanding
that refining the discretization, that is increasing $N$, will often be desirable.
In the context of Bayesian inversion this situation suggests the importance
of two issues:  (i) defining hyper-parameters
in such a way that they are interpretable in the continuum limit $N \to \infty$ and so that their values may be compared between different discretization levels; (ii) understanding the efficiency of algorithms for probing
the posterior distribution, as a function of large $N.$
Here we address these two issues in the context of linear inverse problems
subject to additive Gaussian noise within a hierarchical modelling
framework based on a Gaussian prior for the unknown
field and an inverse-gamma prior for a hyper-parameter, namely the amplitude of the prior variance. The structure of the model is such that the Gibbs sampler
can be easily implemented for probing the posterior distribution. Subscribing to the dogma that one should think infinite-dimensionally before implementing in finite dimensions, we present function space intuition and provide rigorous theory showing that as $N$ increases, the component of the Gibbs sampler for sampling the amplitude of the prior variance becomes increasingly slower. We discuss a reparametrization of the
prior variance that is robust with respect to the increase in dimension; we give numerical experiments which exhibit that our reparametrization prevents the slowing down.
Our intuition on the behaviour of the prior hyper-parameter, with and without reparametrization, is sufficiently general to include a broad class of nonlinear inverse problems as well as other families of hyper-priors.

\vspace{0.4cm}



\vspace{0.4cm}


\pagestyle{myheadings}
\thispagestyle{plain}
\markboth{S. Agapiou, J. M. Bardsley, O. Papaspiliopoulos, A. M. Stuart}
{Analysis of Gibbs sampler for inverse problems}


\section{Introduction}\label{ch3:sec:int}
We consider the possibly  nonlinear inverse problem of recovering an unknown parameter $\bu\in\X$ from a noisy 
indirect observation $\by\in\Y$. We work in a framework where   $\X$ is an infinite-dimensional 
separable Hilbert space with inner product $\pr{\cdot}{\cdot}$ and
norm $\smnorm{\cdot}$, and $\Y$ is also a separable Hilbert space. We
will be especially interested in the case $\Y=\X$ or $\Y=\R^M$. The
unknown parameter and the observation are related through an 
 additive noise model \begin{equation}
\label{ch3:eq:1} \by=\bG(\bu)+\bmeta,\end{equation} where $\bG:\X\to\Y$ is the forward map which is assumed to be continuous, and $\bmeta$ 
is Gaussian noise \begin{align}\label{ch3:eq:2}\bmeta\sim\G(0,\lambda^{-1}\bC_1).\end{align} The linear operator $
\bC_1:\Y\to\Y$ is bounded and positive definite and $\lambda>0$ models the noise level;
we do not enforce that $\bC_1$ is trace-class, thereby allowing the case of
Gaussian white noise where it is the identity.

We adopt a Bayesian approach with a Gaussian prior on the unknown parameter $\bu$ 
\begin{align}\label{ch3:eq:3}\bu|\delta\sim\G(0,\delta^{-1}\bC_0),\end{align}
 where $\bC_0:\X\to\X$ is a positive definite and trace-class operator and $\delta>0$ models the amplitude of the prior variance; the unknown $\bu$ is assumed to be independent of the noise $\bmeta$. The trace-class
assumption on $\bC_0$ ensures that draws from the prior on
$\bu|\delta$ are in $\X$. For a fixed $\bu$ the likelihood is
Gaussian, $\by|\bu,\delta\sim\G(\bG 
 (\bu),\lambda^{-1}\bC_1)$. We work under certain regularity
 conditions on the forward map $\bG$, which imply that the inverse
 problem is sufficiently ill-posed; in particular, for the noise model
 at hand, these conditions imply that the unknown $\bu$ is not
 perfectly identifiable from a single realization of the data. Under
 the additional assumption that the prior on $\bu|\delta$ is such that
 the regularity conditions on $\bG$ are satisfied in its support, it can be shown that  almost surely with respect to the data the posterior on $\bu|\by,\delta$ is well defined, non-degenerate and absolutely continuous with respect to the prior on $\bu|\delta$, \cite{AS10}.

In what follows, we consider the hyper-parameter $\delta$ as a part of
the inference problem, that is, we endow it with a prior
$\rp(\delta)$; this leads to a hierarchical Bayesian model. The potential for the use of hierarchical priors in inverse problems has been highlighted in \cite{KS05}, where the authors express the conviction that \emph{if a parameter is not known, it is a part of the inference problem}; see also \cite{CS08, CHPS09} where conditionally Gaussian hierarchical models have been considered in finite dimensional contexts. Returning to our setting, we note that of course in practice other aspects of the model, such as
parameters that control the regularity of the draws from the prior,
will also be part of the inference problem. Section \ref{ch3:sec:con}
discusses how the results of this paper can be extended to such
situations, but the focus here is the joint hierarchical inference on $\bu$ and
$\delta$.  Statistical inference is achieved by Markov chain Monte Carlo sampling from the
resulting full posterior on $\bu,\delta|\by$,  where by Bayes'
rule \[\rp(\bu,\delta|\by)\propto\rp(\by|\bu,\delta)\rp(\bu|\delta)\rp(\delta)\propto\rp(\bu|\by,\delta)\rp(\delta|\by).\]
A sufficient condition for this posterior to be well defined is that
the prior $\rp(\delta)$ is proper. 

Due to the nature of the pair $(\bu,\delta)\in\X\times
[0,\infty)$, sampling can be achieved by a two-component Metropolis-within-Gibbs (MwG)
algorithm.  There is a range of possible parametrizations for this MwG
algorithm, perhaps the most natural of which is the so-called
\emph{centered algorithm} (CA), 
\cite{PRS07}.  This scheme  alternates between simulating from
$\bu|\by,\delta$ and $\delta|\by,\bu$  using 
Metropolis-Hastings steps.  Each pair of such simulations is one
algorithmic iteration of a  prescribed number
$k_{max}$. For specific models the simulation from the two
conditionals can be done directly, without Metropolis-Hastings, in
which case the resultant algorithm is the Gibbs sampler. Note that the model structure implies that
$\delta$ and $\by$ are conditionally independent given $\bu$, that is
$\delta|\by,\bu\equiv \delta | \bu$. This is the defining property of the
so-called \emph{centered parameterisation} of a hierarchical model,
\cite{PRS07}.

In practice the inverse problem and the algorithm are discretized and Bayesian inference is
implemented in finite dimensions. We then have two sources of error in
the estimated posterior distribution: a) the approximation error due
to the discretization of the unknown and the forward problem, that is
the discretization bias, {discussed in a general Bayesian (non-hierarchical) inverse problem setting in  
\cite{CDS10}; b) the Monte Carlo error due to the use of a Markov
chain Monte Carlo method to sample the discretized posterior
distribution. Assuming that the discretization level of the unknown is
$N$, we have that the total error
is of the order  \begin{align}{\label{ch3:eq:toterr}}\frac{1}{N^s}+\frac{C(N)}{\sqrt{k_{max}}},\end{align}for
some $s>0$ which relates to the quality of approximation of the
unknown and forward problem,  and $C(N)$ which depends on the mixing
properties of the particular algorithm used to probe the
posterior. This picture allows the practitioner to get a rough idea
how to 
distribute the computational budget by balancing investments in
higher discretization levels with investments in longer chains in
order to achieve the desired error level in the estimated posterior
distribution. In reality, of course,  the constants that
multiply these rates will be relevant and hard to determine. 

 There are four principal
motivations for formulating the inverse problem and the simulation
algorithms in infinite
dimensions, while using consistent discretizations (in the sense of
numerical analysis, see subsection \ref{ch3:sec:dm}) for the numerical
implementation. First, such formulation is often more faithful to the
mathematical model that we wish to learn from the data. Second, 
it makes   the inference comparable
across different levels of discretization, so that the estimation of the model with increasing
values of $N$ corresponds to a reduction in the discretization bias at
the cost of additional computation. Third, the prior distribution on
hyperparameters, such as $\delta$, represents the same prior
beliefs across different levels of discretization. On the contrary, 
when the finite-dimensional model is not a consistent discretization
of an infinite-dimensional one, the prior on the hyperparameters might
contain an amount of information that depends on the level of
discretization chosen; see for example the last paragraph in subsection \ref{fd} below. Finally, practically useful algorithms
can be designed for moderate or even small values of $N$ by studying
their behaviour at the asymptotic limit $N \to \infty$.  In fact, it
is usually unrealistic to try to obtain
practically useful theoretical results on the convergence of Markov
chain Monte Carlo for sampling non-trivial targets, unless such
asymptotic regimes are constructed and invoked. This is precisely the
case with the Gibbs sampler and related MwG algorithms, which are
particularly hard to analyse (see for example
 \cite{stable}). Similarly, conceiving of Metropolis-Hastings methods in the infinite-dimensional limit leads to algorithms with provably dimension-independent convergence properties, whilst standard methods have convergence properties which degenerate with increased refinement of the discretization; see \cite{CRSW13} and discussion therein.

 In this paper we investigate theoretically and numerically the
 performance of MwG algorithms  in the asymptotic regime of large
 $N$. In order to have a mathematically tractable analysis, we
 focus on linear inverse problems, see subsection
 \ref{sec:linear}. For these models, and under a commonly adopted prior on
 $\delta$, the MwG becomes a Gibbs sampler.  We establish a result
 on the mean drift and diffusion of the $\delta$-chain in CA, which has the informal interpretation that $C(N)$ is of the order
 $N^{1/2}$. An immediate consequence of this result, is that in order to minimize the total error in (\ref{ch3:eq:toterr}), $k_{max}$ should be scaled like $N^{1+2s}$, whilst for algorithms for which $C(N)$ is uniformly bounded with respect to $N$, the same error level can be achieved by scaling $k_{max}$ like $N^{2s}$; we expect this to be the case for the non-centered algorithm proposed later in this section.  We emphasize that although we prove this result for the linear model and for a specific prior on $\delta$, a detailed understanding of the ideas underlying our proofs, indicates that most of the details of the model,
 including linearity, and the prior used on $\delta$, do not really
 affect the validity of our main finding, that is, that CA deteriorates
 with $N$. The fundamental reason why this algorithm becomes unusable
 for large $N$ is an absolute continuity property, 
 a high-level description of which we now provide. Note, however, that proving the result
 in such generality is definitely beyond the scope of this paper.

  In the infinite-dimensional limit, $\delta$ is an almost sure property of
$\bu|\delta\sim\G(0,\delta^{-1}\bC_0)$. This
means that a single draw of $\bu$ contains infinite information about
the value of $\delta$ that generated it. In measure-theoretic terms,
it means that the prior measures $\rp(\bu | \delta)$ and $\rp(\bu |
\delta')$ for $\delta \ne \delta'$ are mutually singular,
\cite[Remark 2.10]{DP05}. 
Recalling that we work under assumptions which imply that  $\bu|\by,\delta$ is
absolutely continuous with respect to $\bu|\delta$, we deduce that $\delta$ is also an almost sure property of
$\bu|\by,\delta$. As a result,  iterative simulation from the distributions, $\bu|\by,\delta$ and $\delta|\by,\bu$,
 will fail in ever changing the initial value
of $\delta$. On the other hand, recall that we also work under assumptions that
imply  that  $\bu$, hence $\delta$, are not perfectly identifiable
from the data. Therefore,  $\delta|\by$ is non-degenerate (provided
the prior is non-degenerate) and hence
any single value of $\delta$ has zero probability under the data. Concatenating, we have that when iteratively simulating from $\bu|\by,\delta$ and  $\delta|\by,\bu$, the values of $\bu$ will be
changing along the iterations, but will be in fact sampled from a subspace
which has probability zero under $\rp(\bu | \by)$. In other
words CA is \emph{reducible} in infinite dimensions  and
will fail to sample from $\bu,\delta | \by$.  Iterative conditional
sampling of the finite-dimensional approximation of $\bu, \delta|\by$,
will be able to obtain samples from the (approximated) posterior
distribution of $\delta$, but will suffer
from increasingly slow \emph{mixing}  as the discretization
level  $N$ increases. In fact, the dependence between the discretized
unknown parameter $u$ and $\delta$ increases with $N$, and becomes
infinitely strong in the limit $N\to \infty$; it is this dependence
that slows down the MwG.

In order to alleviate the undesirable effects of the strong dependence
between the prior on $\bu$ and $\delta$, using intuition from
\cite{PRS07, RS01}, we reparametrize the prior by writing
$\bu=\delta^{-\frac12}\bv$ where $\bv\sim\G(0,\bC_0)$ and $\delta\sim
\rp(\delta)$. This results in a MwG algorithm which alternates between
a step of updating $\bv|\by,\delta$ 
and a step of updating $\delta|\by,\bv$; this
is an example of a \emph{non-centered algorithm} (NCA),
\cite{PRS07}. Since $\bv$ and $\delta$ are now a priori independent,
and recalling that $\bu$ is not perfectly identified by the data,
the dependence of these two parameters is not perfect conditionally on
the data. Thus, 
the NCA is  \emph{irreducible} in infinite dimensions and is thus robust with respect to the discretization level
$N$.  Hence, for NCA we expect that $C(N)$ is uniformly
bounded with respect to $N$; we show numerical evidence in support of
this statement in section \ref{ch3:sec:sim}.

 \subsection{The linear case - modelling and notation}
\label{sec:linear}

 We will concentrate on the linear inverse problem
case with gamma priors on $\delta$ which has the convenient property
of conditional conjugacy. Specifically, 
we restrict our attention to the case $\bG=\bK$ where $\bK:\X\to\Y$ is
a bounded linear operator.  Then,  the posterior distribution  $\bu|\by,\delta$ is also Gaussian 
  \begin{align*}\bu|\by,\delta\sim\G(\bme_{\lambda,\delta}(\by),\bC_{\lambda,\delta});\end{align*} 
see \cite{AM84,LPS89} where formulae for the posterior mean and
covariance operator are provided.  When the prior distribution and the
noise are specified in terms of precision operators (that is, inverse
covariance operators)  the following expressions for the posterior mean
and precision are known to hold in a range of situations in \cite{ALS13, ASZ12}:
\begin{align}\bC_{\lambda,\delta}^{-1}&=\lambda \bK^\ast\bC_1^{-1}\bK+\delta\bC_0^{-1},\label{ch3:eq:prec}\\
\bC_{\lambda,\delta}^{-1}\bme_{\lambda,\delta}(\by)&=\lambda \bK^\ast
 \bC_1^{-1}\by.\label{ch3:eq:mean}
 \end{align}

In order to introduce discretizations and their
connection to the continuum limit we need some additional notation;
subsection \ref{ch3:sec:dm} gives specific examples of continuum
models and their discretizations, where the notation introduced below
is put into practice. In order to avoid a notational overload, in the
development of the theory we assume that $\X=\Y$ and that the
discretization levels of the unknown and the data are the same. This
assumption is not crucial to our results and we refer to the PhD
thesis \cite[section 4.5]{SA13} for the more general statements. Furthermore, in section \ref{ch3:sec:sim}, we present numerical examples corresponding to both $\Y=\X$ with an increasing discretization level which is the same for both the unknown and the data, and $\Y=\R^M$ for some fixed $M$, whilst the dimension of the discretization of the unknown is increased. The case $\Y=\X$ arises for example when we observe the whole unknown function subject to blurring and noise, while the case $\Y=\R^M$ can arise when we have available blurred and noisy observations of the unknown at only $M$ spatial locations (see subsection \ref{fd}). The two cases can also arise if we work in the spectral domain, depending on the availability of observations of a full or only a partial spectral expansion of a blurred noisy version of the unknown. 

We denote by $\pr{\cdot}{\cdot}_{\R^N}$ and $\smnorm{\cdot}_{\R^N}$ the 
 (possibly scaled) Euclidean inner product and norm in $\R^N$ and by $\smnorm{\cdot}_{2,N}$ the 
 induced operator norm for $N\times N$ matrices. Throughout the paper we assume that this norm and inner product on 
 $\R^N$ are scaled so that,
formally, the large $N$ limit recovers the norm and inner product on the Hilbert space when,
for example, spectral or finite difference approximations are made.  Henceforward, we use boldface and
regular typeface letters to distinguish between infinite and finite-dimensional objects respectively. We assume that we have a way of 
computing discretizations ${\yn}\in\R^N$ of the observation $\by$ and replace the operators $\bK, \bC_0$ 
and $\bC_1$ by the $N\times N$ matrices $K, \C_0$ and $\C_1$ respectively, which arise from a 
consistent, in the sense of numerical analysis,
family of approximations of the corresponding operators. In this finite-dimensional  setting, the unknown
is $u \in \R^N$ and it is assigned a finite-dimensional Gaussian
prior, $u|\delta \sim \G(0,\delta^{-1}\C_0)$. The noise distribution has Lebesgue
density and the corresponding  log-likelihood 
is quadratic in $u$. Thus, standard
Bayesian linear theory (see e.g.  \cite{LS72}) implies  that the posterior is also Gaussian,
$u|y,\delta\sim\G(m_{\lambda,\delta}(y), \C_{\lambda,\delta})$, where
$m_{\lambda,\delta}(y)$ and $\C_{\lambda,\delta}^{-1}$  solve equations
(\ref{ch3:eq:prec}) and (\ref{ch3:eq:mean}) where the boldface
infinite-dimensional quantities are replaced by the corresponding
finite-dimensional regular typeface quantities.

{Bayesian modelling for finite-dimensional approximations of
linear inverse problems using  Gaussian priors and noise models was
recently carried out in \cite{JB12}. The approach consisted in
simultaneous inference for the unknown $u$ and the hyper-parameters 
$\lambda$ and $\delta$. We will concentrate on simultaneous inference on $u$ and $\delta$ only, since $\lambda$ can be efficiently estimated from a single high dimensional realization of the data, for example using quadratic variation. We again refer the interested reader to the PhD thesis \cite[Chapter 4]{SA13} for theoretical and numerical results on the large $N$ behaviour of $\lambda$ when considered as part of the inference problem; we stress here that for low-dimensional data, the inference on $\lambda$ is non-trivial. In \cite{JB12}, a standard conditionally conjugate prior was
used for the hyper-parameter, 
$\delta\sim\Ga(\ad,\bd)$, which in this
type of finite-dimensional Gaussian models is known to lead to a gamma
conditional posterior distribution, \cite[Chapter 5.2]{BS09}\begin{equation}\label{ch3:eq:int4}\delta|{\yn},u\sim\Ga(\ad+
\frac{N}2, \bd+\frac12\norm{\C_0^{-\frac12}u}_{\R^N}^2).\end{equation}}

The inference for this model was carried out using CA which in this case is a Gibbs sampler (see Algorithm \ref{ch3:algstd} in section \ref{ch3:sec:met} below), since both conditional distributions $u|y,\delta$ and $\delta|y,u$ belong to known
parametric families and can be sampled directly. One of the main aims of this paper is to analyze the convergence of this algorithm in the large $N$ limit. We also aim to exhibit via numerical simulations, the deterioration of the performance of CA in the large $N$ limit, as well as the benefits of reparametrizing the prior and using the corrseponding NCA (see Algorithm \ref{ch3:algrep} in section \ref{ch3:sec:met} below).

\subsection{Examples of consistent discretizations}\label{ch3:sec:dm}
In order to aid the understanding of the paper and in anticipation of the subsequent developments, we briefly describe two methods for passing from the continuum infinite-dimensional model in $\X$ to a discrete model in $\R^N$. Here and elsewhere in the paper, we define a Gaussian white noise in $\R^N$ to be a random variable $\zeta$ given as
$\zeta=\sumn\zeta_je_j,$
where $\{e_j\}_{j=1}^N$ is a basis in $\R^N$ which is orthonormal in the possibly scaled Euclidean inner product $\pr{\cdot}{\cdot}_{\R^N}$, and $\{\zeta_j\}_{j\in\N}$ is a sequence of independent standard Gaussian random variables in $\R$.
\subsubsection{Spectral truncation}\label{sp}
Let $\{\be_j\}_{j\in\N}$ be a complete orthonormal basis in $\X$. An element $\bw\in\X$ can be identified with the sequence $\{w_j\}_{j\in\N}$ of coefficients $w_j:=\pr{\bw}{\be_j}$ and by Parseval's identity the Hilbert space norm of $\bw$ can be replaced by the $\ell_2$-norm of the sequence of coefficients (similarly for the inner product). One can then discretize $\bw$ by replacing it with $w\in{\rm{span}}\{\be_1,...,\be_N\}$ which is identified with the truncated sequence of coefficients $\{w_1,...,w_N\}\in\R^N$. The $\ell_2$-norm and inner product are then replaced by the Euclidean norm and inner product.
Let $\bSigma:\X\to\X$ be a bounded operator which is diagonalizable in $\{\be_j\}_{j\in\N}$ with eigenvalues $\{\mu_j^{\Sigma}\}_{j\in\N}$. The operator $\bSigma$ can be identified with the sequence $\{\mu_j^{\Sigma}\}_{j\in\N}$ and we can discretize $\bSigma$ at level $N$ by replacing it with the  finite rank operator which is identified with the $N
\times N$ diagonal matrix $\Sigma=\rm{diag}(\mu_1^{\Sigma},...,\mu_N^{\Sigma})$. If $\bx\sim\G(0,\bSigma)$ is a Gaussian random variable in $\X$, we can discretize by replacing $\bx$ with $x\in{\rm{span}}\{\be_1,...,\be_N\}$ which is identified with a random variable with distribution $\G(0,\Sigma)$ in $\R^N$. Equivalently, $x$ is identified with $\Sigma^\frac12x_0$ where $x_0$ is a Gaussian white noise in $\R^N$ with respect to the standard orthonormal basis of Euclidean space. 
For more details see subsection \ref{ch3:ssec:diag}.

\subsubsection{Finite differences approximation}\label{fd}
Let $\X=L^2(\I), \I=(0,1)$,  and denote by $\bA_0$ the negative Laplacian densely defined on $\X$ with domain $H^2(\I)\cap H^1_0(\I)$, that is with Dirichlet boundary conditions. We discretize the domain $\I$ using a grid of $N$ equally spaced points $\{\frac{1}{N+1},...,\frac{N}{N+1}\}$; we can restrict our attention to the interior points due to the Dirichlet boundary conditions.  We define the inner product and norm in $\R^N$ \begin{align*}\pr{u}{v}_{\R^N}=\frac{1}{N+1}\sumn u_jv_j \quad\mbox{and} \quad\norm{u}_{\R^N}=\bigg(\frac{1}{N+1}\sumn u_j^2\bigg)^\frac12.\end{align*} Note that the natural orthonormal basis on the $N$-dimensional space of grid points with respect to the above norm and inner product is $\{e_j\}_{j=1}^N$, with $e_j=\{\sqrt{N+1}\delta_{ij}\}_{i=1}^N$, where $\delta_{ij}$ is Kronecker's delta.
 For a function $\bu$ in $\I$ which vanishes on the boundary, we consider its discretization on the grid, hence  $u_j=\bu(\frac{j}{N+1})$. We thus have a discrete approximation of $\X$ with norm and inner product which are the discrete analogues of the $L^2$-norm and inner product. We use finite differences to discretize $\bA_0$. In particular, we replace $\bA_0$ by the $N\times N$ matrix 
\begin{equation*}\A_0=(N+1)^2 \begin{bmatrix} 
\;\;2 & -1 & \;\;0 & \;\;\hdots & \;\;0 \\
-1& \;\;2 & -1 & \;\;\ddots & \;\;\vdots\\
\;\;0 &\;\;\ddots & \;\;\ddots & \;\;\ddots &\;\;0\\
\;\;\vdots &\;\;\ddots &-1 & \;\;2 &-1 \\
\;\;0&\;\;\hdots&\;\;0&-1&\;\;2
\end{bmatrix}.
\end{equation*}
If $\bz\sim\G(0,\bSigma)$ is a Gaussian random variable in $\X$ where
$\bSigma$ is a function of $\bA_0$ (for example a power), we discretize $\bz$ by considering the $N$-dimensional random variable $z=\Sigma^\frac12z_0$ defined on the grid, where $\Sigma$ is the corresponding function of the matrix $\A_0$ and $z_0$ is a Gaussian white noise with respect to $\{e_j\}_{j=1}^N$.

In subsection \ref{ch3:nex2} we consider subsampling at a set of $M$ equally spaced points amongst the $N$ grid points, where $\frac{N+1}{M+1}$ is a nonnegative power of 2. To this end, we define the matrix $P\in\R^{M\times N}$ by
\begin{align*}P_{i,j}=\left\{\begin{array}{ll} 1,  & \mbox{if} 
\;\mbox{$j=i\frac{N+1}{M+1}$} 
                                     \\ 0, &\mbox{otherwise}.
                                    \end{array}\right.
\end{align*} The matrix $P$ maps the vector of values on the fine grid $\{{\bf u}(\frac{j}{N+1})\}_{j=1}^N$ to the subsampled vector of the values on the coarse grid $\{{\bf u}(\frac{i}{M+1})\}_{i=1}^M$. 
If we fix $M$ and let $N$ increase, then $P$ corresponds to a discretization of the operator ${\bf P}:C(\I)\to\R^M$ defined as $M$ pointwise evaluations at the points $x_i=\frac{i}{M+1}, \;i=1,...,M$, $({\bf Pu})_i={\bf u}(\frac{i}{M+1})$, for any continuous 
function ${\bf u}$. A formal calculation suggests that the adjoint of the pointwise evaluation operator at $x\in \I$, is an
 operator mapping $r\in\R$ to $r\delta_x$, where $\delta_x$ is the Dirac distribution at $x$. This suggests that  ${\bf P^\ast}:\R^M\to C(\I)$, maps $r\in\R^M$ to the linear combination of Dirac distributions $r_1\delta_{x_1}+...+r_M\delta_{x_M}$. At the same
  time the matrix $P^T\in\R^{N\times M}$ maps the vector of values on the coarse grid $\{{\bf u}(\frac{i}{M+1})\}_{i=1}^M$ to a 
  vector in $\R^N$ which is zero everywhere except from the $i\frac{N+1}{M+1}$-th components where it is equal to ${\bf u}(\frac{i}{M+1})$, $i=1,...,M$. Combining, and in order to capture the effect of the Dirac distribution at the locations $\frac{i}{M+1}$, we have that ${\bf P^\ast}$ should be discretized using the matrix $(N+1)P^T$.
  
Note that if $\G(0,\delta^{-1}\mathcal{T}^{-1})$ is used as a prior on $u|\delta$ at level $N$, where $\mathcal{T}$ is the $N\times N$ tridiagonal matrix in the definition of $\A_0$, then this corresponds to having a prior with covariance matrix $(N+1)^2\delta^{-1}\A_0$. In particular, if $\delta\sim\Ga(\ad,\bd)$, then we have that $\frac{1}{(N+1)^2}\delta\sim\Ga(\ad,(N+1)^2\bd)$ where in the large $N$ limit the last gamma distribution converges to a point mass at zero, while $\A_0$ approximates $\bA_0$. This means that as $N\to\infty$ the correlation structure of the prior is described by the limiting $\bA_0$ but with an amplitude which becomes larger and larger with ever increasing confidence; in other words as $N$ grows the prior on $u|\delta$ looks increasingly flatter.

\subsection{Notation} 
We use subscripts to make explicit the dependence of the $\delta$-chain on the discretization 
level $N$ and superscripts to denote the iteration number in the Gibbs sampler. 
For a random variable $x$ which depends on the mutually independent random variables $z_1$ 
and $z_2$, we use $\E^{z_1}[x]$ to denote the expectation of $x$ with 
respect to $z_1$ for fixed $z_2$. We use $x_1\stackrel{\mathcal{L}}{=}
x_2$ to denote that the random variables $x_1$ and $x_2$ have the same law. Finally, for two sequences 
of positive numbers $\{s_j\}$ and $\{t_j\}$, we use the notation $s_j\asymp t_j$ to mean that $s_j/t_j$ is 
bounded away from zero and infinity uniformly in $j$. 

\subsection{Paper structure}
In the next section we present the centered Gibbs and non-centered MwG algorithms in our assumed linear conjugate setting; we also discuss the option of integrating $u$ out of the data likelihood and the resulting marginal algorithm. In section \ref{ch3:sec:main} we present our main result on the deterioration of the centered Gibbs sampler which holds 
under certain assumptions made at the discrete level 
and which are stated explicitly in the same section. 
Our discrete level assumptions are typically inherited 
from Assumptions \ref{ch3:infass1} on the underlying infinite-dimensional 
model also stated in section \ref{ch3:sec:main}, when consistent numerical discretizations are used. In section \ref{ch3:sec:ex} we exhibit three classes of linear inverse problems satisfying our assumptions on the 
underlying infinite-dimensional model. For the first two of these classes, that is a class of mildly ill-posed 
and a class of severely ill-posed linear inverse problems both in a simultaneously diagonalizable setting, 
we also explicitly prove that our discrete level assumptions are inherited from the infinite-dimensional 
assumptions when discretizing via spectral truncation (see subsections \ref{ch3:ssec:diag} and 
\ref{ch3:ssec:sev}). In section \ref{ch3:sec:sim} we present numerical evidence supporting our theory and intuition on the deterioration of the centered algorithm and the merits of using the non-centered algorithm, using both spectral truncation (subsection 
\ref{ch3:nex1}) and discretization via finite differences and subsampling (subsection \ref{ch3:nex2}). The main body of the paper ends with concluding remarks in section \ref{ch3:sec:con}, 
while the Appendix in section \ref{ch3:sec:ap} contains the proof of our main result as well as several technical lemmas.


\section{Sampling algorithms}\label{ch3:sec:met}
We now present in more detail the different algorithms for 
sampling $u,\delta|y$ in linear hierarchical inverse problems, and provide a
high-level comparison of their relative merits in the asymptotic regime of large $N$. 

\subsection{Centered Algorithm (CA)}\label{CA}
We first provide pseudo-code for the most natural algorithm for sampling
$u,\delta|y$ in this linear conjugate setting, that is the centered
Gibbs sampler used in \cite{JB12} and discussed in section \ref{ch3:sec:int}. 

\begin{framed}
\begin{algor}\label{ch3:algstd}{\ }
\emph{\begin{enumerate}
\item[0.] Initialize $\delta^{(0)}$ and set $k=0;$
\item[1.] $u^{(k)}\sim \G\big(m_{\lambda,\delta^{(k)}}({\yn}),\C_{\lambda,\delta^{(k)}}\big);$
\item[2.]  $\delta^{(k+1)}\sim \Ga(\ad+\frac{N}2,\bd+\frac12\norm{\C_0^{-\frac12}u^{(k)}}_{\R^N}
^2);$
\item[3.] Set $k=k+1$. If $k<k_{max}$ return to step 1, otherwise stop.
\end{enumerate}}
\end{algor}
\end{framed}

\subsection{Non-centered Algorithm (NCA)}\label{NCA} 
We now formulate in more detail the non-centered algorithm introduced in section
\ref{ch3:sec:int}. We define the algorithm in the infinite-dimensional setting, and
then discretize it. We reparametrize the prior by writing
$\bu=\delta^{-\frac12} \bv$, where now $\bv \sim \G(0,\bC_0)$, and the observation
model becomes
\begin{equation}{\by}=\delta^{-\frac12} \bK
  \bv+\bmeta\,.\end{equation}
The MwG sampler is used to sample $\bv,\delta | \by$ by iteratively sampling
from the two conditionals. Recall from the discussion on CA in section \ref{ch3:sec:int}, that $\delta|\by,\bu\equiv \delta|\bu$ and note that $\delta |
\by,\bv$, no longer simplifies to $\delta | \bv$, since even
conditionally on $\bv$, $\delta$ and $\by$ are dependent; this is the
non-centered  property in the hierarchical model, \cite{PRS07}. Additionally, note
that a practically useful way to sample from $\bv | \by,\delta$, which
recycles available code for CA, is to first sample $\bu | \by,\delta$,
as in CA, and then transform $\bu$ to $\bv$ via $\bv =
\delta^{\frac12} \bu$. Finally, for reasons of efficiency described below, we  prefer to
sample $\tau=\delta^{-\frac12}$ instead of $\delta$ directly. In order to obtain the
same Bayesian model as the one before the transformation, the prior
distribution for $\tau$ should be the one obtained from the prior on
$\delta$ after the $1/\sqrt{\delta}$ transformation, that is a square
root of an inverse-gamma distribution. Of course, we can
deterministically calculate $\delta=1/\tau^2$ after each such update,
to get $\delta$-samples and proceed to the next conditional simulation
in the algorithm. 

The finite-dimensional discretization of the algorithm is obtained in
the same way as CA. We notice that 
the log-likelihood is quadratic in $\tau$, for given $v$. We can
exploit this property to sample $\tau$ efficiently. The conditional posterior $\tau|{\yn},v$  is not
Gaussian, because the
prior on $\tau$ is not Gaussian, hence for our numerical results we replace
direct simulation from the conditional with a Metropolis-Hastings step
that targets the conditional. Given that the conditional posterior is
the product of the prior and the conditional likelihood, and we expect
the likelihood to be the dominant term of the two, we use the
likelihood, seen as a function of $\tau$, as a proposal density in the
Metropolis-Hastings step. The likelihood as a
function of $\tau$ is Gaussian $\G(r_{\lambda,v},q_{\lambda,v}^2)$, where 
\begin{equation}\frac{1}{q_{\lambda,v}^2}=\lambda\norm{\C_1^{-\frac12} Kv}_{\R^N}^2, \quad\quad\frac{r_{\lambda,v}}{q_{\lambda,v}^2}=\lambda\pr{ K^\ast\C_1^{-1}{\yn}}{v}_{\R^N},\end{equation}
hence easy to simulate
from. Proposals generated in this way are immediately rejected if
negative, and if not they are accepted according to the
Metropolis-Hastings ratio that by construction only involves the prior
density. Note that the same complication would arise had we chosen to work with $\delta$ instead of $\tau$, since $\delta|y,v$ is also not a known distribution. The difference in that case is that there is no apparent good proposal density for the Metropolis-Hastings step, since the likelihood is not a known distribution as a function of $\delta$.

We use the following Gibbs sampler, where $p(\cdot)$ denotes the density of the square root of the inverse-gamma distribution with parameters $\ad, \bd$:

\begin{framed}
\begin{algor}\label{ch3:algrep}{\ }\emph{\begin{enumerate}
\item[0)]Initialize $\tau^{(0)}$, calculate $\delta^{(0)}=1/(\tau^{(0)})^2$ and set $k=0;$
\item[1)] $u^{(k)}\sim\G\big(m_{\si{k},\delta^{(k)}}({\yn}),\C_{\si{k},\delta^{(k)}}\big)$;\\
$v^{(k)}=(\de{k})^{\frac12}u^{(k)}$;
\item[2)] propose $\tau\sim\G(r_{\si{k},v^{(k)}},q_{\si{k},v^{(k)}}^2)$;\\
 if $\tau\leq0$ reject;
 if $\tau>0$ accept with probability $\frac{p(\tau)}{p(\tau^{(k)})}\wedge1$ otherwise reject;\\
 if $\tau$ accepted set $\tau^{(k+1)}=\tau$, otherwise set $\tau^{(k+1)}=\tau^{(k)}$;\\
$\de{k+1}=1/(\tau^{(k+1)})^2$;
\item[3)]Set $k=k+1$. If $k<k_{max}$ return to step 1, otherwise stop.
\end{enumerate}}
\end{algor}
\end{framed}

\subsection{Marginal Algorithm (MA)}\label{MA}
Given that $\bu$ (hence $\bK\bu$) and $\bmeta$ are independent
Gaussian random variables, the marginal distribution of the data $\by$ given
$\delta$ is also Gaussian, 
\[\by|\delta\sim\G(0,\delta^{-1}\bK\bC_0\bK+\lambda^{-1}\bC_1)\,.
\]
 One can then use Bayes' theorem to get that \[\rp(\delta|\by)\propto
 \rp(\by|\delta)\rp(\delta).\] This distribution can be sampled using
 the random walk Metropolis (RWM) algorithm. In order to get samples from $\bu,\delta|\by$, we alternate between drawing $\delta|\by$ and updating $\bu|\by,\delta$. Furthermore, it is beneficial to the
 performance of the RWM, to sample $\log(\delta)|\by$ instead of
 $\delta|\by$; of course, samples from $\log(\delta)|\by$ can be deterministically transformed to samples from $\delta|\by$. The resultant algorithm is what we 
 call the  \emph{marginal algorithm} (MA). MA in the discrete
                                 level is as follows, where
                                 $p(\cdot)$ now denotes the density of
                                 the logarithm of a gamma distribution with
                                 parameters $\ad,\bd$ and $\rho=\log(\delta)$:
\begin{framed}
\begin{algor}\label{ch3:algmar}{\ }\emph{\begin{enumerate}
\item[0)]Initialize $\rho^{(0)}$ and set $k=0;$
\item[1)] $u^{(k)}\sim\G\big(m_{\si{k},\de{k}}({\yn}),\C_{\si{k},\de{k}}\big);$
\item[2)] propose $\rho\sim\G(\rho^{(k)},s^2)$;\\
accept with probability $\frac{\rp(y|\exp({\rho}))p_0(\rho)}{\rp\left(y|\exp(\rho^{(k)})\right)p_0(\rho^{(k)})}\wedge1$ otherwise reject;\\
 if $\rho$ accepted set $\rho^{(k+1)}=\rho$, otherwise set $\rho^{(k+1)}=\rho^{(k)}$;\\
 set $\de{k+1}=\exp(\rho^{(k+1)})$;
\item[3)]Set $k=k+1$. If $k<k_{max}$ return to step 1, otherwise stop.
\end{enumerate}}
\end{algor}
\end{framed}
We follow the rule-of-thumb proposed in \cite{GGR96} and choose the RWM proposal variance $s^2$ to achieve an acceptance probability around 44\%.

\subsection{Contrasting the methods}\label{contr}

As discussed in section \ref{ch3:sec:int}, and is formally shown
in section \ref{ch3:sec:main}, CA   will deteriorate as the
discretization level of the unknown, $N,$ becomes larger.  To get a
first understanding of this
phenomenon in the linear-conjugate setting, note that the $\Ga(\ad,\bd)$ distribution has mean and variance $\ad\bd^{-1}$ and 
$\ad\bd^{-2}$ respectively. Hence,  for any $\mu>0$, as $N$ grows, the random variable $
\Ga(\ad+\frac{N}2,\bd+\mu\frac{N}2)$ behaves like a Dirac distribution centred on $\mu^{-1}$.
Furthermore, we will show that, because of the consistency of
the approximation of the operators defining the Bayesian
inverse problem, together with scaling of the norms on $\R^N$
to reproduce the Hilbert space norm limit, it is natural to assume that
\[\norm{\C_0^{-\frac12}u^{(k)}}^2_{\R^N}\simeq (\de{k})^{-1}N.\]
Using the limiting behaviour of the gamma distribution described above, this means 
that as the dimension $N$ increases, we have $\de{k+1}\simeq \de{k}$ hence the $\delta$-chain 
makes very small moves and slows down.

In contrast, both conditionals $\bu|\by,\delta$ and $\delta|\by,\bv$ sampled in NCA are non-degenerate even
in the infinite-dimensional limit. 
Our numerical results show that this reparametrization is indeed robust
with respect to the increase in dimension (see section
\ref{ch3:sec:sim}), although establishing formally that a spectral gap exists
for NCA in this limit is beyond the scope of this paper.

Similarly, both distributions $\bu|\by,\delta$ and $\delta|\by$ sampled in MA are non-degenerate in the continuum limit, hence MA is robust
with respect to $N$. Moreover, MA is optimal with respect to the dependence between the two components of the algorithm, since the $\delta$-chain is independent of the $u$-draws; there is a loss of efficiency due to the use of RWM to sample $\delta|y$, but provided the proposal variance is optimally tuned, this will only have a minor effect in the performance of MA.
For these reasons, in section \ref{ch3:sec:sim} we use the optimally tuned MA as
the gold standard with which we compare the performance of CA and
NCA.  Nevertheless, we stress here that: \begin{enumerate}
\item[i)] MA requires at each iteration the
  potentially computationally expensive calculation of the square root
  and the determinant of the precision matrix of $y|\delta$. This makes
  the implementation of MA in large scale linear inverse problems less
  straightforward compared to CA and NCA.  \item[ii)]even though we view MA as a gold, albeit potentially expensive, standard in our linear setting, for nonlinear problems MA is not available. On the contrary, CA and NCA are straightforward to extend to the nonlinear case (see Section \ref{ch3:sec:con}); this is one of the principal motivations for studying the optimal parametrization of Gibbs sampling in this context.\end{enumerate}


\section{Theory}\label{ch3:sec:main}In this section we present our theory concerning the behaviour of CA as the discretization level increases, in the linear inverse problem setting introduced in subsection \ref{sec:linear}.
We first formulate our assumptions on the underlying infinite-dimensional model as well as a corresponding set of discrete-level assumptions, before presenting our main result on the large $N$ behaviour of Algorithm \ref{ch3:algstd}. 
\subsection{Assumptions} We work under the following assumptions on the underlying infinite-dimensional linear inverse problem:
\begin{assumptions}\label{ch3:infass1}{\ }
\begin{enumerate}
\item[i)] For any $\lambda,\delta>0$, we have $\bme_{\lambda,\delta}(\by)\in\D(\bC_0^{-\frac12})$ $\by$-almost surely; that is, the posterior mean belongs to the Cameron-Martin space of the prior on $\bu|
\delta;$
\item[ii)] $\bC_1^{-\frac12}\bK\bC_0 \bK^\ast\bC_1^{-\frac12}$ is trace-class; that is, the prior is sufficiently 
regularizing.
\end{enumerate}
\end{assumptions}

Assumption \ref{ch3:infass1}(ii) implies the second and 
third conditions of the Feldman-Hajek theorem \cite[Theorem
2.23]{DZ92}. Together with 
Assumption \ref{ch3:infass1}(i), they thus imply that  $\by$-almost surely  $\bu|\by, \delta$ is absolutely continuous with respect to
$\bu|\delta$ and hence the infinite-dimensional intuition on the behaviour of CA described in section \ref{ch3:sec:int} applies. 

In the following, we assume that $\C_0$ and $\C_1$ are positive definite $N\times N$ matrices which 
are the discretizations of the positive definite operators $\bC_0$ and $\bC_1$ respectively, and the $N\times 
N$ matrix $K$ is the discretization of the bounded operator $\bK$. Our analysis of the $\delta$-chain is valid under the following assumptions at the discrete level:
\begin{assumptions}\label{ch3:ass1}
{\ }
\begin{enumerate}
\item[i)] For almost all data $\by$, for any $\lambda,\delta>0$, there exists a constant $c_1=c_1(\by;\lambda,
\delta)\geq0$, independent of $N$, such that \begin{align*}\norm{\C_{0}^{-\frac12}m_{\lambda,\delta}
({\yn})}_{\R^N}\leq c_1;\end{align*}
\item[ii)] there exists a constant $c_2\geq0$, independent of $N$ and $\by$, such that \begin{align*}
\tr(\C_1^{-\frac12}K\C_{0}K^\ast\C_1^{-\frac12})\leq c_2.\end{align*}
\end{enumerate}
\end{assumptions}

These assumptions are typically 
inherited from Assumptions \ref{ch3:infass1} when consistent discretizations are used; see subsection \ref{ch3:sec:dm} and section \ref{ch3:sec:ex} for more details and examples. 

 \subsection{Main Result}
 We now present our main result on the behaviour of Algorithm \ref{ch3:algstd} in the asymptotic regime of large $N$. 
We start by noting that the two steps of updating $u|{\yn},\delta$ and $\delta|{\yn},u$ in Algorithm \ref{ch3:algstd}, can be compressed to give one 
step of updating $\delta$ and involving the noise in the $u$ update. Indeed, we denote by $\de{k+1}_N$  
the $\delta$-draw in the $k+1$ iteration of the Gibbs sampler where the problem is discretized in $\R^N$.
This draw is made using the previous draw of $u|{\yn},\delta$, which assuming that $\de{k}_N=\delta$, is 
denoted by $u_{\delta}^{(k)}$ and can be written as \begin{equation}\label{ch3:eq:uu}
u_{\delta}^{(k)}=m_{\lambda,\delta}({\yn})+\C_{\lambda,\delta}^{\frac12}\zeta,
\end{equation}where $\zeta$ is an $N$-dimensional Gaussian white noise representing the fluctuation in 
step 1, and $ \C_{\lambda,\delta}, m_{\lambda,\delta}$ are given by the formulae (\ref{ch3:eq:prec}), 
(\ref{ch3:eq:mean}) respectively.
 Hence we have
\begin{align}\label{ch3:eq:dd}\delta_N^{(k+1)}\sim\Ga(\ad+\frac{N}2,\bd+\frac12\norm{\C_0^{-\frac12}
u_{\delta}^{(k)}}_{\R^N}^2).\end{align}

Assumptions \ref{ch3:ass1} ensure that the squared norm appearing in (\ref{ch3:eq:dd}) behaves like $
\delta^{-1}N$, as assumed in the discrete level intuition discussed in subsection \ref{contr}. This is made precise in the following lemma which forms the backbone of our analysis and is proved in subsection \ref{ch3:ssec:ap1}. \begin{lemma}\label{ch3:lem1}
Under Assumptions \ref{ch3:ass1}, for any $\lambda,\delta>0$ we have,
\begin{align}\label{ch3:eq:denom}
\bd+\frac12\norm{\C_0^{-\frac12}u^{(k)}_{\delta}}^2_{\R^N}=\delta^{-1}\frac{N}{2}+\delta^{-1}\sqrt{\frac{N}{2}}
W_{1,N}+F_N(\delta),
\end{align}
where i) $W_{1,N}$ only depends on the white noise $\zeta$ in (\ref{ch3:eq:uu}), has mean zero and 
variance one, higher order moments which are bounded uniformly in $N$, and converges weakly to a 
standard normal random variable as $N\to\infty;$ ii) $F_N(\delta)$ depends on the data ${\yn}$ and $\by$-almost surely has finite moments of all positive orders uniformly in $N$ (where the expectation is taken 
with respect to  $\zeta$).
\end{lemma}

Combining with the scaling property of the gamma distribution as in the intuition described in subsection \ref{contr}, we show that as the dimension increases 
the $\delta$-chain makes increasingly smaller steps, and quantify the scaling of this slowing down. 
Indeed, we prove that 
for large $N$ the $
\delta$-chain makes moves which on average are of order $N^{-1}$ with fluctuations of order $N^{-
\frac12}$. As a result, it takes $\mathcal{O}(N)$ steps for the $\delta$-chain to move by $\mathcal{O}(1)$. 

\begin{theorem}\label{ch3:thm1}
Let $\lambda>0$ and consider Algorithm \ref{ch3:algstd} under Assumptions \ref{ch3:ass1}. In the limit $N\to\infty$,
we have almost surely with respect to $\by$  and where all the expectations are taken with respect to the randomness in the algorithm:\begin{enumerate}\item[i)]the expected step in the $\delta$-chain 
scales like $\frac{2}N$, that is, for any $\delta>0,$ 
\begin{align*}\frac{N}2\E\left[\de{k+1}_N-\de{k}_N|\de{k}_N=\delta\right]=(\ad+1)\delta-f_N(\delta;{\yn})
\delta^2+\mathcal{O}(N^{-\frac12}),\end{align*}
where $f_N(\delta;{\yn})$ is bounded uniformly in $N$.
In particular, if there exists $f(\delta;\by)\in\R$ such that $f_N(\delta;{\yn})\to f(\delta;\by)$, then \begin{align*}
\frac{N}2\E\left[\de{k+1}_N-\de{k}_N|\de{k}_N=\delta\right]=(\ad+1)\delta-f(\delta;\by)\delta^2+\smallO{1};
\end{align*}

\item[ii)]the variance of the step also scales like $\frac{2}N$ and in particular for any $\delta>0,$
\begin{align*}\frac{N}2\Var\left[\de{k+1}_N-\de{k}_N|\de{k}_N=\delta\right]=2\delta^2+\mathcal{O}(N^{-\frac12}).\end{align*}

\end{enumerate}

\end{theorem}

\begin{remark}\label{ch3:rem1}{\ }
\begin{enumerate}
\item[i)]The proof of Theorem \ref{ch3:thm1} can be found in subsection \ref{ch3:ssec:dproof} in the Appendix. Equation (\ref{ch3:eq:sp1}) is a key identity, as it very clearly separates the three sources of fluctuation in the draw $\de{k+1}_N$, that is, the fluctuation in the Gaussian-draw $u|y,\delta$, the fluctuation in the gamma-draw $\delta|y,u$ and the fluctuation in the data.
\item[ii)]$f_N(\delta;{\yn}):=\E^\zeta[F_N(\delta;{\yn})]$, where $F_N$ is defined in the proof of Lemma 
\ref{ch3:lem1}. The assumption on the convergence of $f_N(\delta;{\yn})$ is trivially satisfied under 
Assumptions \ref{ch3:ass1}, if the discretization scheme used is such that if the vector $x\in\R^N$ and the 
$N\times N$ matrix $T$ are the discretizations at level $N$ of $\bx\in\X$ and the linear operator $\bT$ 
respectively, then $\norm{T x}_{\R^N}$ is a non-decreasing sequence. This is the case for example in 
spectral truncation methods, when $\bT$ is diagonalizable in the orthonormal basis used (see 
subsection \ref{sp}).
\end{enumerate}
\end{remark}

Theorem  \ref{ch3:thm1} suggests that 
\begin{align}\label{ch3:eq:em}
\de{k+1}_N-\de{k}_N\approx \frac2N\Big((\ad+1)\de{k}_N-f_N(\de{k}_N;y)(\de{k}_N)^2\Big)+\frac{2\de{k}_N}{\sqrt{N}}\Xi,
\end{align}
where $\Xi$ is a real random variable with mean zero and variance one. In the case where $f_N$ has a limit, the last expression looks like the Euler-Maruyama discretization of the stochastic differential equation
\begin{align}
\label{ch3:eq:dl}d\delta=\big(\ad+1-f(\delta;\by)\delta\big)\delta dt+\sqrt{2}\delta dW,\end{align} where 
$W=W(t)$ is a standard Brownian motion, with time step $\frac2N$. This is another manifestation of the fact that it takes $\mathcal{O}(N)$ steps for the $\delta$-chain to make a move of $\mathcal{O}(1)$ size.

Note that \eqref{ch3:eq:em} implies that the 
\emph{expected square jumping distance} of the Markov chain for
$\delta$ generated by CA is $\mathcal{O}(1/N)$. Recall (see
for example \cite{SR09} for a recent account) that this distance is
defined as $\E[(\de{k+1}_N-\de{k}_N)^2]$, where $\de{k}_N$ is drawn
from the stationary distribution. Hence,  it is the expected squared
step of the chain in stationarity. It is easy to check that it equals
$2 \Var (\de{k}_N) (1-Corr(\de{k}_N, \de{k+1}_N))$, where again all
quantities are computed in stationarity. Although the
expected square jumping distance is a sensible and practically useful measure of
efficiency of a Markov chain, there is no explicit result that links
it to the variance of Monte Carlo averages formed by using the output
of the chain. This variance will not only depend on autocorrelation at
other lags, but also on the function being averaged. Still, it gives a
rough idea: if the autocorrelation function associated to the identity
function is geometrically decaying, with lag-1 autocorrelation $\rho_N$, then the variance of the sample
average of $k_{max}$, $\de{k}_N$ values in stationarity will be $\Var
(\de{k}_N) (1+\rho_N)/\big((1-\rho_N)k_{max}\big)$. The point here is that $\rho_N$
behaves like $1-c/N$, for some $c$, but $\Var
(\de{k}_N)$ is $\mathcal{O}(1)$. Hence, the Monte Carlo error
associated with $k_{max}$ draws in stationarity is
$\mathcal{O}(\sqrt{N / k_{max}})$.

 
 \section{Examples satisfying our assumptions}\label{ch3:sec:ex}
We now present three families of linear inverse problems satisfying Assumptions \ref{ch3:infass1} on the underlying conitnuum model: a family of mildly ill-posed inverse problems, where the operators defining the problem are simultaneously diagonalizable, \cite{KVZ12}; a family of severely ill-posed inverse problems again in a diagonal setting, \cite{KVZ13, ASZ12}; and  a family of mildly ill-posed inverse problems in a nondiagonal setting, \cite{ALS13}. We expect that  Assumptions \ref{ch3:ass1}, will be satisfied by consistent discretizations of these models. 
Indeed, we show that our discrete level assumptions are satisfied if we discretize the two diagonal examples using spectral truncation (see subsection \ref{sp}). Furthermore, in section \ref{ch3:sec:sim} we provide numerical evidence that our ideas also apply in nondiagonal settings and when using other discretization schemes, in particular discretization via finite difference approximations (see subsection \ref{fd}). We do not prove that discretization via finite differences satisfies our discrete level assumptions as it is beyond the scope of this paper; we expect however this to be the case.

\subsection{Linear mildly ill-posed simultaneously diagonalizable inverse problem}\label{ch3:ssec:diag}

We consider the linear inverse problem setting of subsection \ref{sec:linear}, where $\bK,\bC_0$ and $\bC_1$ commute with each other and $\bK^\ast \bK, \bC_0$ and  $\bC_1$ are simultaneously diagonalizable with common complete orthonormal eigenbasis $\{\be_j\}_{j\in\N}$. Note that we do not assume that $\bK$ and $\bC_1$ are  compact, but we do assume that $\bK^\star \bK$ and $\bC_1$ are both diagonalizable in $\{\be_j\}_{j\in\N}$; in particular, we allow for $\bK$ and $\bC_1$ to be the identity. For any $\bw\in\X$, let $w_j:=\pr{\bw}{\be_j}$. 
Let $\bSigma$ be a positive definite and trace class operator in $\X$ which is diagonalizable in the orthonormal basis $\{\be_j\}_{j\in\N}$, with eigenvalues $\{\mu^\Sigma_j\}_{j\in\N}$. Then for any $\brho\in\X$, we can write a draw $\bx\sim\G(\rho,\bSigma)$ as  \begin{equation*}
\bx=\brho+\sumj\sqrt{\mu^\Sigma_j}\gamma_j\be_j,\end{equation*}
where $\gamma_j$ are independent standard normal random variables in $\R$; this is the Karhunen-Loeve expansion \cite[Chapter III.3]{RA90}. In fact, the Karhunen-Loeve expansion makes sense even if $\mu^\Sigma_j$ are not summable, that is if $\bSigma$ is not trace class in $\X$; the expansion then defines a Gaussian measure in a bigger space than $\X$ in which $\bSigma$ is trace class. This expansion suggests that since we are in a simultaneously diagonalizable setting we can use the Parseval identity and work entirely in the frequency domain as in subsection \ref{sp}. Indeed, we identify an element $\bw\in\X$ with the sequence of coefficients $\{w_j\}_{j\in\N}$, and the norm and inner product in $\X$ with the $\ell^2$-norm and inner product. Furthermore, we identify the operators $\bC_0, \bC_1$ and $\bK$ with the sequences of their eigenvalues $\{\mu^{\C_0}_j\}_{j\in\N}, \{\mu^{\C_1}_j\}_{j\in\N}$ and $\{\mu^{K}_j\}_{j\in\N}$ respectively. Algebraic operations on the operators $\bC_0, \bC_1, \bK$ are defined through the corresponding operations on the respective sequences. 

We make the following assumptions on the spectral decay of $\bK, \bC_0$ and $\bC_1$:
\begin{assumptions}\label{ch3:decass}
The eigenvalues  of $\bK^\ast \bK, \bC_0$ and $\bC_1$, denoted by $(\mu^{K}_j)^2, \mu^{\C_0}_j, \mu^{\C_1}_j$, respectively, satisfy \addtocounter{footnote}{0}\footnote{$\alpha,\beta$ not to be confused with $\upalpha, \upbeta$ used respectively as shape and rate parameters of the gamma distribution.}\addtocounter{footnote}{-1}\begin{enumerate}
\item[-]$(\mu^{K}_j)^2\asymp j^{-4\ell}, \;\ell\geq0;$
\item[-]$\mu^{\C_0}_j\asymp j^{-2\alp}, \;\alp>\frac12$$;$
\item[-]$\mu^{\C_1}_j\asymp j^{-2\bet}, \;\bet\geq0$.
\end{enumerate}
\end{assumptions}

Let $\nu$ be the joint distribution of $\by$ and $\bu$, where $\bu|\delta\sim\G(0,\delta^{-1}\bC_0)$ and $\by|\bu,\delta\sim\G(\bK \bu,\lambda^{-1}\bC_1)$.
Then in this diagonal case, it is straightforward to show in the infinite-dimensional setting that the conditional posterior $\bu|\by,\delta$ is $\nu$-almost surely Gaussian, $\G(\bme_{\lambda,\delta}(\by),\bC_{\lambda,\delta})$, where $\bC_{\lambda,\delta}$ and $\bme
_{\lambda,\delta}(\by)$  satisfy (\ref{ch3:eq:prec}) and (\ref{ch3:eq:mean}) respectively. We make the following additional assumption:

\begin{assumption}\label{ch3:dc:ass}
The parameters $\alp,\bet, \ell$ in Assumptions \ref{ch3:decass} satisfy $2\alp+4\ell-2\bet>1$.
\end{assumption}

We show that under Assumptions \ref{ch3:decass} and \ref{ch3:dc:ass}, Assumptions \ref{ch3:infass1} on the underlying infinite-dimensional model are satisfied $\nu$-almost surely. Without loss of generality assume $\delta=\lambda=1$. For Assumption \ref{ch3:infass1}(i), we have using the Karhunen-Loeve expansion and Assumption \ref{ch3:decass},\begin{align*}
\E^\nu\norm{\bC_0^{-\frac12}\bme(\by)}^2\leq c\E^\nu\sumj\frac{j^{2\alp-4\ell+4b}}{(j^{-4\ell+2\bet}+j^{2\alp})^2}(j^{-2\ell-\alp}\zeta_j+j^{-\bet}\xi_j)^2,
\end{align*}
where $\{\zeta_j\}_{j\in\N}, \{\xi_j\}_{j\in\N}$ are two independent sequences of independent standard normal random variables. The assumption $2\alp+4\ell-2\bet>1$ secures that the right hand side is finite, hence $\bme(\by)\in\D(\bC_0^{-\frac12}) 
\;\nu$-almost surely. For Assumption \ref{ch3:infass1}(ii),  the operator $\bC_1^{-\frac12}\bK\bC_0 \bK^\ast\bC_1^{-\frac12}$ has eigenvalues that decay like $j^{-2\alp-4\ell+2\bet}$ and hence are summable by Assumption \ref{ch3:dc:ass}.

We define the Sobolev-like spaces $\mathcal{H}^t, t\in\R$: for $t\geq0$, define \begin{equation*}\mathcal{H}^t:=\{\bu\in\X: \norm{\bu}_{\mathcal{H}^t}:=\sumj j^{2t}\pr{u_j}{\be_j}^2<\infty\},\end{equation*} and for $t<0$, $\mathcal{H}^{t}:=(\mathcal{H}^{-t})^\ast$. We assume to have data of the following form:
\begin{assumption}\label{ch3:ass22} $\by=\bK{\bu^\dagger}+\hl^{-\frac12}\bC_1^{\frac12}\bxi$, where ${\bu^\dagger}\in \mathcal{H}^{\bet-2\ell}$ is the underlying true solution and $\bxi$ is a Gaussian white noise, $\xi\sim\G(0,I)$. \end{assumption}

Note that under Assumptions \ref{ch3:decass}, \ref{ch3:dc:ass} and \ref{ch3:ass22}, it is straightforward to check that Assumption \ref{ch3:infass1}(i) is also satisfied $\bxi$-almost surely. Indeed, using the Karhunen-Loeve expansion we have\begin{align*}\E\norm{\bC_0^{-\frac12}\bme(\by)}^2\leq c\E\sumj\frac{j^{2\alp-4\ell+4b}}{(j^{-4\ell+2\bet}+j^{2\alp})^2}(j^{-2\ell}(u^\dagger_j)^2+\hl^{-\frac12}j^{-\bet}\xi_j)^2,\end{align*} where $\{\xi_j\}_{j\in\N}$ is a sequence of independent standard normal random variables. The assumption $2\alp+4\ell-2\bet>1$ together with ${\bu^\dagger}\in \mathcal{H}^{\bet-2\ell}$ secure that the right hand side is finite. Assumption \ref{ch3:infass1}(ii) is independent of $\by$, hence also holds by our previous considerations.

A natural way to discretize random draws in this setup is by truncating the Karhunen-Loeve expansion which is equivalent to the spectral truncation in subsection \ref{sp}. We assume to have discrete data of the form \begin{equation*}{\yn}=K{u^\dagger}+\hl^{-\frac12}\C_1^\frac12\xi,\end{equation*}
where $K, \C_1, {u^\dagger}$ and  $\xi$ are discretized as in subsection \ref{sp}. The prior is also discretized using spectral truncation, $u\sim\G(0,\C_0)$. 
We show that Assumptions \ref{ch3:ass1} are satisfied under Assumptions \ref{ch3:decass} and \ref{ch3:dc:ass}, for this data and discretization scheme.  

By Assumption \ref{ch3:decass}, there exists a constant $c\geq0$ independent of $N$, such that
\begin{align*}\E\norm{\C_0^{-\frac12}m({\yn})}^2_{\R^N}\leq c\E\sumn\frac{j^{2\alp-4\ell+4b}}{(j^{-4\ell+2\bet}+j^{2\alp})^2}(j^{-2\ell}{u^\dagger_j}+j^{-\bet}\xi_j)^2,
\end{align*}
where the right hand side is bounded uniformly in $N$, since we are summing nonnegative numbers and we have seen that under Assumptions \ref{ch3:dc:ass} and \ref{ch3:ass22} the corresponding infinite series is summable. 
Furthermore, again by Assumption \ref{ch3:decass}, there exists another constant $c\geq0$ independent of $N$, such that  \begin{align*}\tr(\C_1^{-\frac12}K\C_0K^\ast\C_1^{-\frac12})\leq c\sumn j^{-2\alp-4\ell+2\bet},\end{align*} where the right hand side is bounded uniformly in $N$, since we have seen that under Assumption \ref{ch3:dc:ass} the corresponding infinite series is summable. 

\subsection{Linear severely ill-posed simultaneously diagonalizable inverse problem}\label{ch3:ssec:sev}
We consider the setting of \cite{KVZ13, ASZ12}, that is, a similar situation with the previous example, where instead of having $(\mu_j^{K})^2\asymp j^{-4\ell}$  we now have $(\mu_j^{K})^2\asymp e^{-2sj^b},$ for $b,s>0$. The proof of the validity of Assumptions \ref{ch3:infass1} $\nu$-almost surely is identical to the proof in the previous example, where we now have the added advantage of the exponential decay of the eigenvalues of $\bK^\ast \bK$. 
We can also prove that for data of the form $\by=K{\bu^\dagger}+\lambda^{-\frac12}\bC_1^{\frac12}\bxi$, where now it suffices to have ${\bu^\dagger}\in\X$, Assumption \ref{ch3:infass1} is satisfied $\bxi$-almost surely. Finally, in a similar way to the previous example, Assumptions \ref{ch3:ass1} are valid if we discretize this setup by spectral truncation (subsection \ref{sp}). 

\subsection{Nondiagonal linear inverse problem}
We consider the setting of \cite{ALS13}, that is the linear inverse problem setting of subsection \ref{sec:linear}, where $\bK^\ast \bK, \bC_0$ and $\bC_1$ are not necessarily simultaneously diagonalizable but they are related to each other via a range of norm equivalence assumptions expressing that $\bK\simeq \bC_0^\ell$ and $\bC_1\simeq \bC_0^\beta$ for some $\ell, \beta\geq0$ (see \cite[Assumption 3.1]{ALS13}). Here $\simeq$ is used loosely to indicate two operators which induce equivalent norms. As before let $\nu$ be the joint distribution of $\by$ and $\bu$, where $\bu|\delta\sim\G(0,\delta^{-1}\bC_0)$ and $\by|\bu,\delta\sim\G(\bK \bu,\lambda^{-1}\bC_1)$. Then as in the simultaneously diagonalizable case examined above, we have that the conditional posterior $\bu|\by,\delta$ is $\nu$-almost surely  $\G(\bme_{\lambda,\delta}(\by),\bC_{\lambda,\delta})$, where $\bC_{\lambda,\delta}$ and $\bme
_{\lambda,\delta}(\by)$ satisfy (\ref{ch3:eq:prec}) and (\ref{ch3:eq:mean}) respectively (see \cite[Theorem 2.1]{ALS13}). It is implicit in \cite[Theorem 2.1]{ALS13} that $\bme_{\lambda,\delta}(\by)\in\D(\bC_0^{-\frac12})$ $\nu$-almost surely, hence Assumption \ref{ch3:infass1}(i) holds $\nu$-almost surely. Assumption \ref{ch3:infass1}(ii) also holds $\nu$-almost surely since if $\{\phi_j\}_{j\in\N}$ is a complete orthonormal system of eigenfunctions of $\bC_0$ and $\{\mu^{\C_0}_j\}_{j\in\N}$ the corresponding eigenvalues, by \cite[Assumption 3.1(3)]{ALS13} we have $\norm{\bC_1^{-\frac12}\bK\bC_0^\frac12\phi_j}^2\leq c\norm{\bC_0^{-\frac{\bet}2+\ell+\frac12}\phi_j}^2=c(\mu^{\C_0}_j)^{-\bet+2\ell+1}$ which is summable by \cite[Assumption 3.1(1) and 3.1(2)]{ALS13}. Hence, $\bC_1^{-\frac12}\bK\bC_0^\frac12$ is Hilbert-Schmidt thus $\bC_1^{-\frac12}\bK\bC_0\bK^\ast\bC_1^{-\frac12}$ is trace-class. We believe that Assumptions \ref{ch3:ass1} on the discrete level are also satisfied in this example if consistent discretization methods are used, however proving this is beyond the scope of the present paper.

\section{Numerical Results}\label{ch3:sec:sim}
We now present numerical simulations supporting our result in section \ref{ch3:sec:main} on the large $N$ behaviour of CA described in subsection \ref{CA} and our intuition contained in subsection \ref{contr} on the benefits of the reparametrization described in subsection \ref{NCA}. We consider an instance and a modification of the mildly ill-posed diagonal setting presented in subsection \ref{ch3:ssec:diag}. In subsection \ref{ch3:nex1} we use spectral truncation (see subsections \ref{sp}, \ref{ch3:ssec:diag}) and in subsection \ref{ch3:nex2} we use finite differences approximation (see subsection \ref{fd}).


\subsection{Signal in white noise model using truncated Karhunen-Loeve expansion}\label{ch3:nex1}
We consider the simultaneously diagonalizable setup described in subsection \ref{ch3:ssec:diag}, where $\X=L^2(\I), \I=(0,1)$. We consider the orthonormal basis $\be_j(x)=\sqrt{2}\sin(j\pi x), \;x\in \I$, and define the operators $\bK, \bC_0$ and $\bC_1$  directly through their eigenvalues $\mu^K_j=1, \mu^{\C_0}_j=j^{-3}$ and $\mu^{\C_1}_j=1,$ for all $j\in\N$, respectively. In particular, this is the \emph{normal mean model}, in which one assumes observations of the form \[y_i=u_i+\eta_j, \quad j\in\N,\] where $\eta_j\sim\G(0,\lambda^{-1})$ and the unknown is $\{u_j\}_{j\in\N}\in\ell^2$. This model is clearly equivalent to the \emph{white noise model}, \begin{align}\label{wn}\by=\bu+\bmeta,\end{align} where $\bmeta=\lambda^{-\frac12}\bm{\xi}$ and $\bm{\xi}$ is an unobserved Gaussian white noise, see subsection \ref{sp}. Note that $\bm{\xi}$ whose covariance function is a Dirac delta function, is not realizable in the basic Hilbert space $\X$ (instead $\X$ is the corresponding Cameron-Martin space), but can be interpreted in process form as for example in \cite{BHMR07, LC08} in the context of inverse problems. 
Although it can be argued that white noise data models are unrealistic at the
very smallest scales, they are a useful idealization of noise which is best thought
of as a continuous process with very short correlation lengthscales; in particular if
the correlation lengthscale is much smaller than the grid scale used, then it
is reasonable to use a white noise model.
The white noise model (\ref{wn}) is an important statistical model which is known to be asymptotically equivalent to several standard statistical models, for example nonparametric regression, \cite{BL96, LZ00}. It is also practically relevant, since it is a nontrivial special case of the deconvolution inverse problem, \cite{JKPR04, KR13}. Finally, it gives rise to Gaussian posterior distributions which are well studied in the sense of posterior consistency, see \cite{KVZ12, ALS13, KR13}.

 Defining $\bA_0$ to be the negative Laplace operator in $\I$ with Dirichlet boundary conditions, we recognize that we use a Gaussian prior with covariance operator $\bC_0$ proportional to $\bA_0^{-\frac{3}2}$.
Assumptions \ref{ch3:decass} are satisfied with $\alp=1.5$ and $\bet=\ell=0$; since $2\alp+4\ell-2\bet=3>1$, Assumption \ref{ch3:dc:ass} is also satisfied. We assume that we have data produced from the underlying true signal ${\bu^\dagger}(x)=\sumj {u^\dagger_j} \sqrt{2}\sin(j\pi x),$ for $x\in\I$, where ${u^\dagger_j}=j^{-2.25}\sin(10j)$ and $\hl=200$, and in particular we have that the coefficients of $\by$, are given as \begin{equation*}y_j={u^\dagger_j}+\hl^{-\frac12}\xi_j,\end{equation*} where $\xi_j$ are standard normal random variables. It is straightforward to check that ${\bu^\dagger}\in \mathcal{H}^{t}$ for any $t<1.75$, hence Assumption \ref{ch3:ass22} is also satisfied. According to the considerations in subsection \ref{ch3:ssec:diag}, we thus have that Assumptions \ref{ch3:ass1} hold when using the spectral truncation discretization method. This example is studied in \cite{SVZ13} where the interest is in studying the asymptotic performance of the posterior in the small noise limit (see section \ref{ch3:sec:con}). 

We use the hierarchical setup presented in subsection \ref{sec:linear} and implement Algorithms \ref{ch3:algstd} (CA),  \ref{ch3:algrep} (NCA) and \ref{ch3:algmar} (MA) contained in section \ref{ch3:sec:met} at discretization levels $N=32, 512, 8192$, with hyper-parameters $\ad=1,\bd=10^{-4},$ chosen to give uninformative hyper-priors, that is, hyper-priors whose variance is much larger than their mean. Following the discussion in subsection \ref{contr}, we view MA as the gold standard and benchmark CA and NCA against it. We use $10^4$ iterations and choose $\de{0}=1$ in all cases.  In order to have fair comparisons, we use a fixed burn-in time of $10^3$ iterations. We take the viewpoint that we have a fixed computational budget, hence we choose not to increase the burn-in time as $N$ increases as one can do if infinite resources are available. 

In Figure \ref{ch3:fig1} we plot the true solution, the noisy data and the sample means and credibility bounds using CA and NCA for $N=8192$. The sample means and credibility bounds at other discretization levels of the unknown are similar and are therefore omitted.

\begin{figure}[htp]
            \includegraphics[type=pdf, ext=.pdf, read=.pdf, width=0.32\columnwidth]{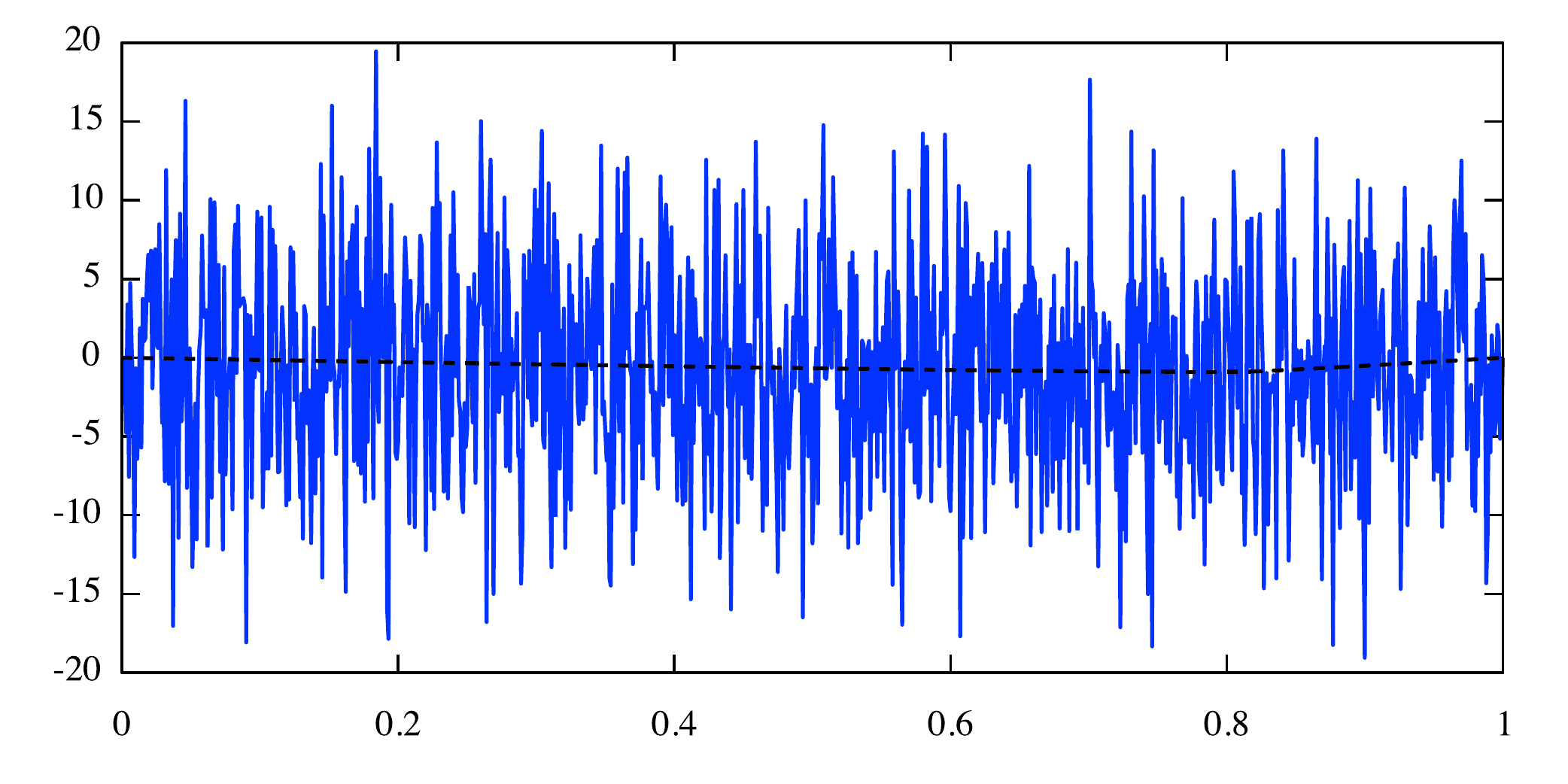} 
            \includegraphics[type=pdf, ext=.pdf, read=.pdf, width=0.32\columnwidth]{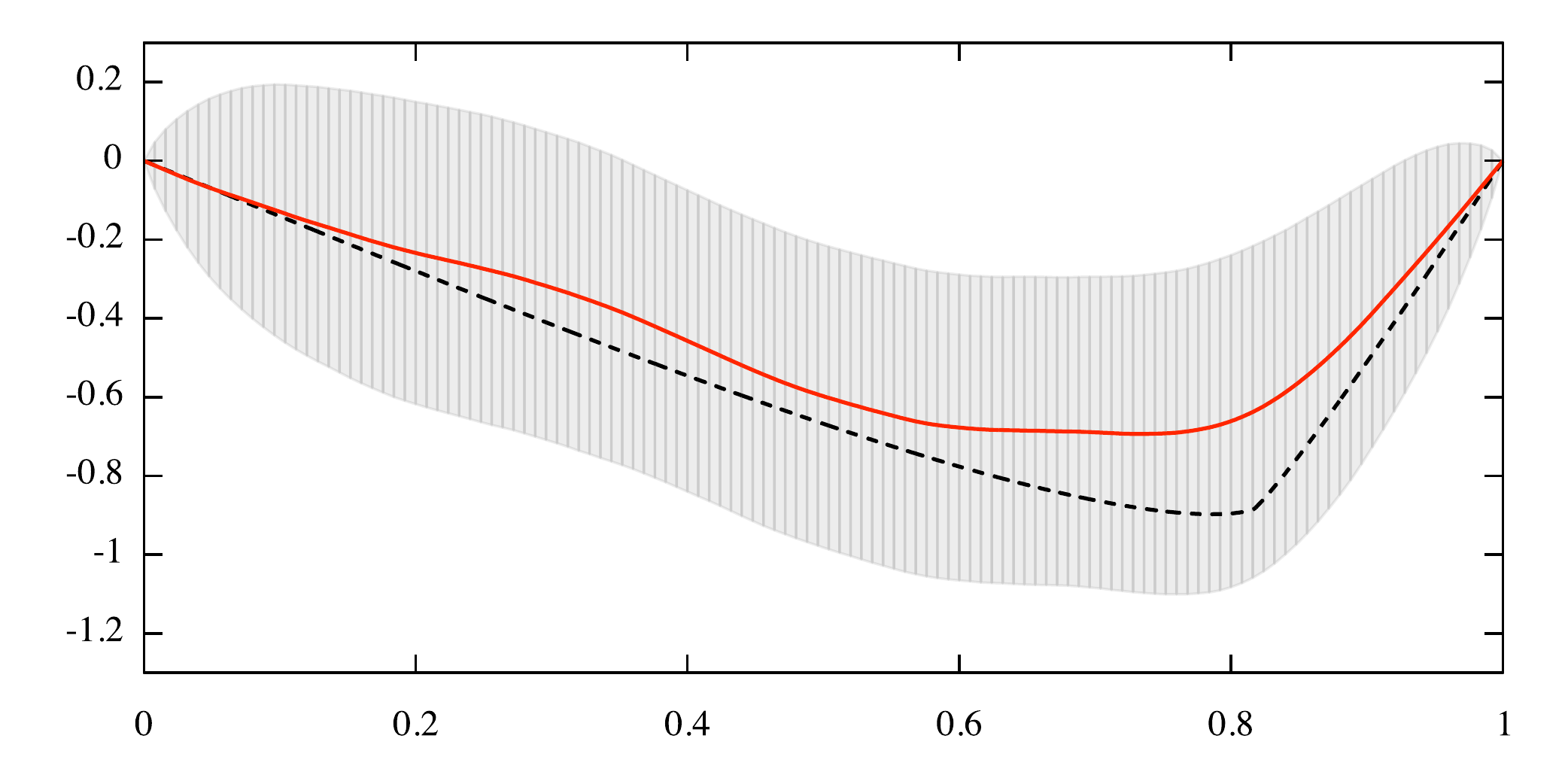}
            \includegraphics[type=pdf, ext=.pdf, read=.pdf, width=0.32\columnwidth]{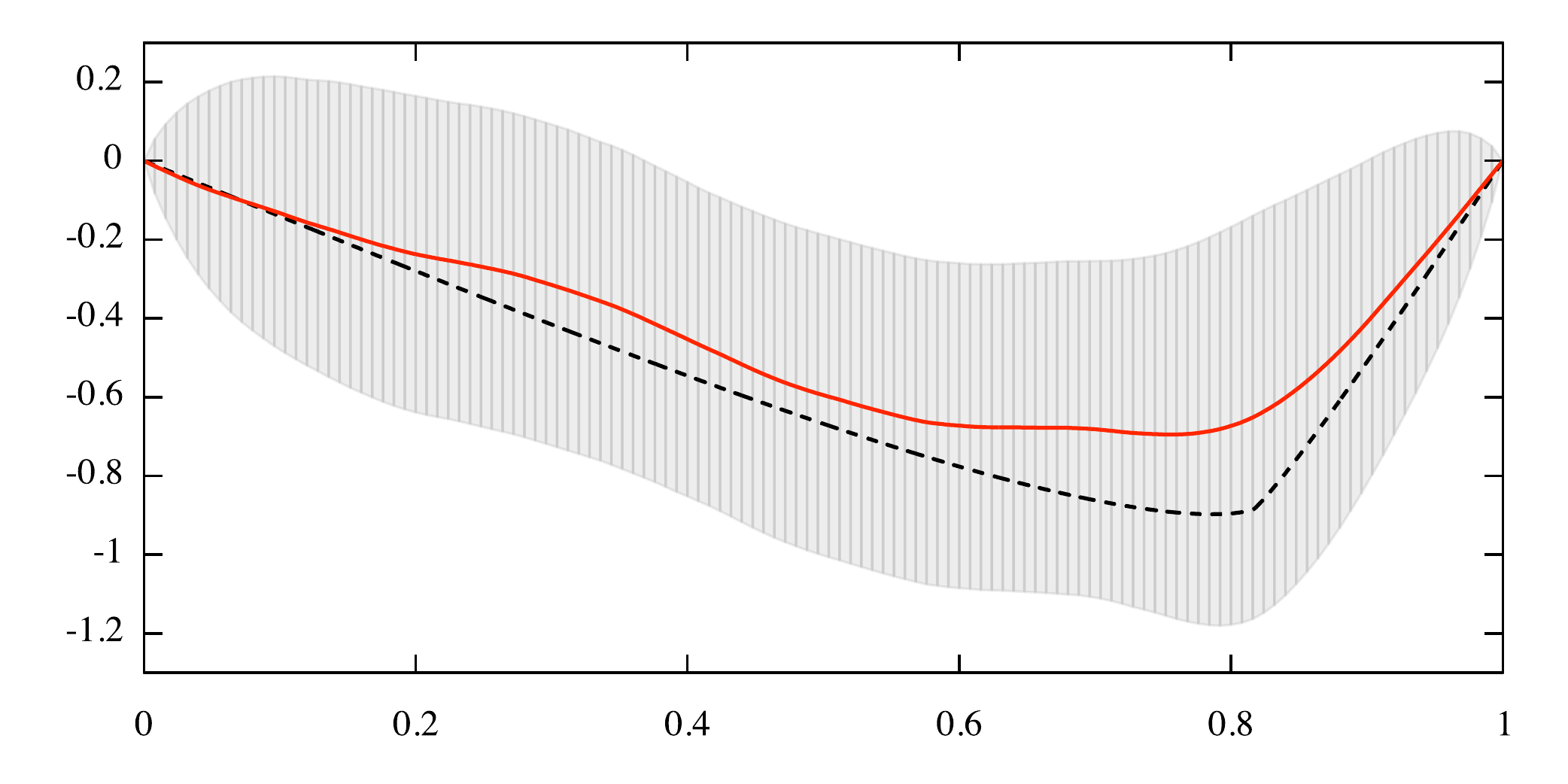}
            \caption{Left: true solution (dashed black) and noisy data (blue continuous). Middle and right: true solution (dashed black), sample mean (red continuous) and 87.5$\%$ credibility bounds (shaded area) for CA (middle) and NCA (right). Dimension is $N=8192$.}   
            \label{ch3:fig1}         
\end{figure}

 \begin{figure}[htp]
           \center{   
            \includegraphics[type=pdf, ext=.pdf, read=.pdf, width=0.32\columnwidth]{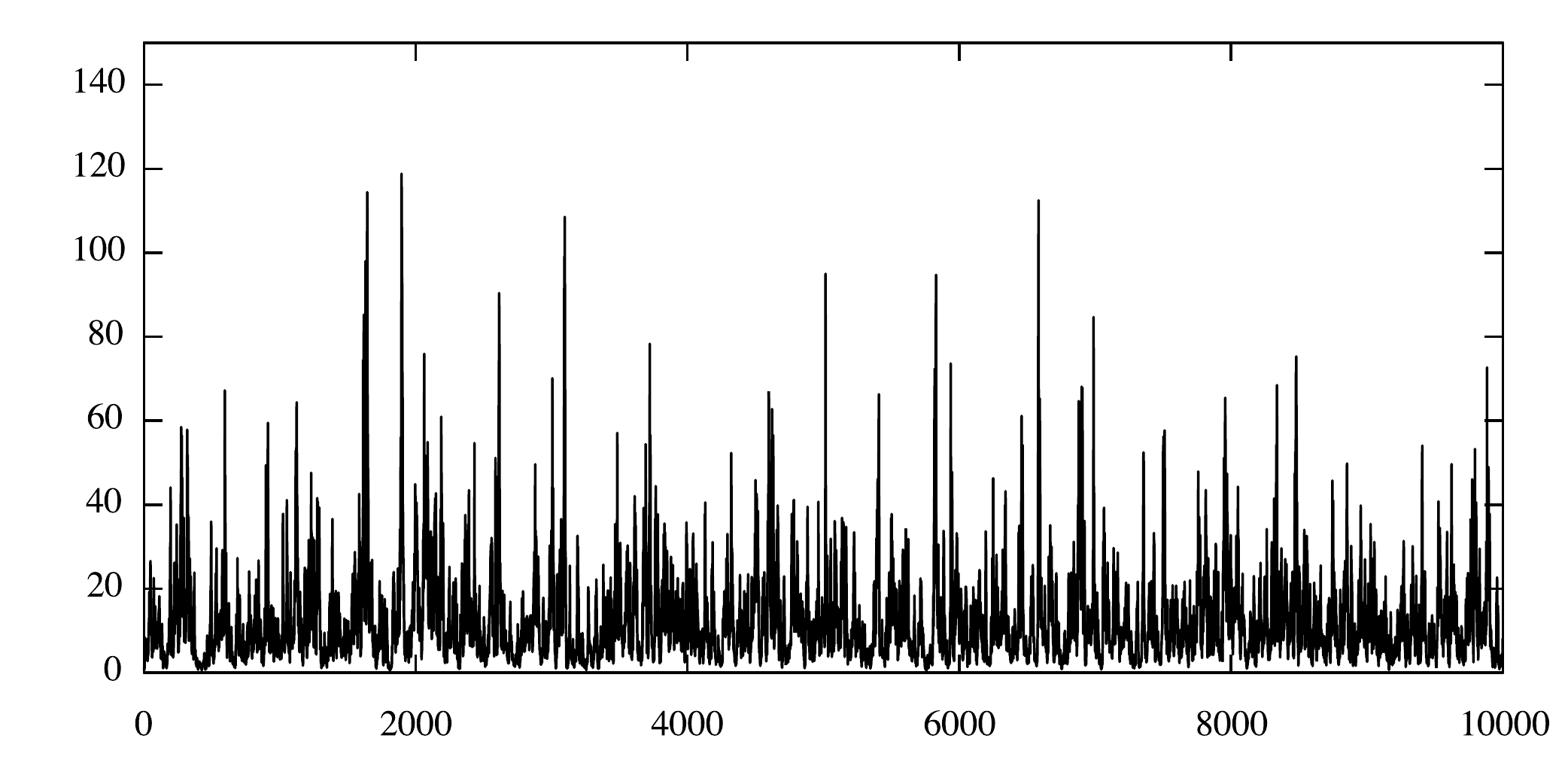}   
             \includegraphics[type=pdf, ext=.pdf, read=.pdf, width=0.32\columnwidth]{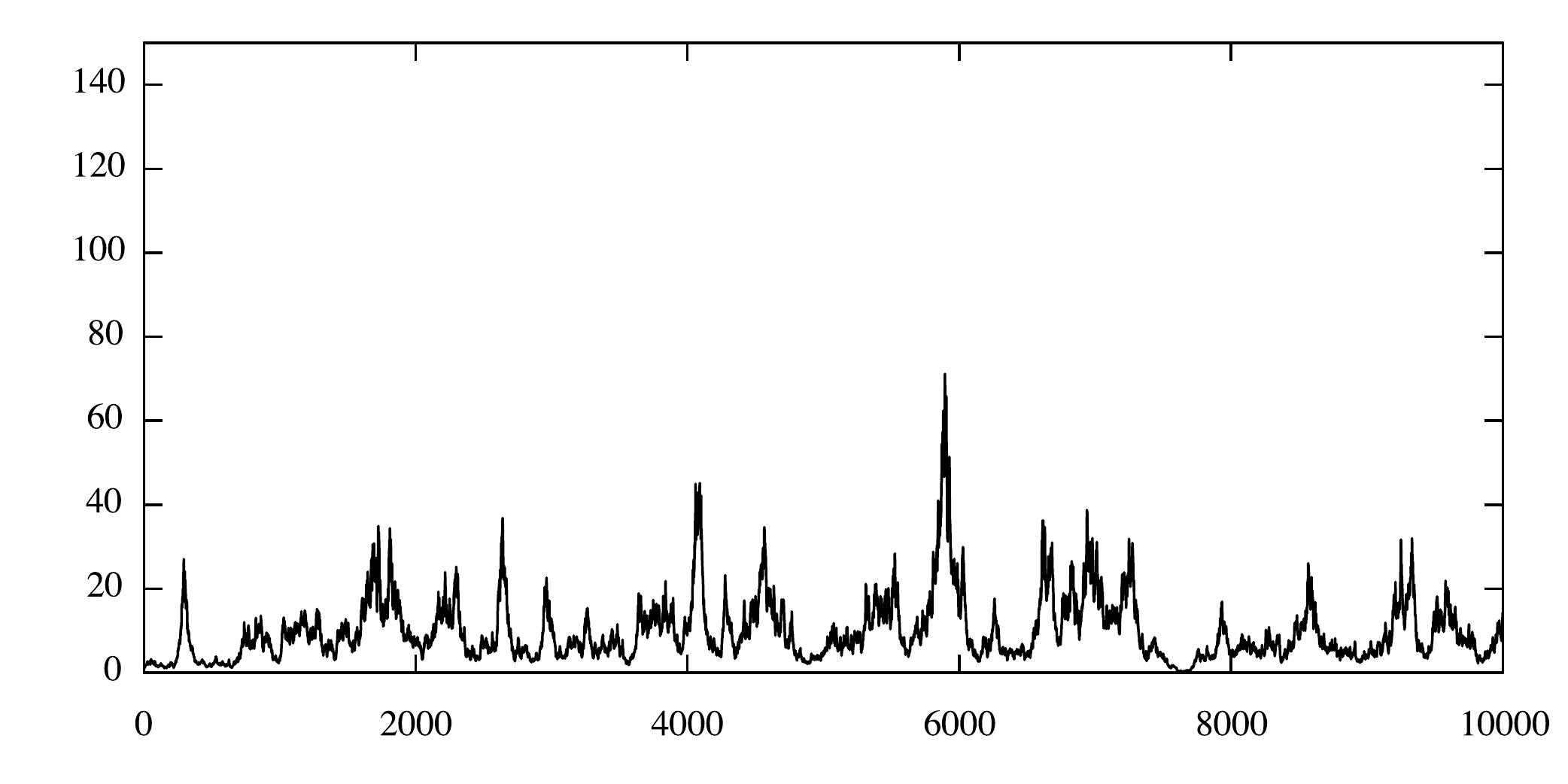}  
             \includegraphics[type=pdf, ext=.pdf, read=.pdf, width=0.32\columnwidth]{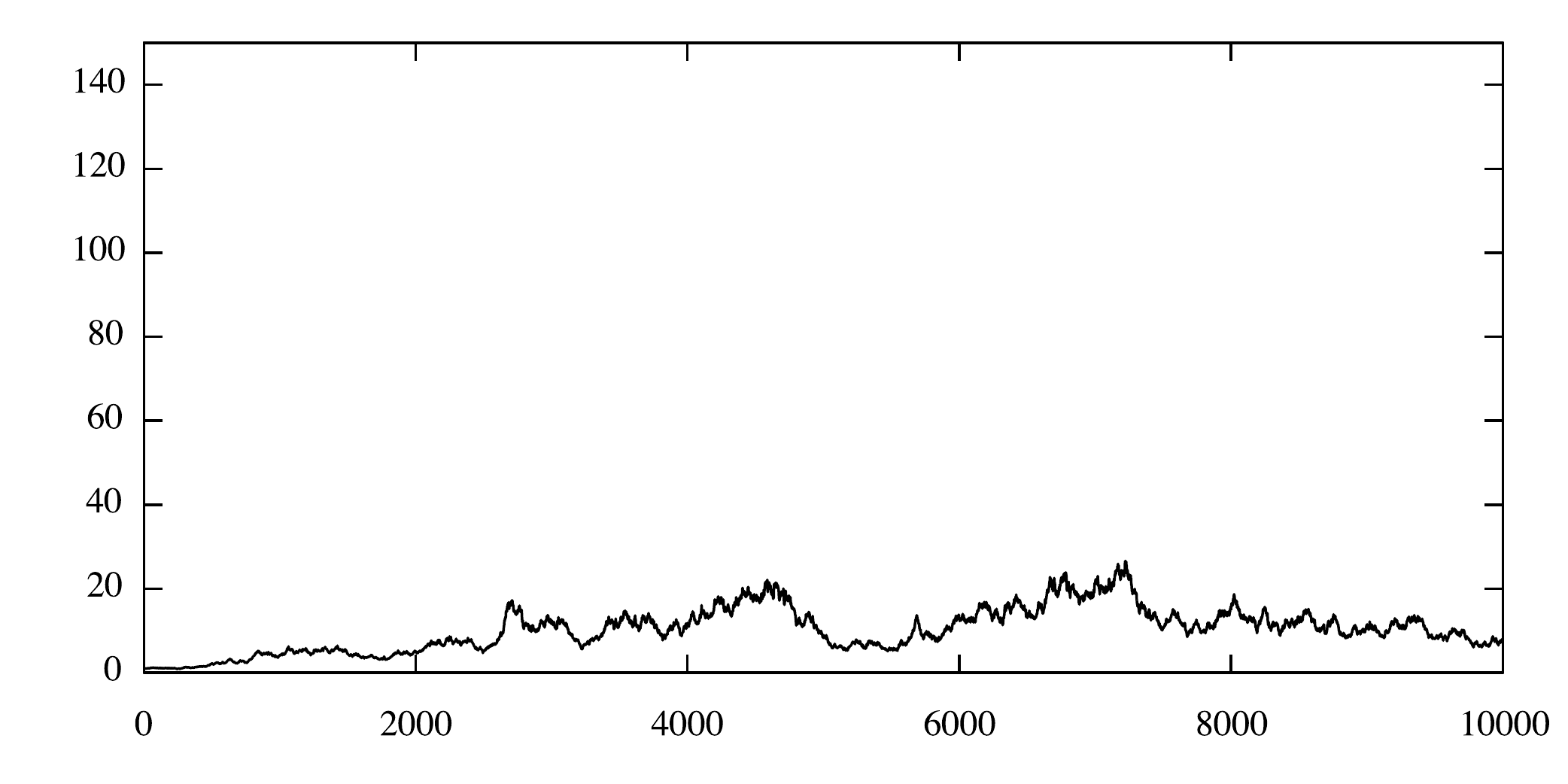}       
                   
            \includegraphics[type=pdf, ext=.pdf, read=.pdf, width=0.32\columnwidth]{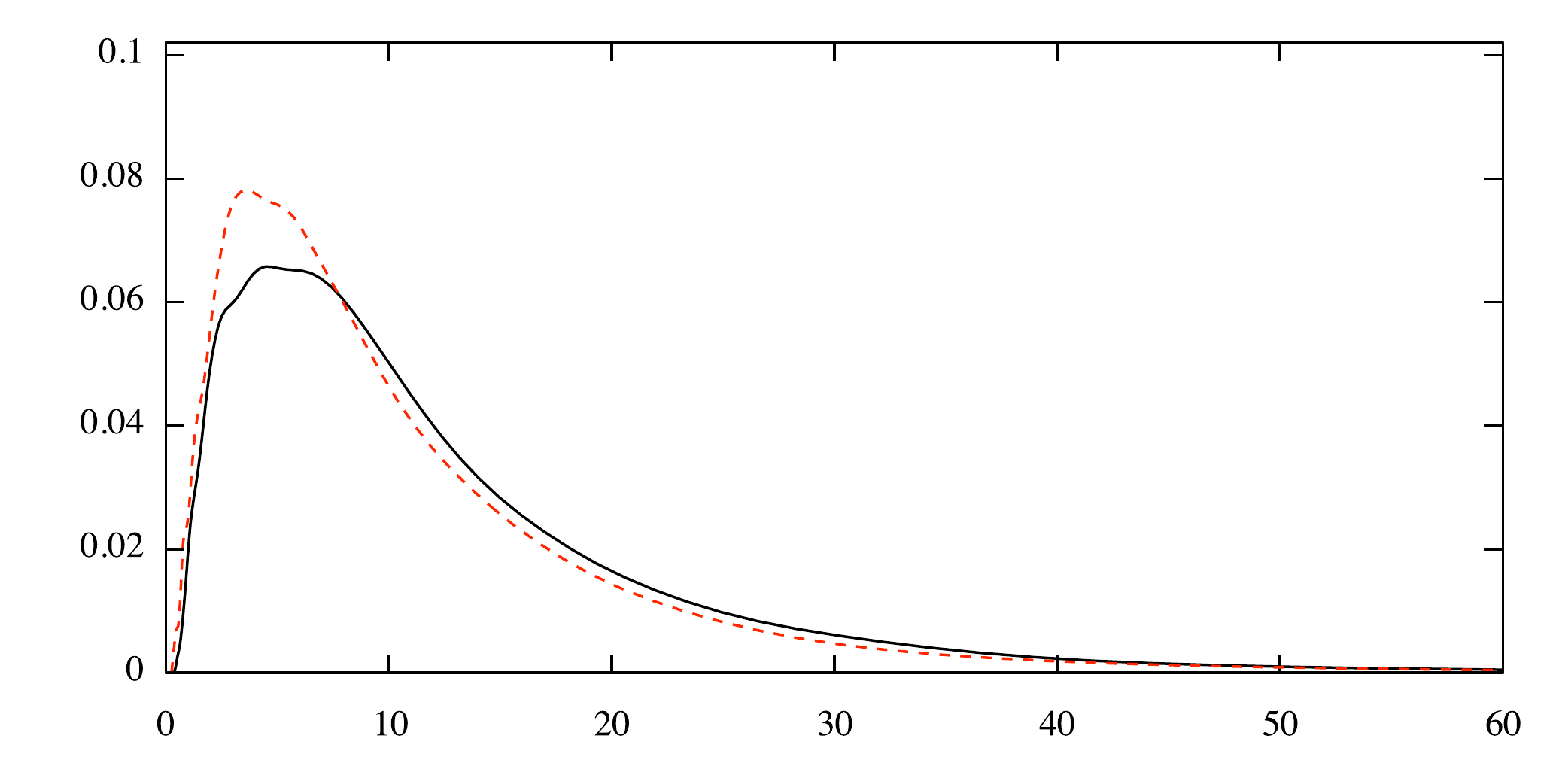}       \includegraphics[type=pdf, ext=.pdf, read=.pdf, width=0.32\columnwidth]{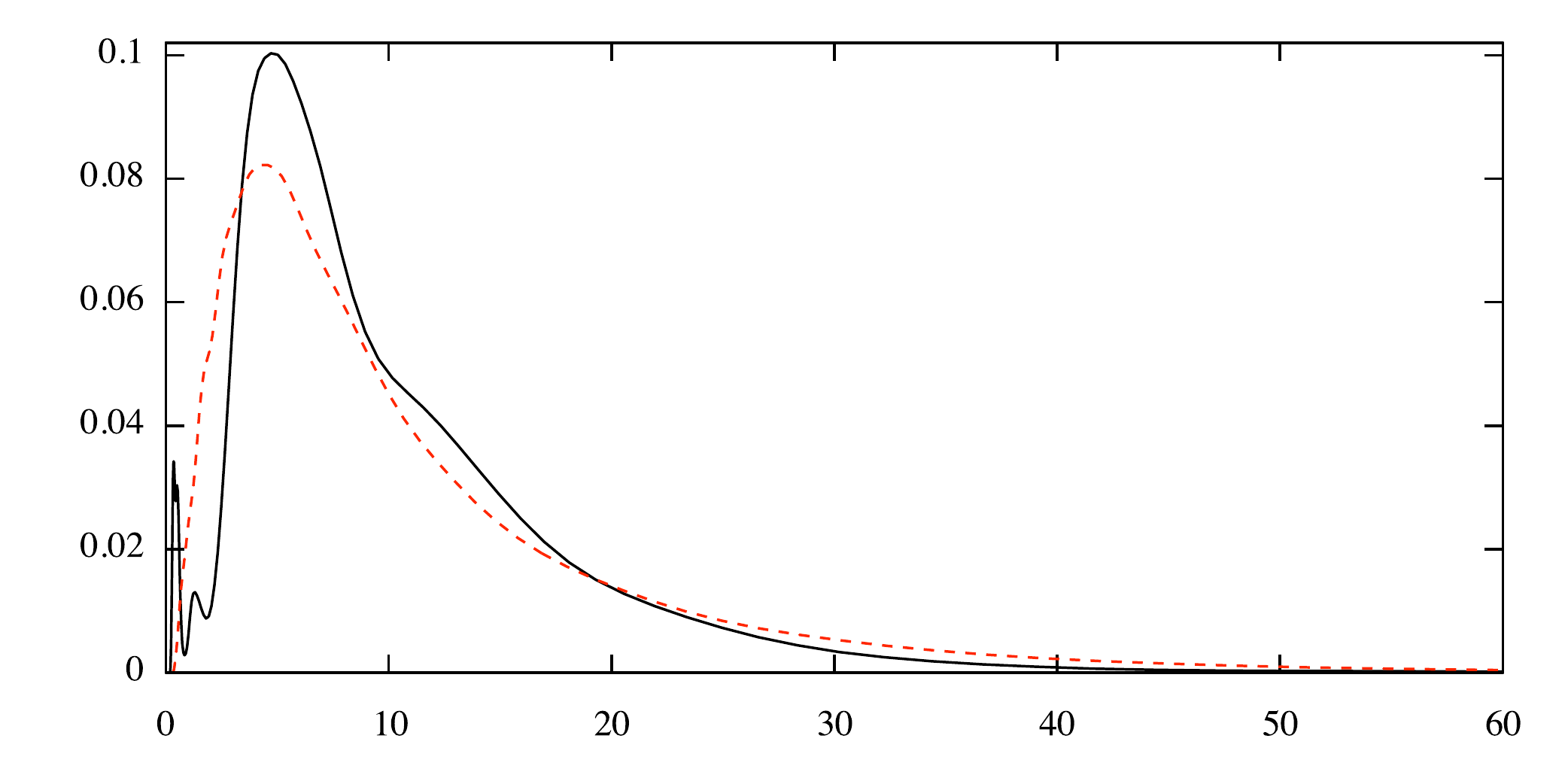}              
            \includegraphics[type=pdf, ext=.pdf, read=.pdf, width=0.32\columnwidth]{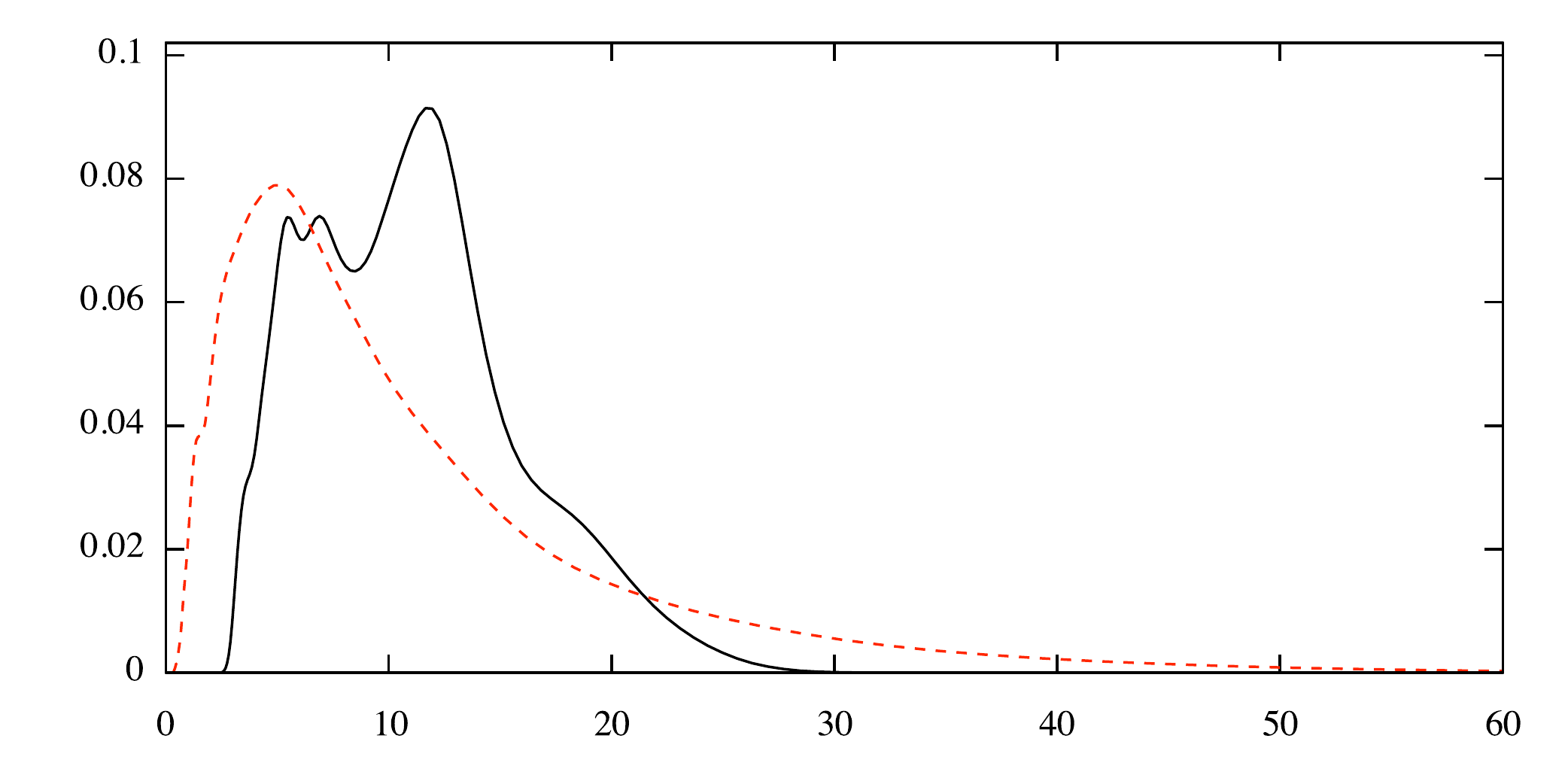}} 
            \caption{CA: $\delta$-chains (top)  and kernel density estimates of the posterior on $\delta$ (bottom) for dimensions $N=32,  512$ and $8192$ left to right. In dashed red in the density plots is the density estimate using MA, considered as the gold standard.} 
            \label{ch3:fig2}
\end{figure}

In Figure \ref{ch3:fig2} we see that for CA, in small dimensions the $\delta$-chain has a healthy mixing, however as predicted by Theorem \ref{ch3:thm1}, as $N$ increases it becomes increasingly slower and exhibits diffusive behaviour. This is also reflected in the density plots where we observe that as $N$ increases, the kernel density estimates computed using CA look less and less like the density estimates computed using MA which we consider to be optimal in this setting. In Figure \ref{ch3:fig3} we see that for NCA as expected, the $\delta$-chain appears to be robust with respect to the increase in dimension; this is also reflected in the density estimates using NCA which now look very close to the ones obtained using MA for all discretization levels.

 \begin{figure}[htp]
            \center{   
                     
            \includegraphics[type=pdf, ext=.pdf, read=.pdf, width=0.32\columnwidth]{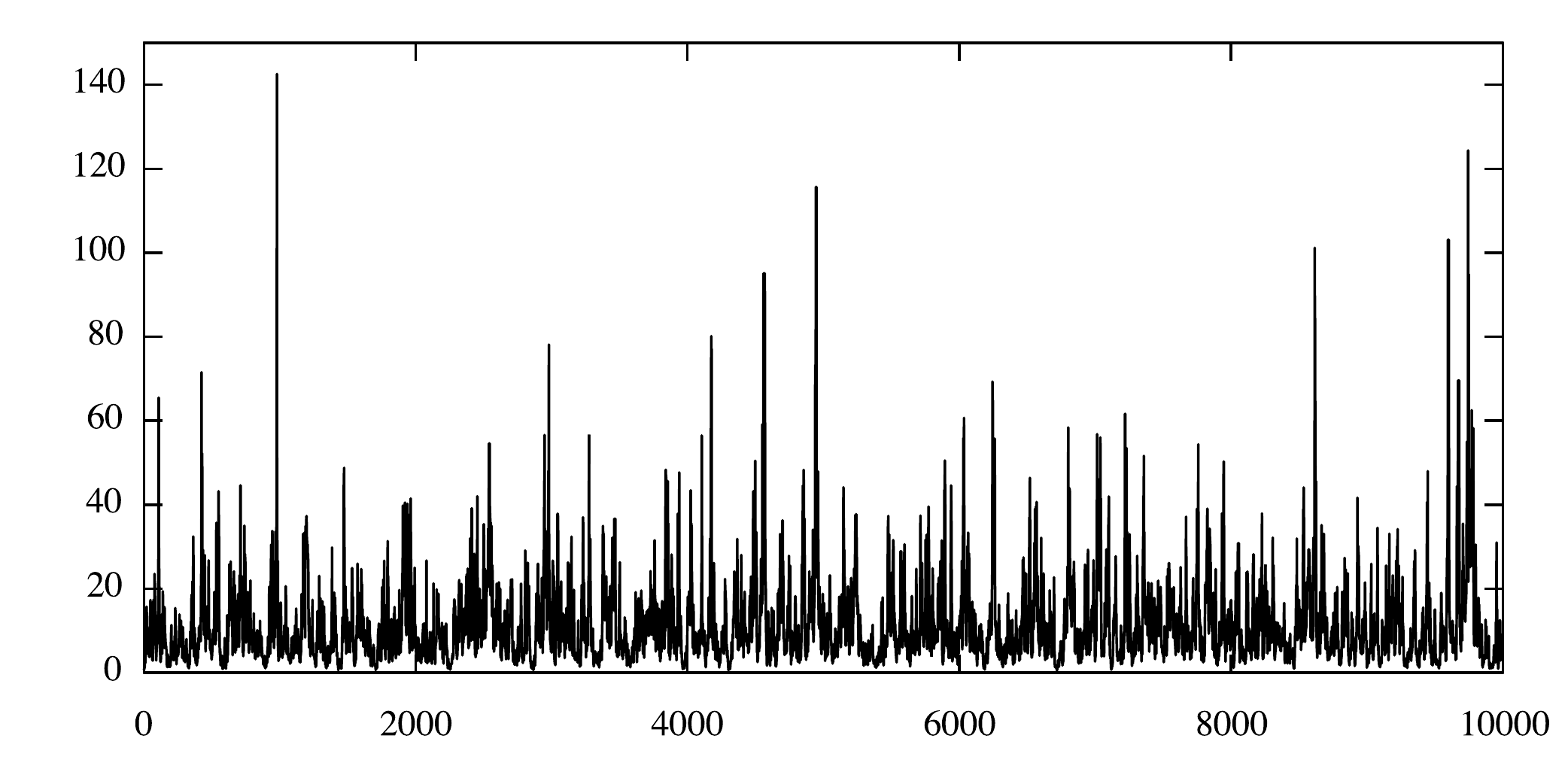}                 
            \includegraphics[type=pdf, ext=.pdf, read=.pdf, width=0.32\columnwidth]{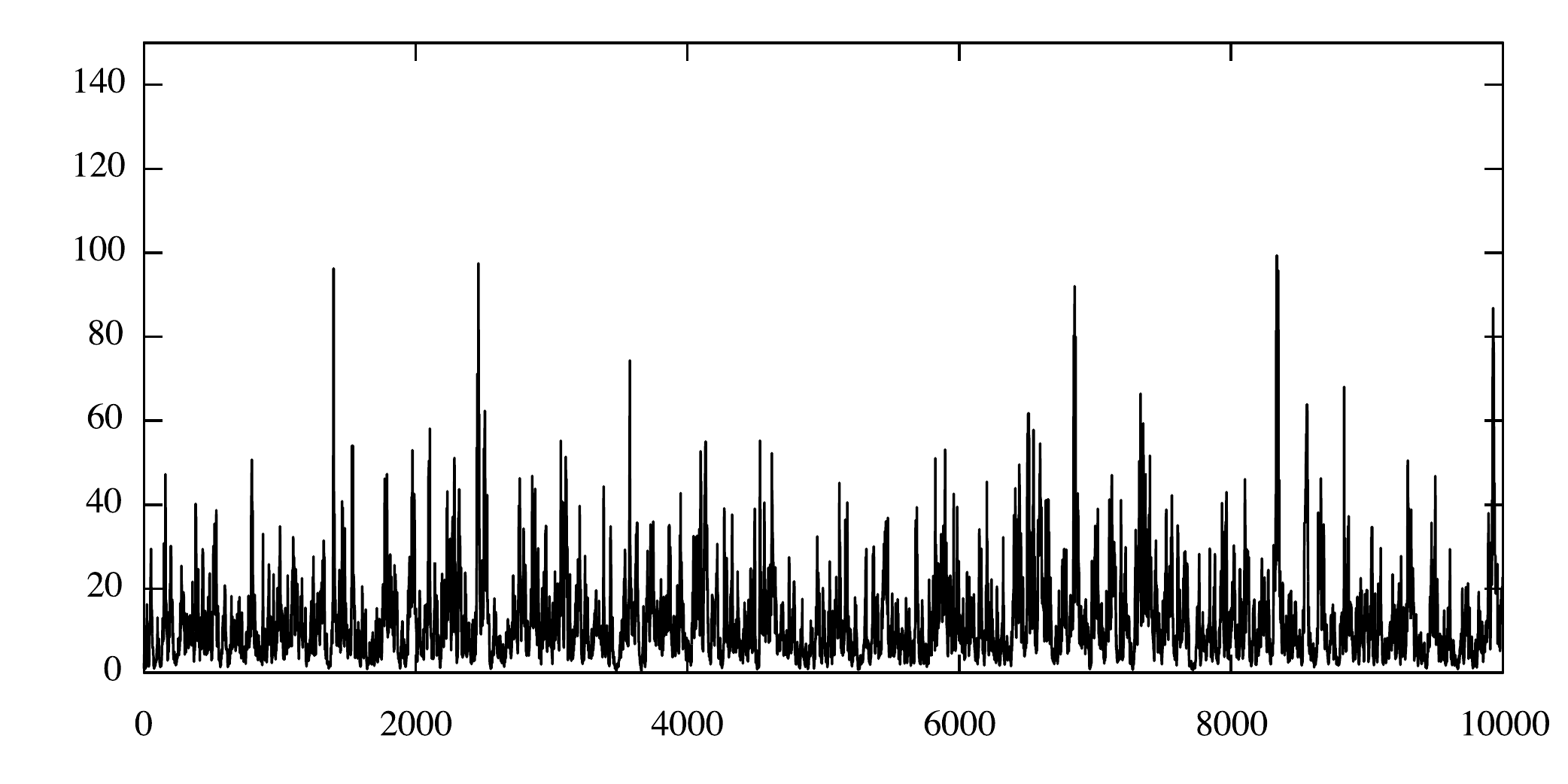}                   
            \includegraphics[type=pdf, ext=.pdf, read=.pdf, width=0.32\columnwidth]{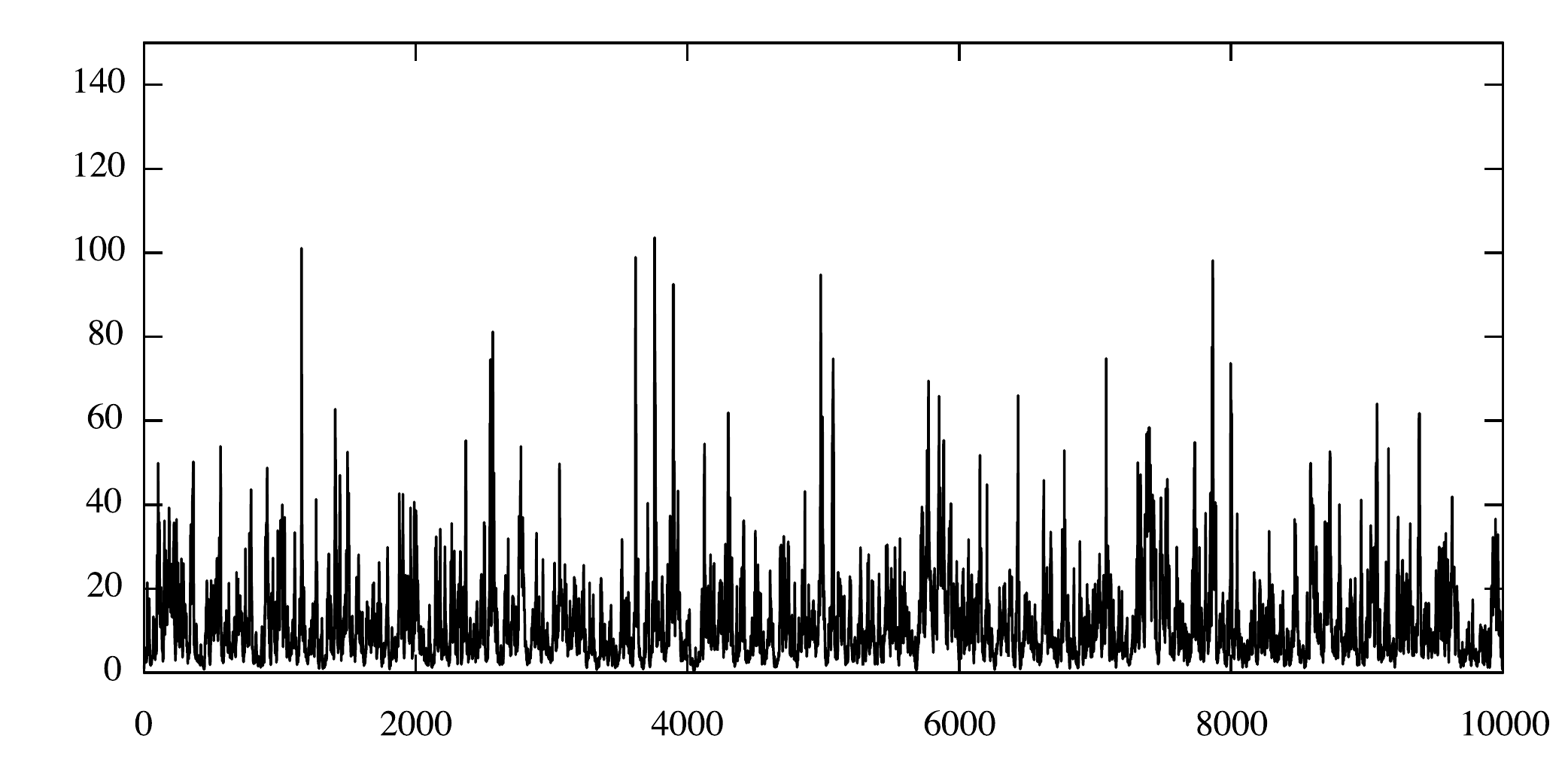}                       
            \includegraphics[type=pdf, ext=.pdf, read=.pdf, width=0.32\columnwidth]{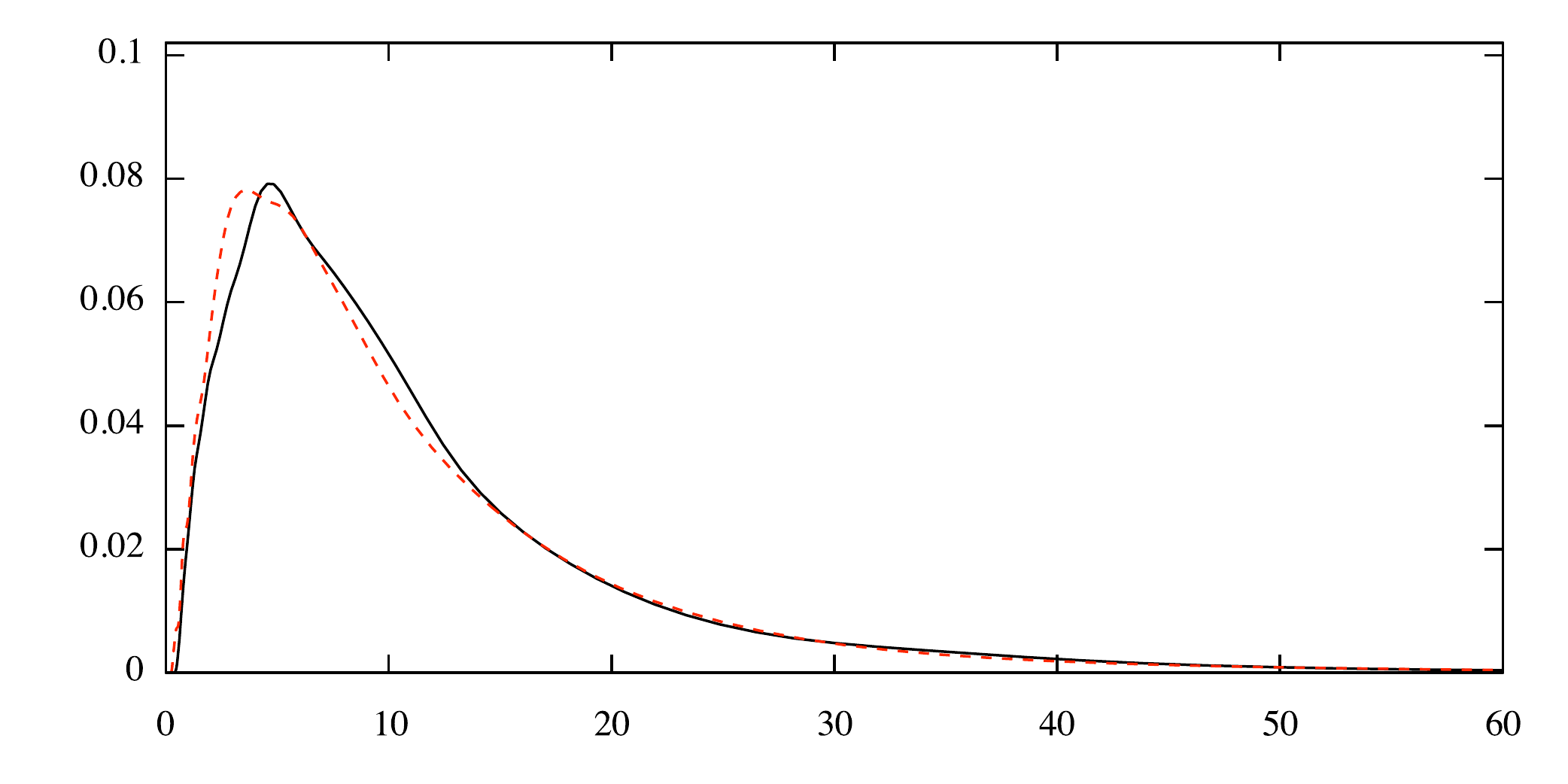}           
            \includegraphics[type=pdf, ext=.pdf, read=.pdf, width=0.32\columnwidth]{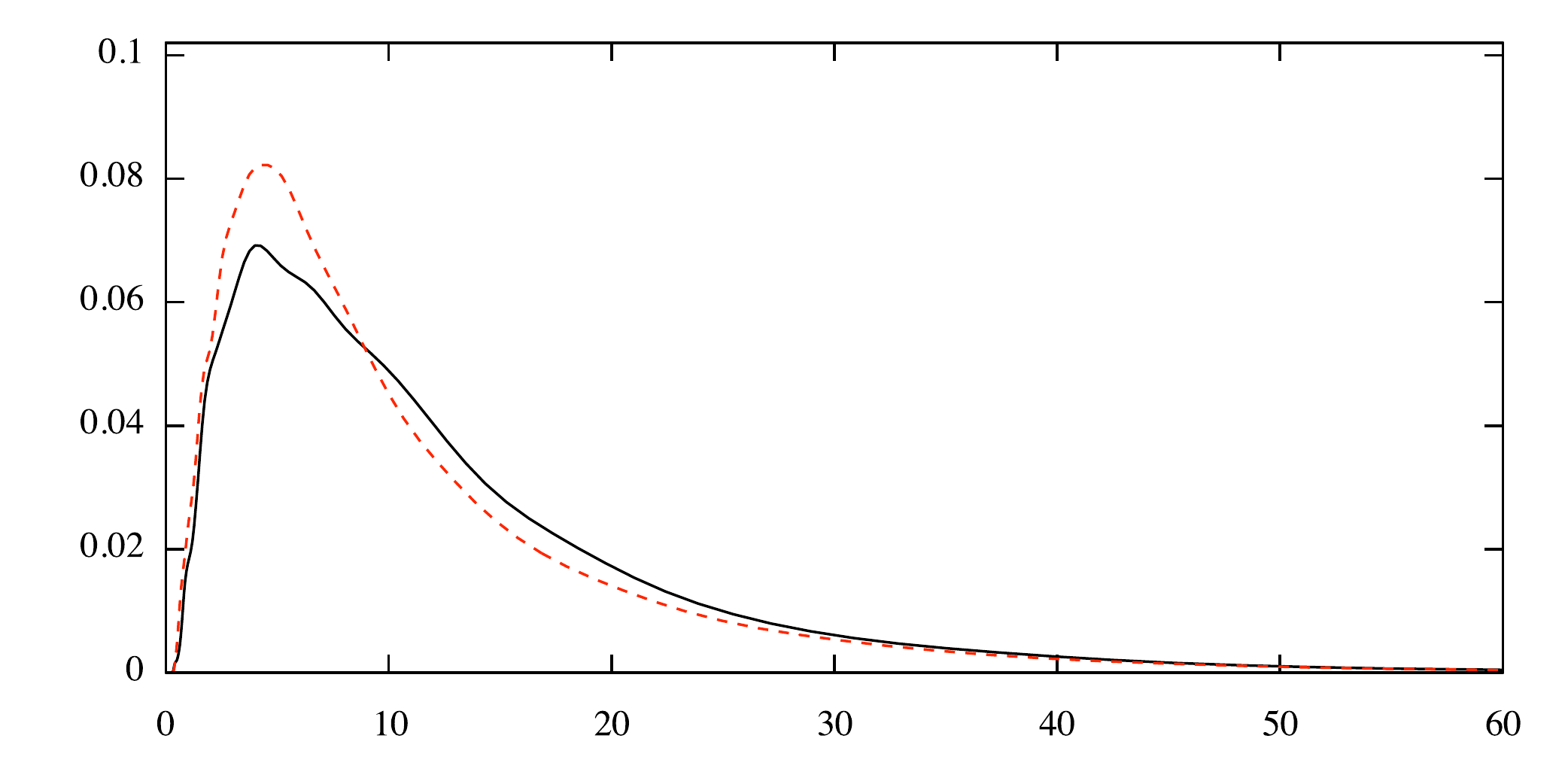}            
           \includegraphics[type=pdf, ext=.pdf, read=.pdf, width=0.32\columnwidth]{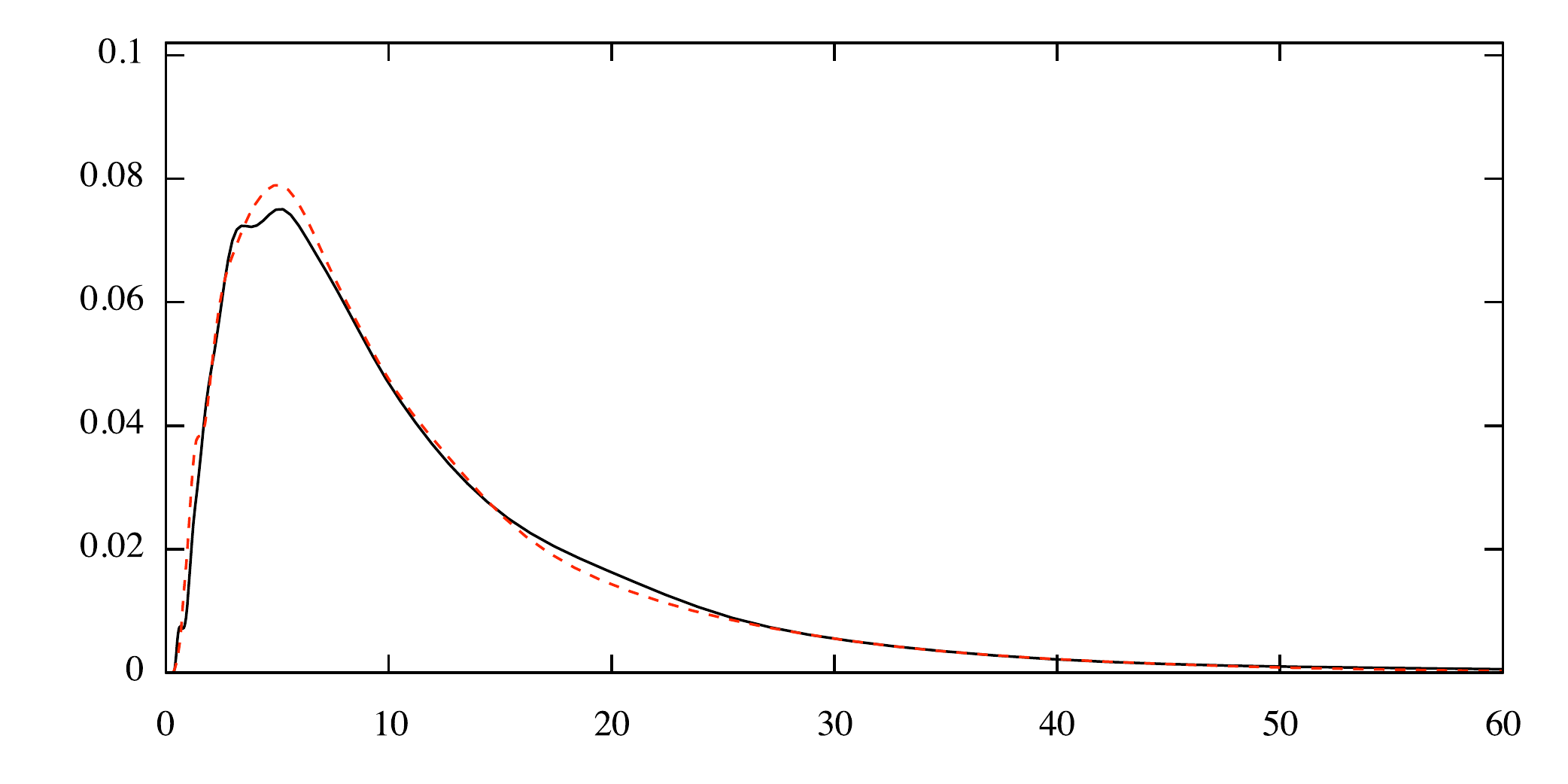}}
            \caption{NCA: $\delta$-chains (top)  and kernel density estimates of the posterior on $\delta$ (bottom) for dimensions $N=32,  512$ and $8192$ left to right. {In dashed red in the density plots is the density estimate using MA, considered as a gold standard.}}  
            \label{ch3:fig3}
    \end{figure}

            Our observations in Figures \ref{ch3:fig2} and \ref{ch3:fig3} are supported by the autocorrelation plots presented in Figure \ref{ch3:fig4}. The rate of decay of correlations in the $\delta$-chain in CA appears to decrease as the dimension increases, and in particular for $N=8192$ the correlations seem not to decay at all. On the contrary, the rate of decay of correlation in the $\delta$-chain in NCA appears not to be affected by the increase in dimension and is very similar to the one in MA.
 
\begin{figure}[htp]
           \center{ \includegraphics[type=pdf, ext=.pdf, read=.pdf, width=0.3\columnwidth]{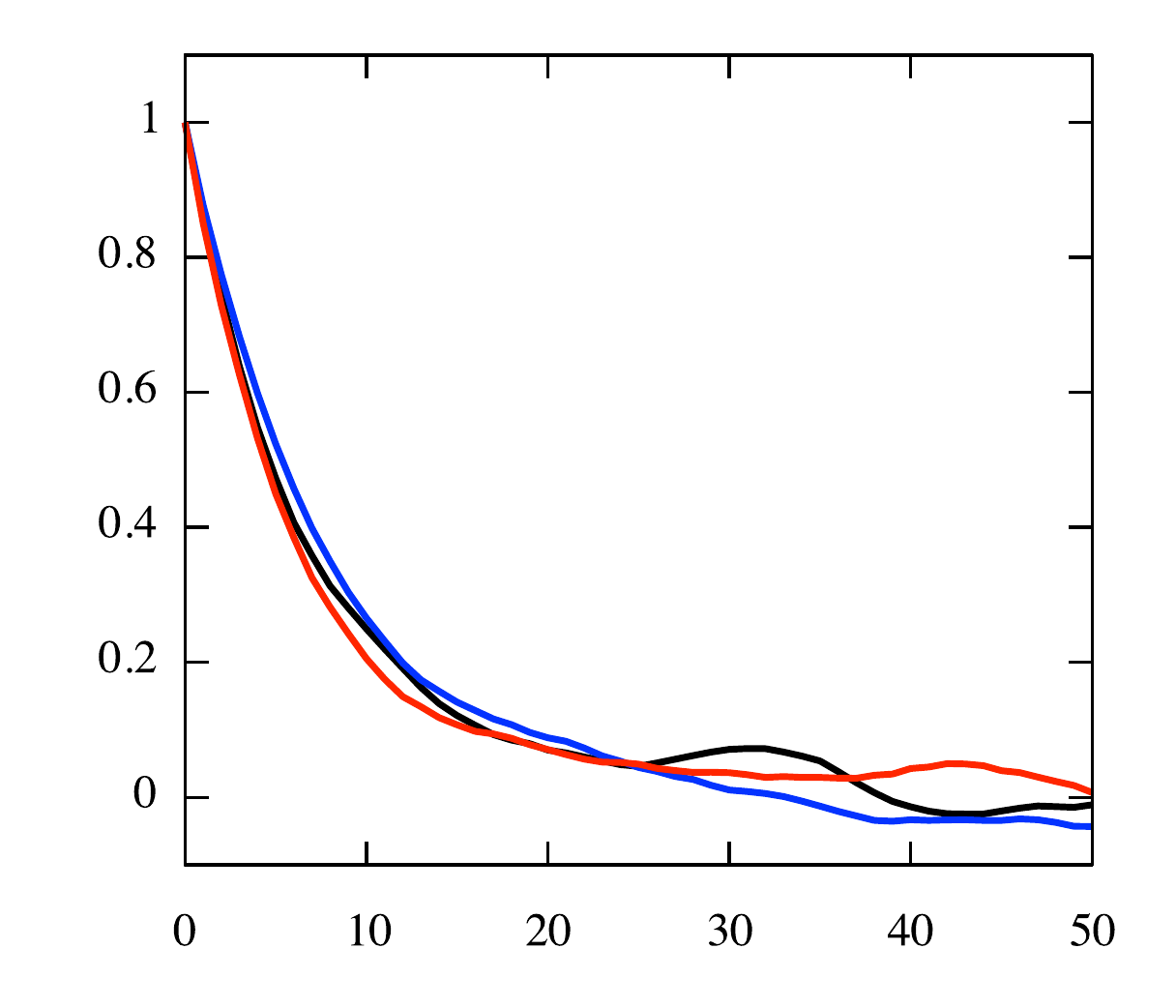}
            \includegraphics[type=pdf, ext=.pdf, read=.pdf, width=0.3\columnwidth]{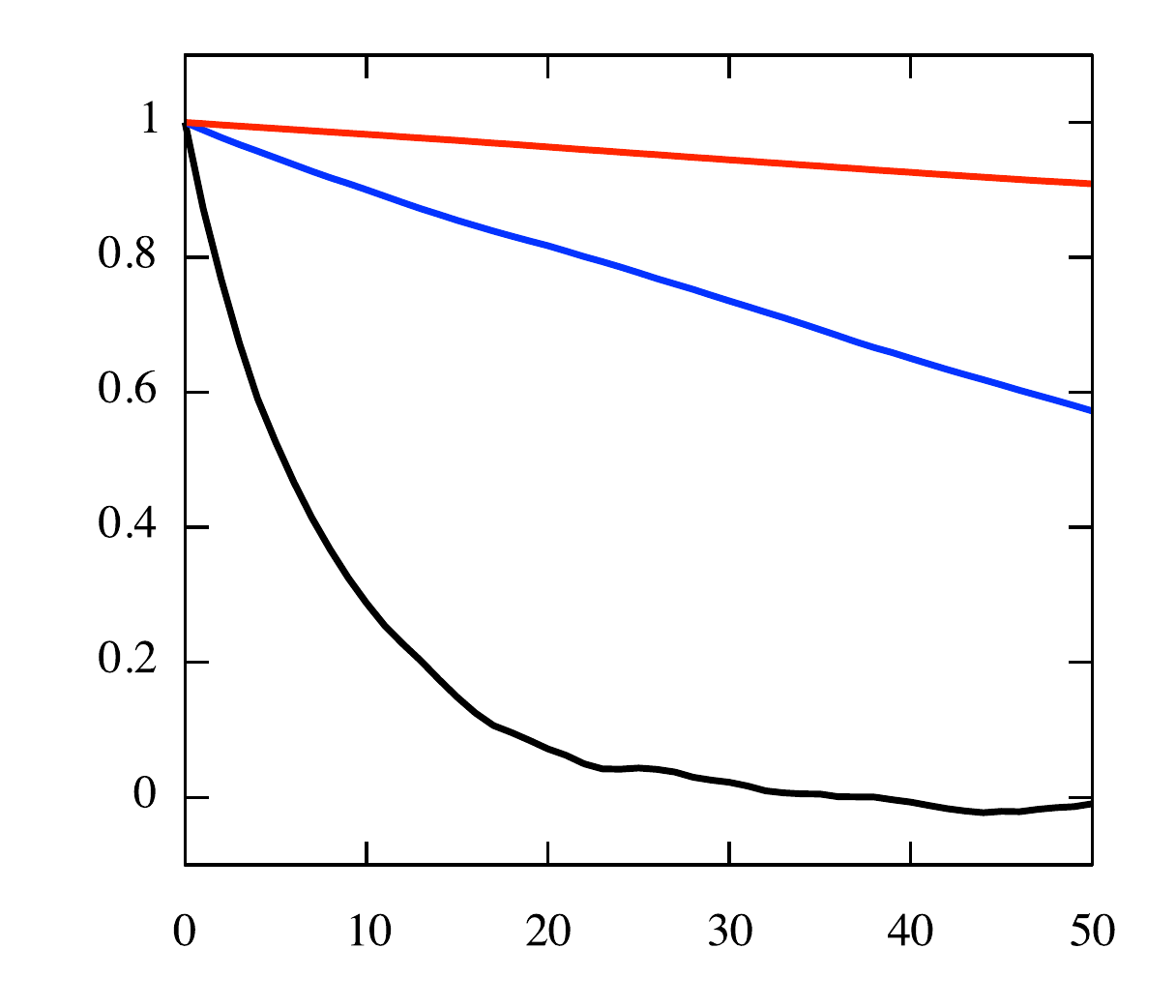}
             \includegraphics[type=pdf, ext=.pdf, read=.pdf, width=0.3\columnwidth]{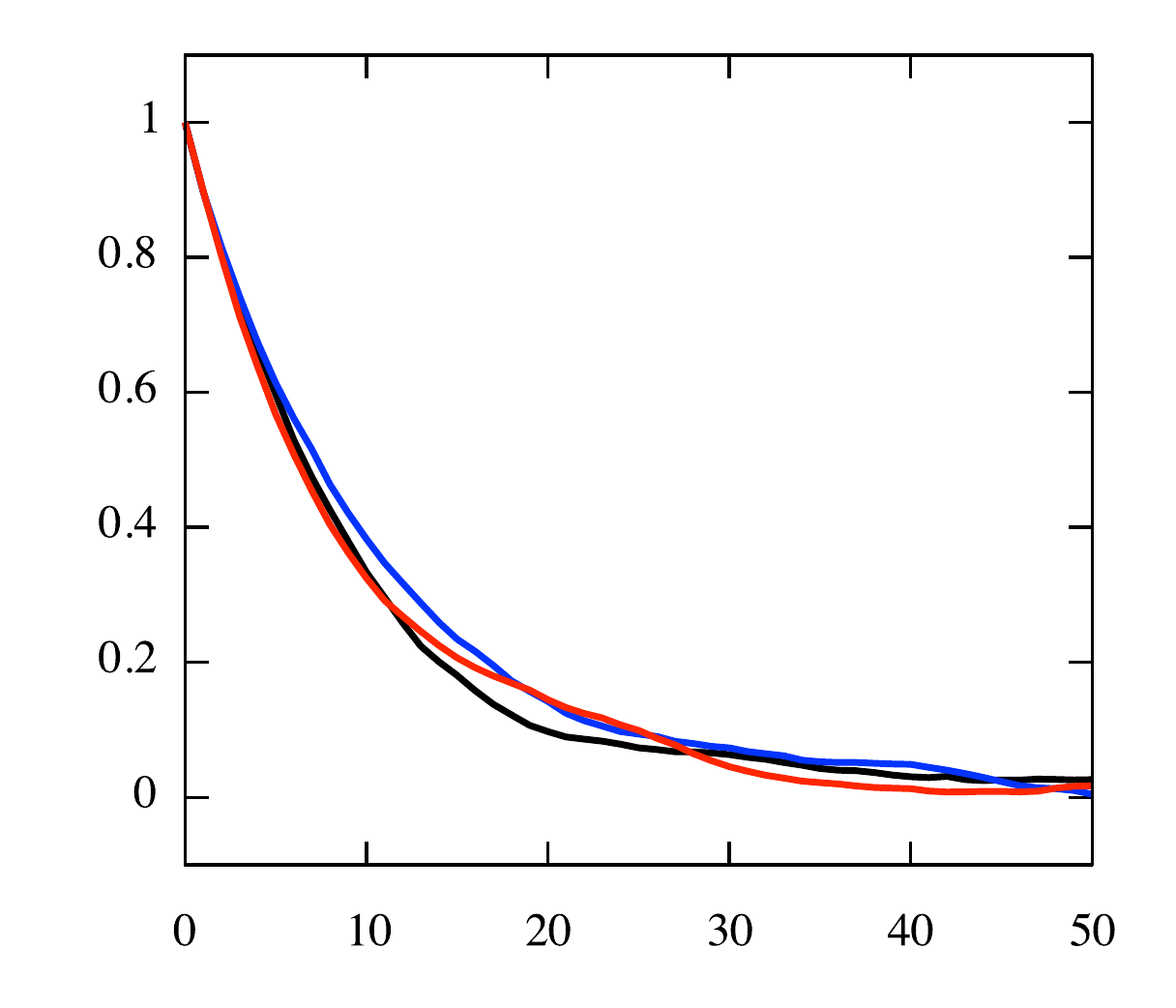}
           
           }   
            \caption{Autocorrelation functions of $\delta$-chain, dimensions 32 (black), 512 (red) and 8192 (blue); {left for MA, middle for CA, right for NCA}.}  
            \label{ch3:fig4}       
\end{figure}


\subsection{Linear Bayesian inverse problem with coarse data using finite difference discretization}\label{ch3:nex2}
We consider a modification of the simultaneously diagonalizable setup described in subsection \ref{ch3:ssec:diag}, where $\X=L^2(\I), \I=(0,1)$ and we allow $\bK$ to map $\X$ into $\R^M$ and hence have data $y\in\R^M$. This setting is not directly covered by the theoretical analysis presented in section \ref{ch3:sec:main}, however our theory readily generalizes to cover this setting; we refer the interested reader to the PhD thesis \cite[section 4.5]{SA13} for more details. The generalized analysis holds again under Assumptions \ref{ch3:ass1} on the discrete-level based on intuition which holds for problems satisfying Assumptions \ref{ch3:infass1} on the underlying continuum model for the unknown $\bu$. 

In particular, we consider the problem of recovering a true signal ${\bu^\dagger}$, by observing a blurred version of it at $M$ equally spaced points $\{\frac{1}{M+1},...,\frac{M}{M+1}\}$, polluted by additive independent Gaussian noise of constant variance $\hl^{-1}$. We define $\bA_0$ to be the negative Laplacian with Dirichlet boundary condtions in $\I$. We let $\bP$ be defined as in subsection \ref{fd} and define $\tilde{\bK}=(\bI+\frac{1}{100\pi^2} \bA_0)^{-1}$, and consider the case $\bK=\bP \tilde{\bK}$, $\bC_0=\bA_0^{-1}$ and $\C_1=I_M$ in the setting of subsection \ref{sec:linear} and where $I_M$ is the $M\times M$ identity matrix. Notice that due to the smoothing effect of $\tilde{\bK}$, the operator $\bK$ is bounded in $\X$. However, due to the presence of $\bP$, $\bK$ is not simultaneously diagonalizable with $\bC_0$.

We now check that this problem satisfies Assumptions \ref{ch3:infass1}. Indeed, assuming without loss of generality that $\hl=\delta=1$, by \cite[Example 6.23]{AS10} we have that the posterior covariance and mean satisfy (\ref{ch3:eq:prec}) and (\ref{ch3:eq:mean}), hence $\bC_0^{-\frac12}\bme(y)=\bC_0^{-\frac12}(\bC_0^{-1}+\bK^\ast \bK)^{-1}\bK^\ast y=(I+\bC_0^\frac12\bK^\ast \bK\bC_0^\frac12)^{-1}\bC_0^\frac12\bK^\ast y$, where $\bC_0^\frac12 \bK^\ast y\in \X,$  and $(I+\bC_0^\frac12\bK^\ast \bK\bC_0^\frac12)^{-1}$ is bounded in $\X$ by the nonnegativity of $\bC_0^\frac12\bK^\ast \bK\bC_0^\frac12$. Furthermore, we have that $\tr(\C_1^{-\frac12}\bK\bC_0 \bK^\ast \C_1^{-\frac12})=\tr(\bK \bC_0 \bK^\ast)$, which is finite since $\bK \bC_0 \bK^\ast$ is an $M\times M$ matrix.

We discretize this setup at level $N$, using the finite differences approximation as explained in subsection \ref{fd}. In particular, we discretize $\bA_0, \bP$ and $\bP^\ast$ by replacing them with the matrices $\A_0, P$ and $(N+1)P^T$  respectively as in subsection \ref{fd}; this induces a discretization of the operators $\bK$ and $\bC_0$ by replacing them with the corresponding matrices $K$ and $\C_0$ calculated through the appropriate functions of $\A_0$ and $P$. In defining $K$, we also replace the identity operator by the $N\times N$ identity matrix. We do not prove that this discretization scheme satisfies Assumptions \ref{ch3:ass1}, however we expect this to be the case. 

We assume that we have data produced from the underlying true signal ${\bu^\dagger}(x)=0.75\cdot\1_{[0.1,0.25]}(x)+0.25\cdot\1_{[0.35,0.38]}+\sin^4(2\pi x)\cdot\1_{[0.5,1]}(x), \;x\in\I.$ In particular, we construct data of the form \[y=\bK{\bu^\dagger}+\lambda^{-\frac12}\C_1^\frac12\xi,\]
where $\lambda=100$ and 
using a discretization level $N_c=8192$ for the unknown; we treat this discretization level as the continuum limit.

We implement Algorithms \ref{ch3:algstd} (CA), \ref{ch3:algrep} (NCA) and \ref{ch3:algmar} (MA) for constant number of observation points $M=15$,  and for discretization levels of the unknown $N=15, 127, 1023$, with hyper-parameters $\ad=1,\bd=10^{-4},$ chosen to give uninformative hyper-priors, that is, hyper-priors whose variance is much larger than their mean. Following the discussion in subsection \ref{contr}, we view MA as the gold standard and benchmark CA and NCA against it. We use $10^4$ iterations and choose $\de{0}=1$ in all cases. We again use a constant burn-in time of $10^3$ iterations.

In Figure \ref{ch3:fig6} we plot the true solution, the noisy data and the sample means and credibility bounds using CA and NCA for $N=1023$. The sample means and credibility bounds at other discretization levels of the unknown are similar and are therefore omitted. 
\begin{figure}[htp]
            \includegraphics[type=pdf, ext=.pdf, read=.pdf, width=0.32\columnwidth]{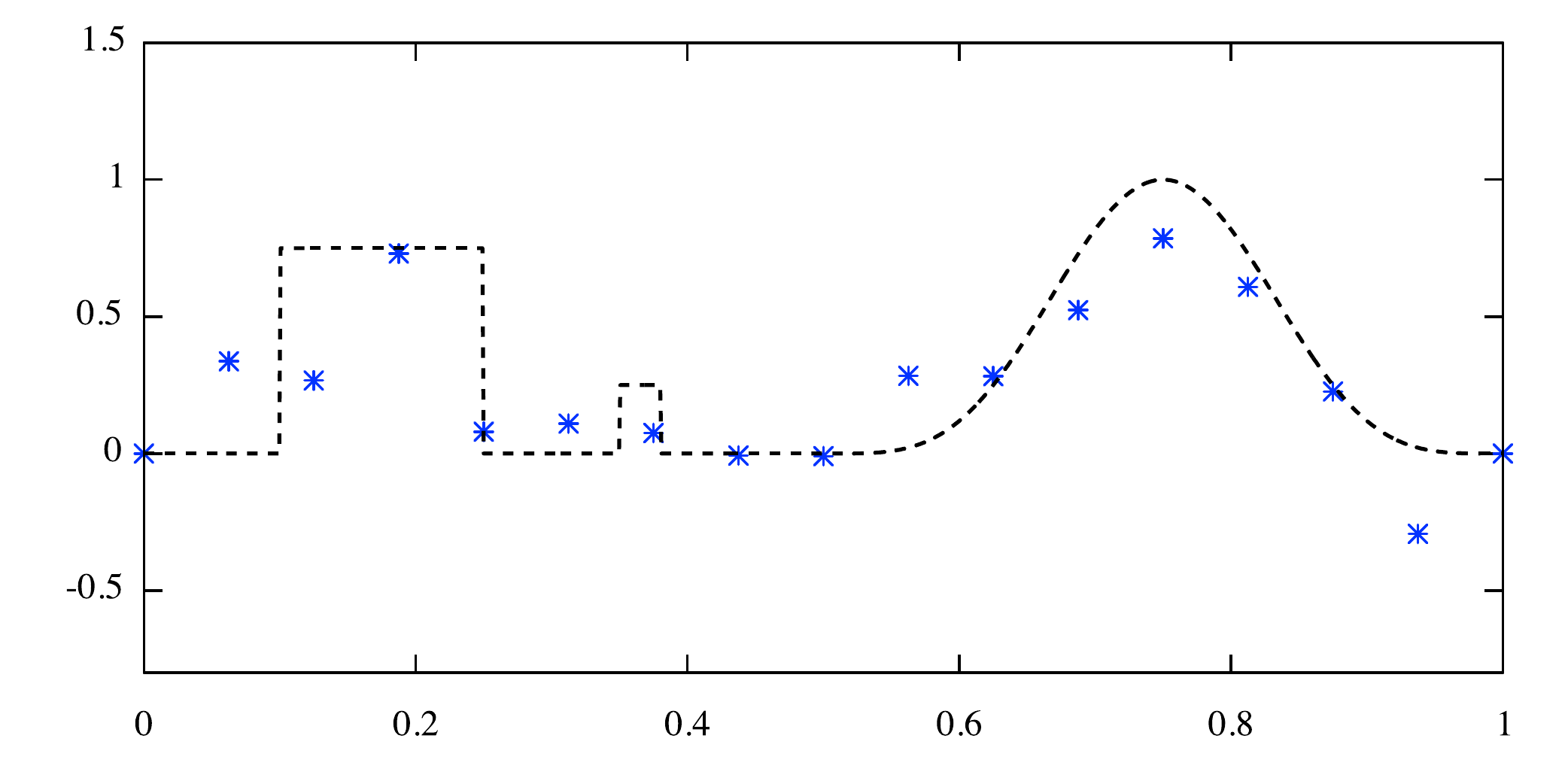} 
            \includegraphics[type=pdf, ext=.pdf, read=.pdf, width=0.32\columnwidth]{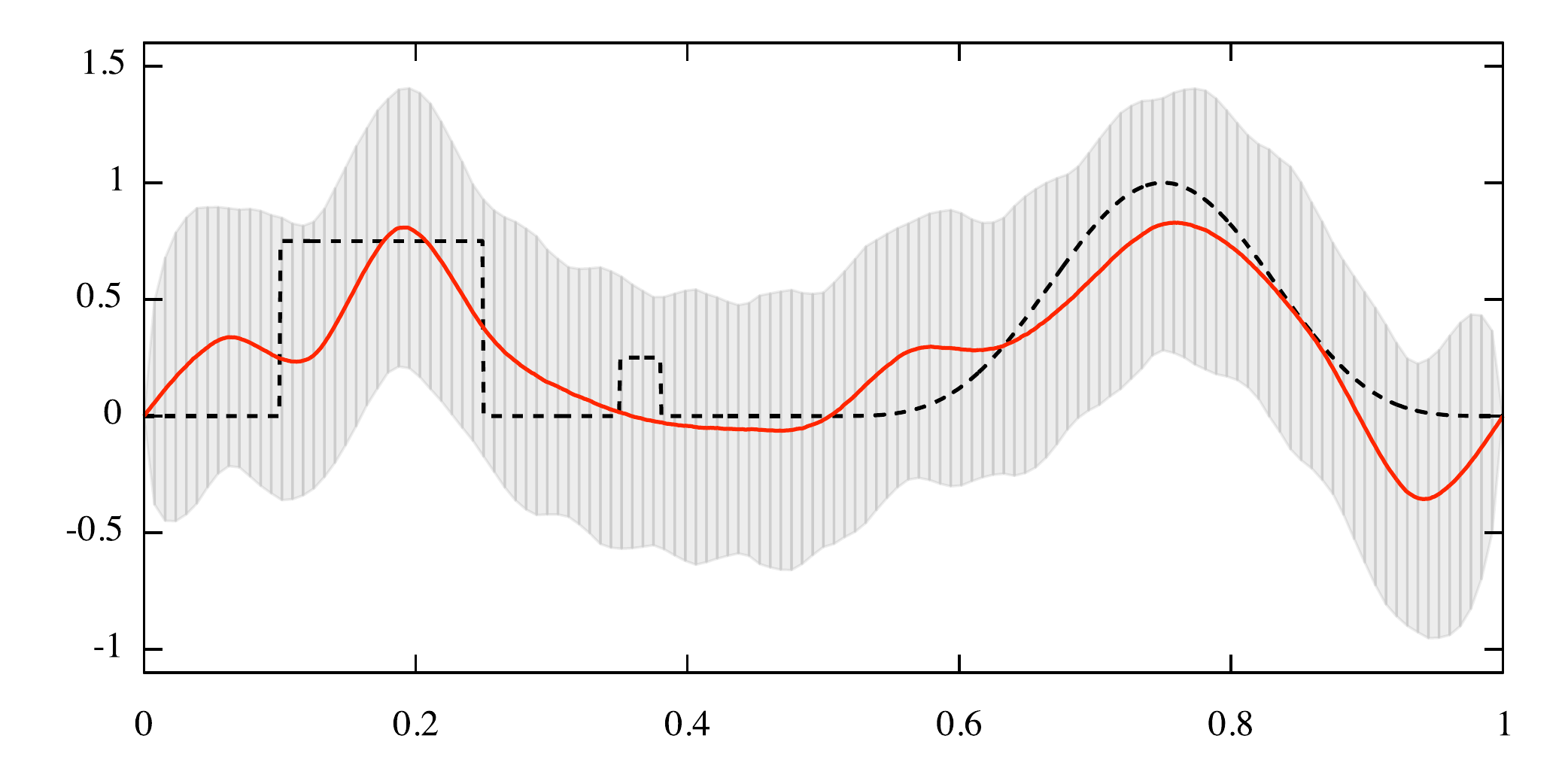}
              \includegraphics[type=pdf, ext=.pdf, read=.pdf, width=0.32\columnwidth]{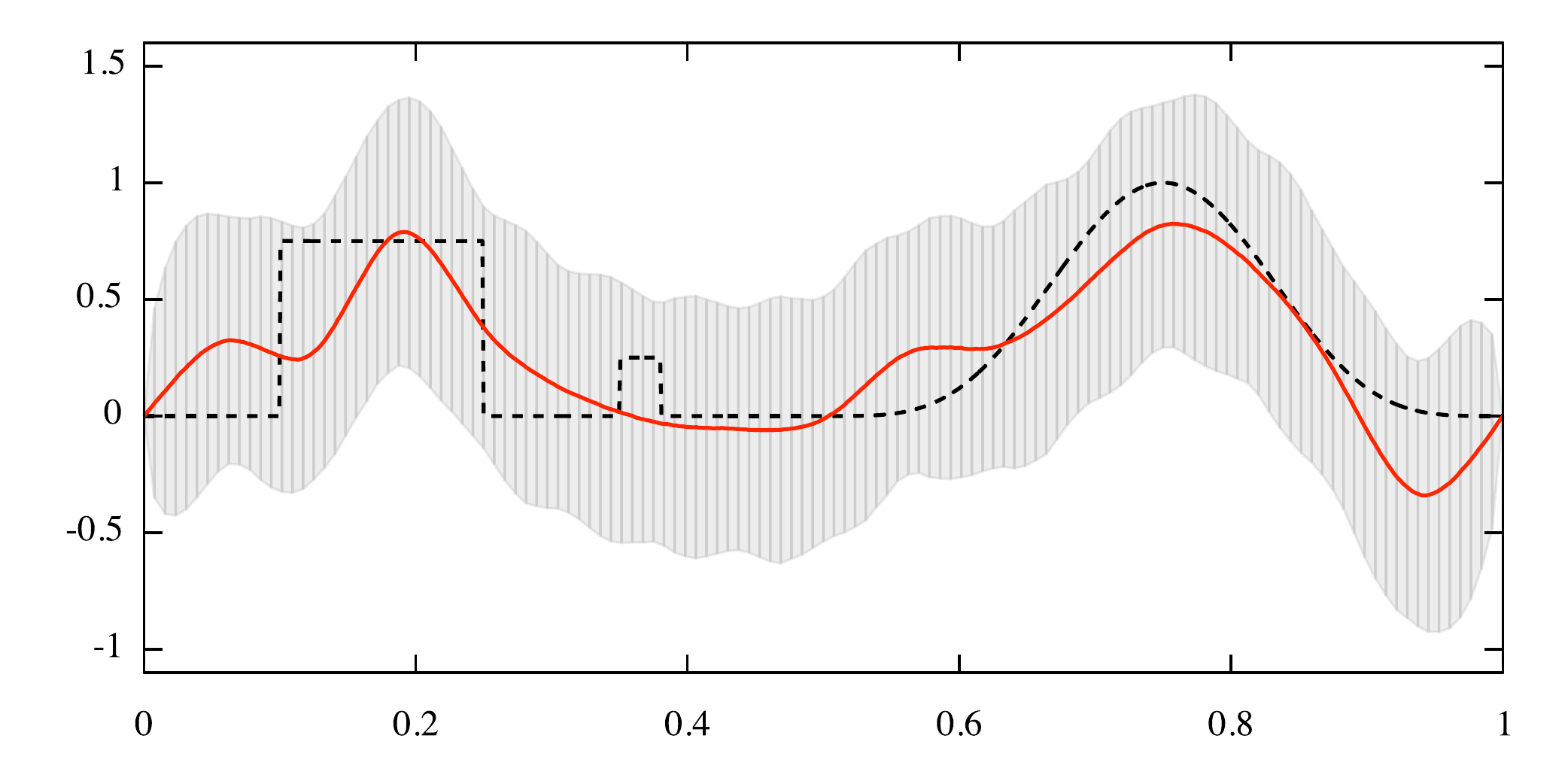}
            \caption{Left: true solution (dashed black) and discrete blurred noisy data (blue asterisks). Middle and right: true solution (dashed black), sample mean (red continuous) and 87.5$\%$ credibility bounds (shaded area) for CA (middle) and NCA (right). Dimensions of true solution and observed data are $N=1023$ and $M=15$ respectively.} 
            \label{ch3:fig6}           
\end{figure}

In Figure \ref{ch3:fig7} we see that for CA, in small dimensions the $\delta$-chain has a healthy mixing, however as predicted by our theory, as $N$ increases it becomes increasingly slower and exhibits diffusive behaviour. This is also reflected in the density plots where we observe that as $N$ increases, the kernel density estimates computed using CA look less and less like the density estimates computed using MA which we consider to be optimal in this setting. In Figure \ref{ch3:fig8} we see that for NCA the $\delta$-chain appears to be robust with respect to the increase in dimension; this is also reflected in the density estimates using NCA which now look very close to the ones obtained using MA for all discretization levels.

 \begin{figure}[htp]
           \center{
            \includegraphics[type=pdf, ext=.pdf, read=.pdf, width=0.32\columnwidth]{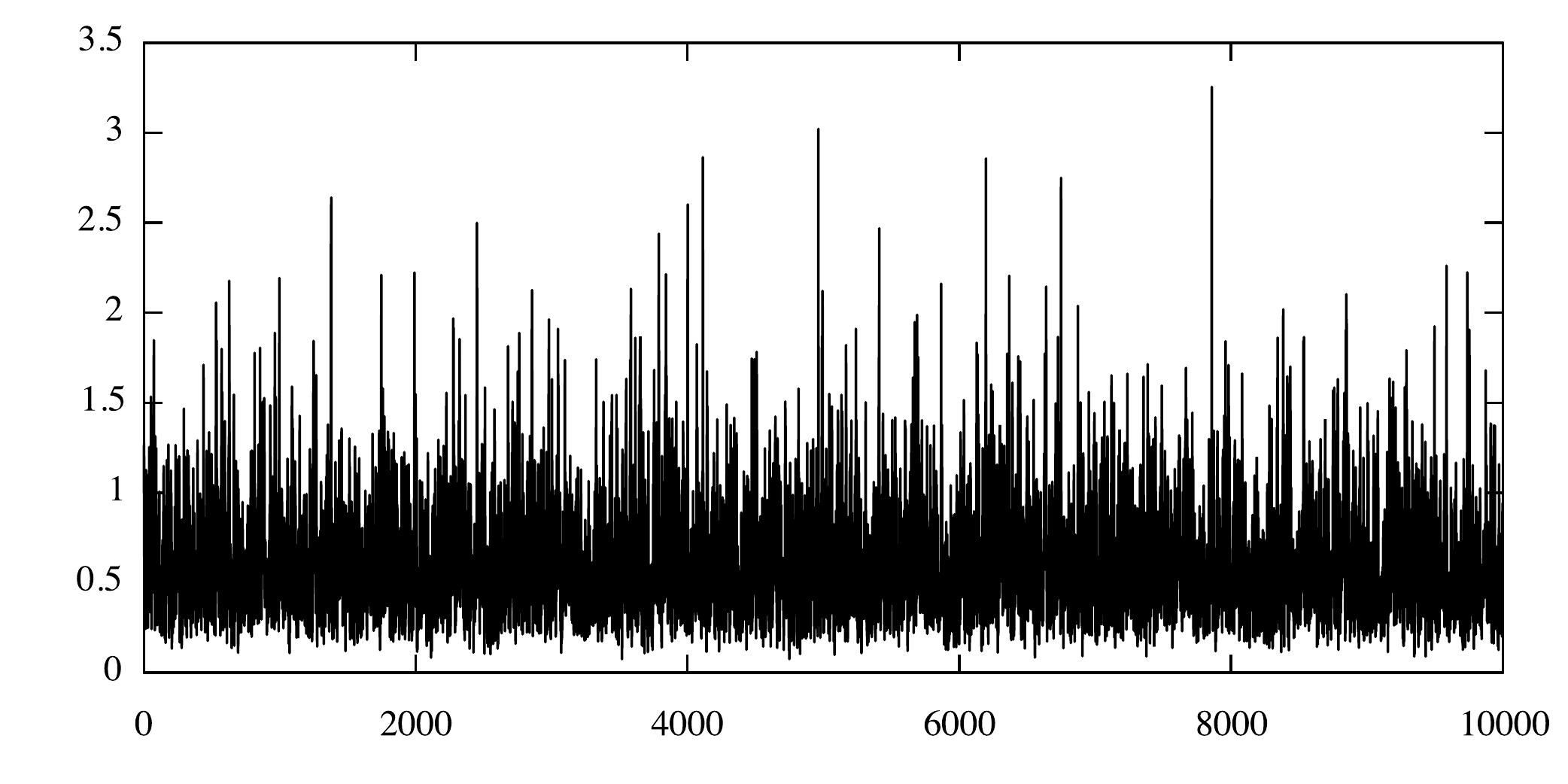}         
           \includegraphics[type=pdf, ext=.pdf, read=.pdf, width=0.32\columnwidth]{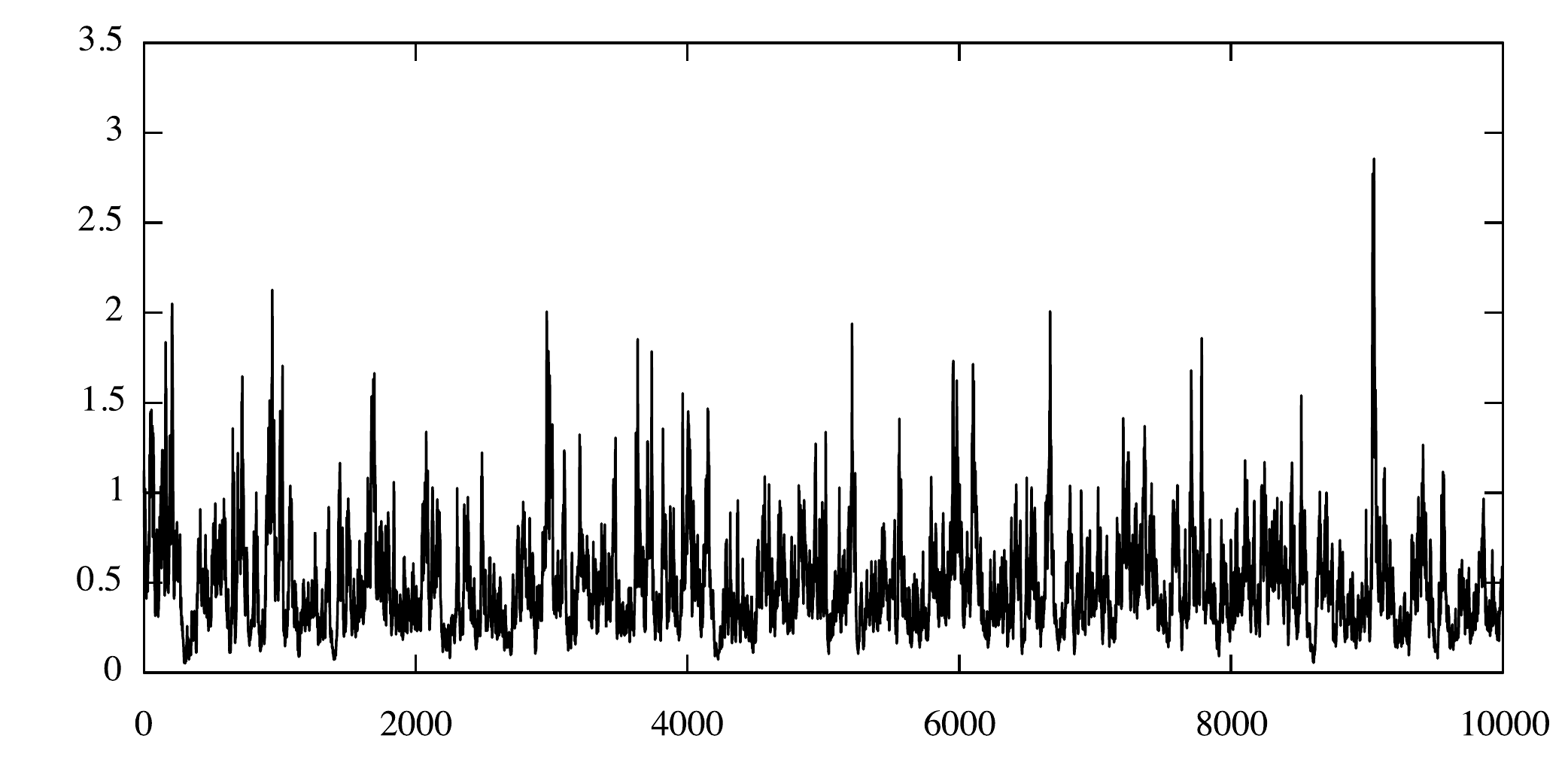}  
          \includegraphics[type=pdf, ext=.pdf, read=.pdf, width=0.32\columnwidth]{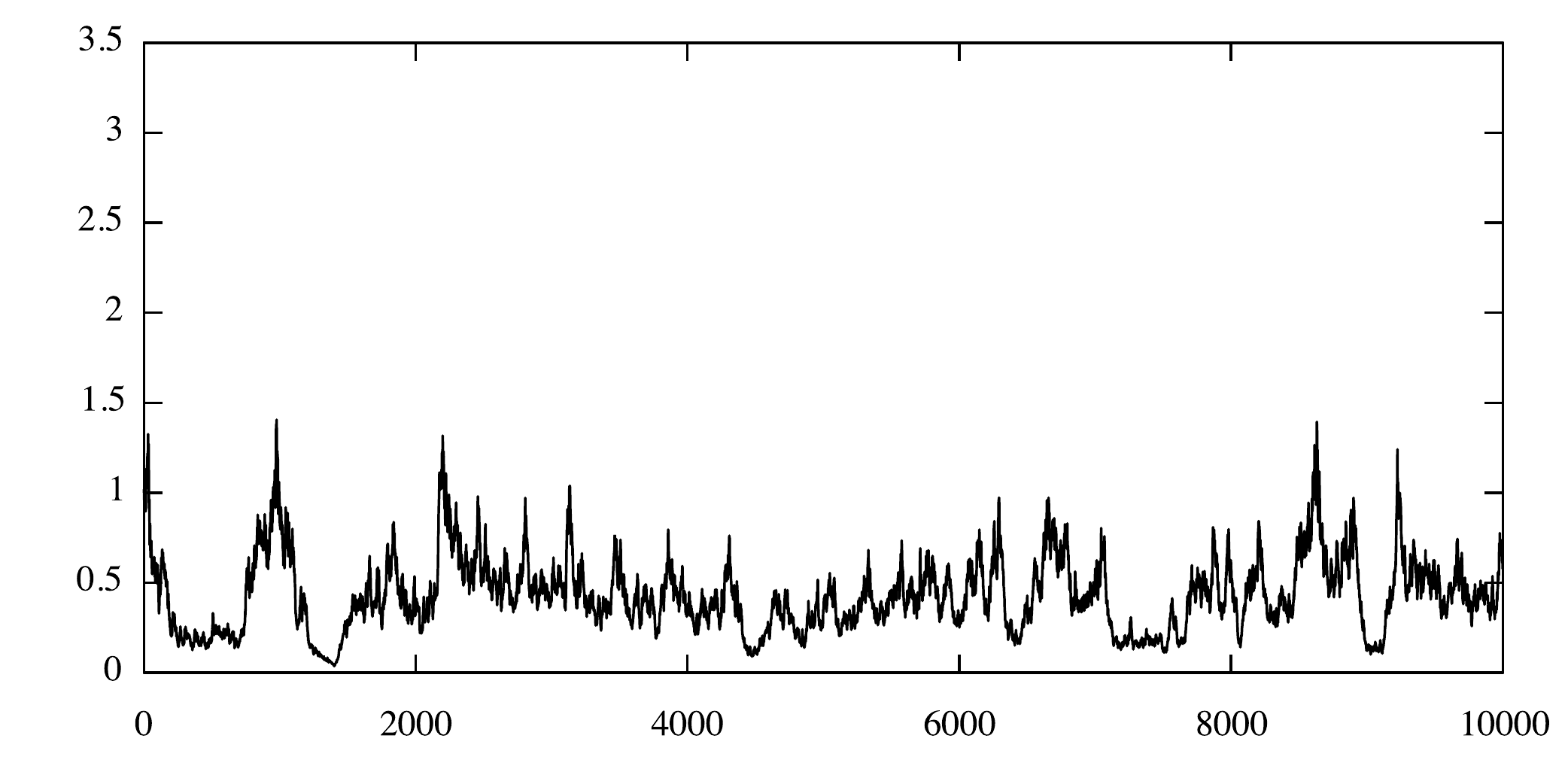}
          \includegraphics[type=pdf, ext=.pdf, read=.pdf, width=0.32\columnwidth]{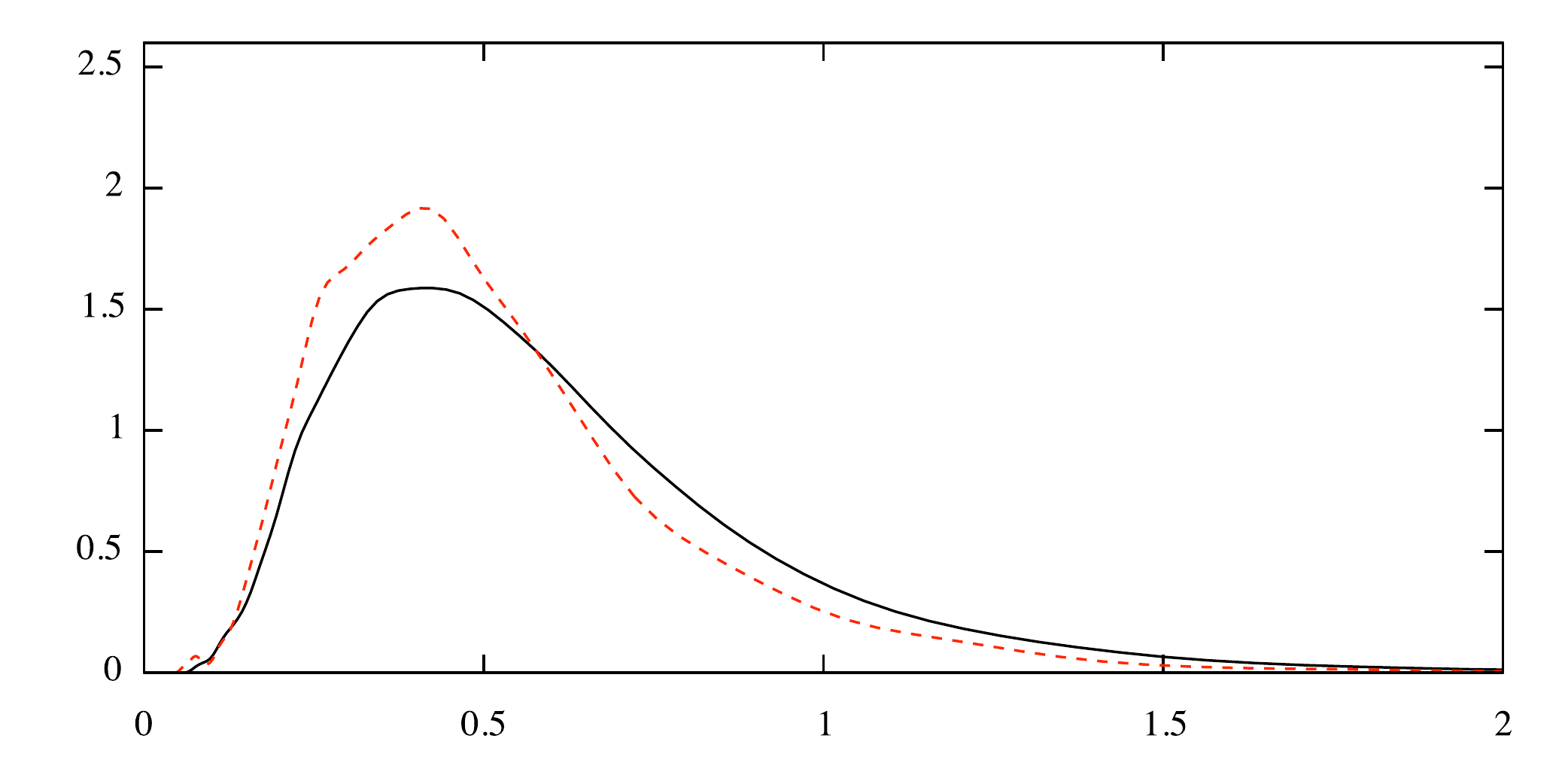}               
            \includegraphics[type=pdf, ext=.pdf, read=.pdf, width=0.32\columnwidth]{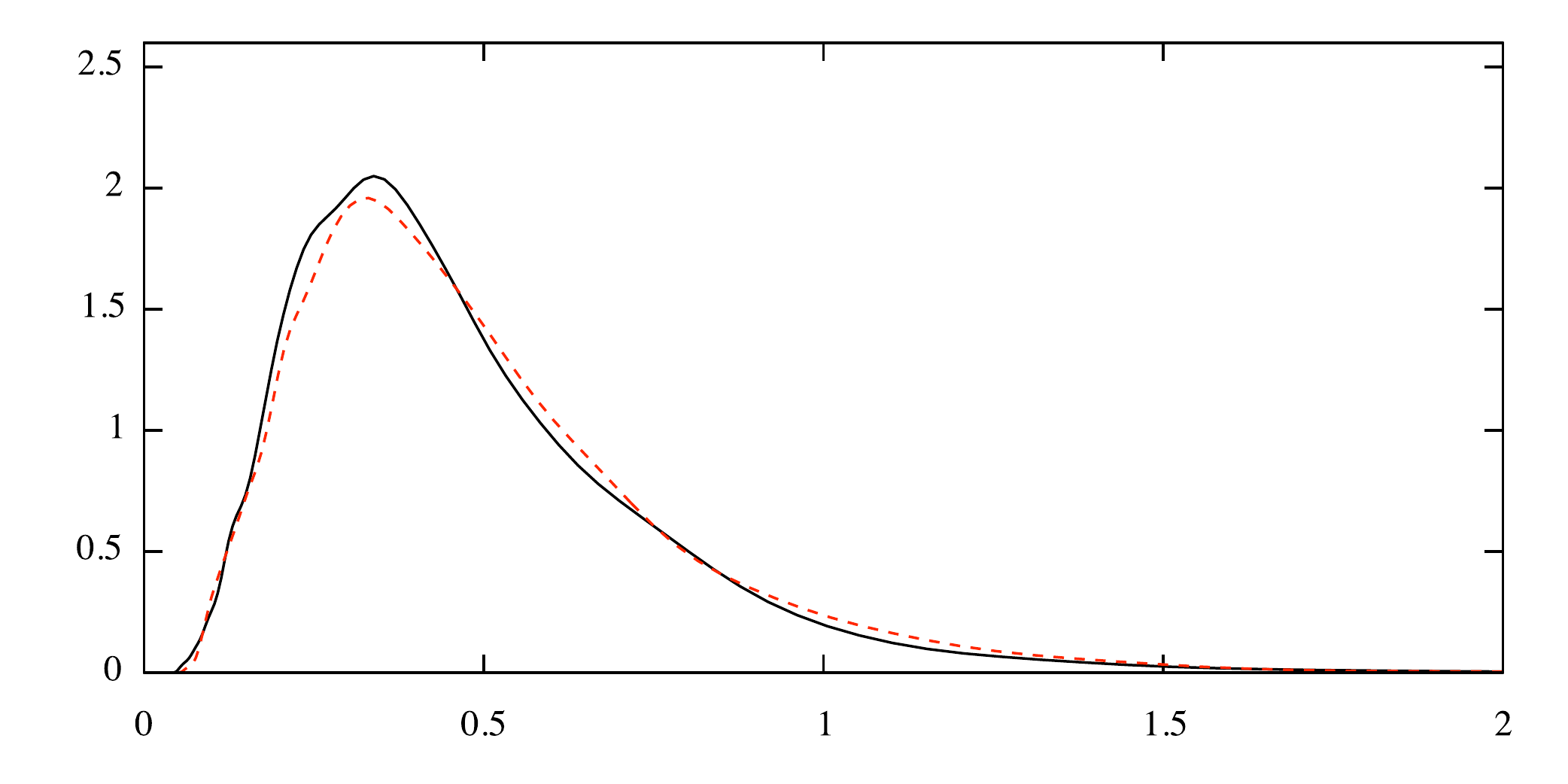}            
            \includegraphics[type=pdf, ext=.pdf, read=.pdf, width=0.32\columnwidth]{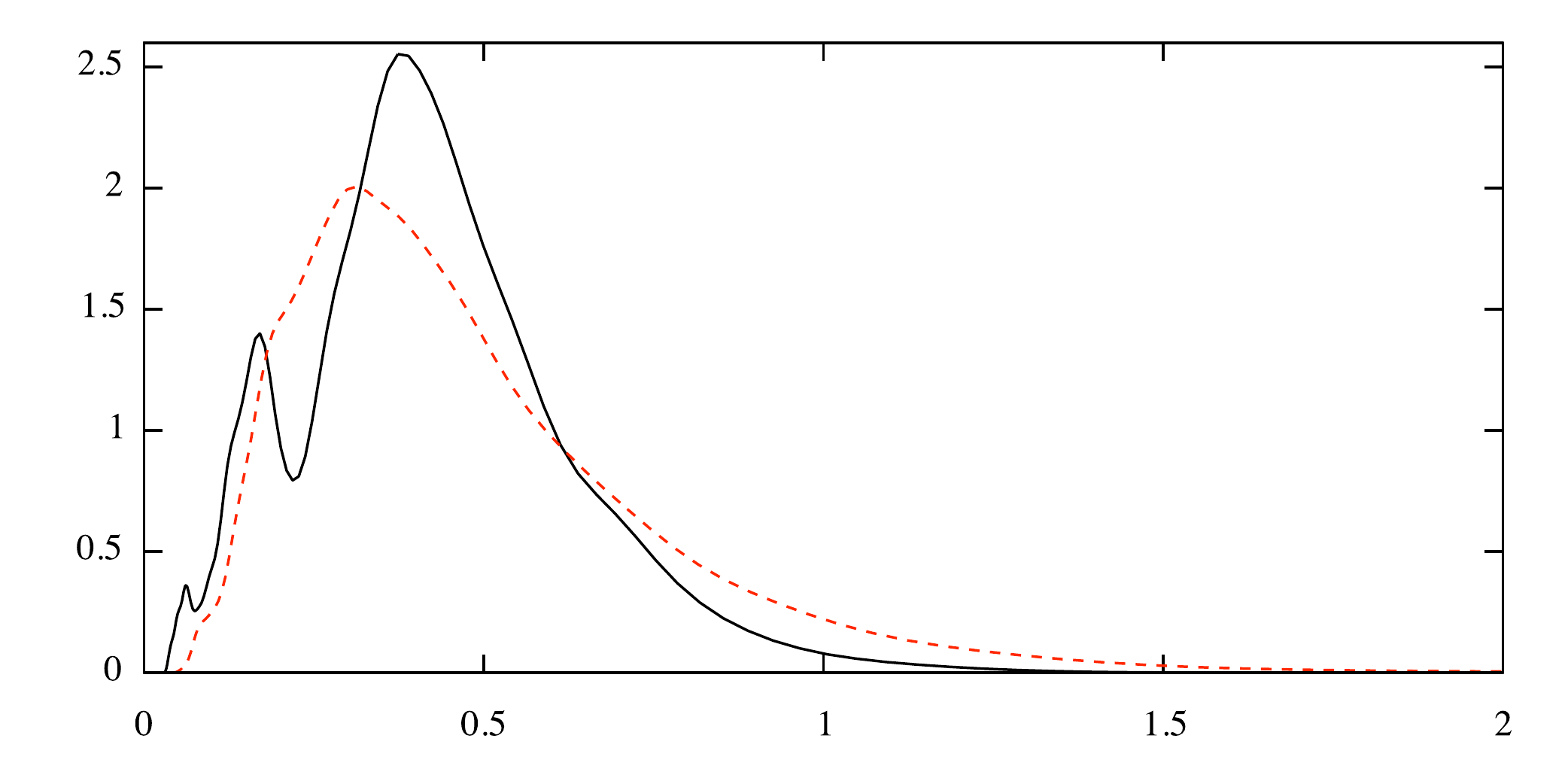}} 
            \caption{CA: $\delta$-chains (top)  and kernel density estimates of the posterior on $\delta$ (bottom) for dimensions $N=15,  127$ and $1023$ left to right. In dashed red in the density plots is the density estimate using MA, considered as a gold standard.} 
              \label{ch3:fig7}
\end{figure}

 \begin{figure}[htp]
           \center{
            \includegraphics[type=pdf, ext=.pdf, read=.pdf, width=0.32\columnwidth]{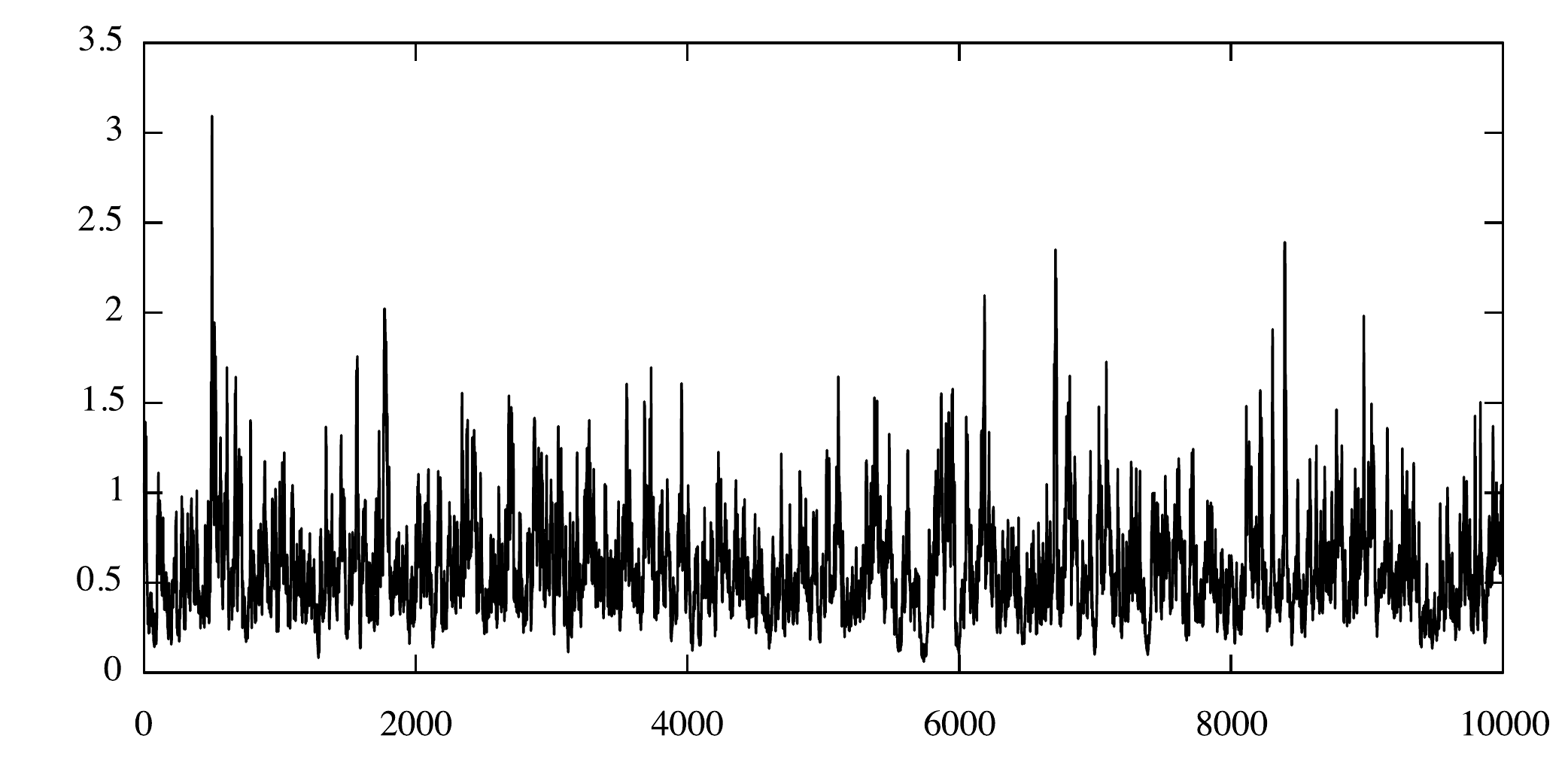}            
                        \includegraphics[type=pdf, ext=.pdf, read=.pdf, width=0.32\columnwidth]{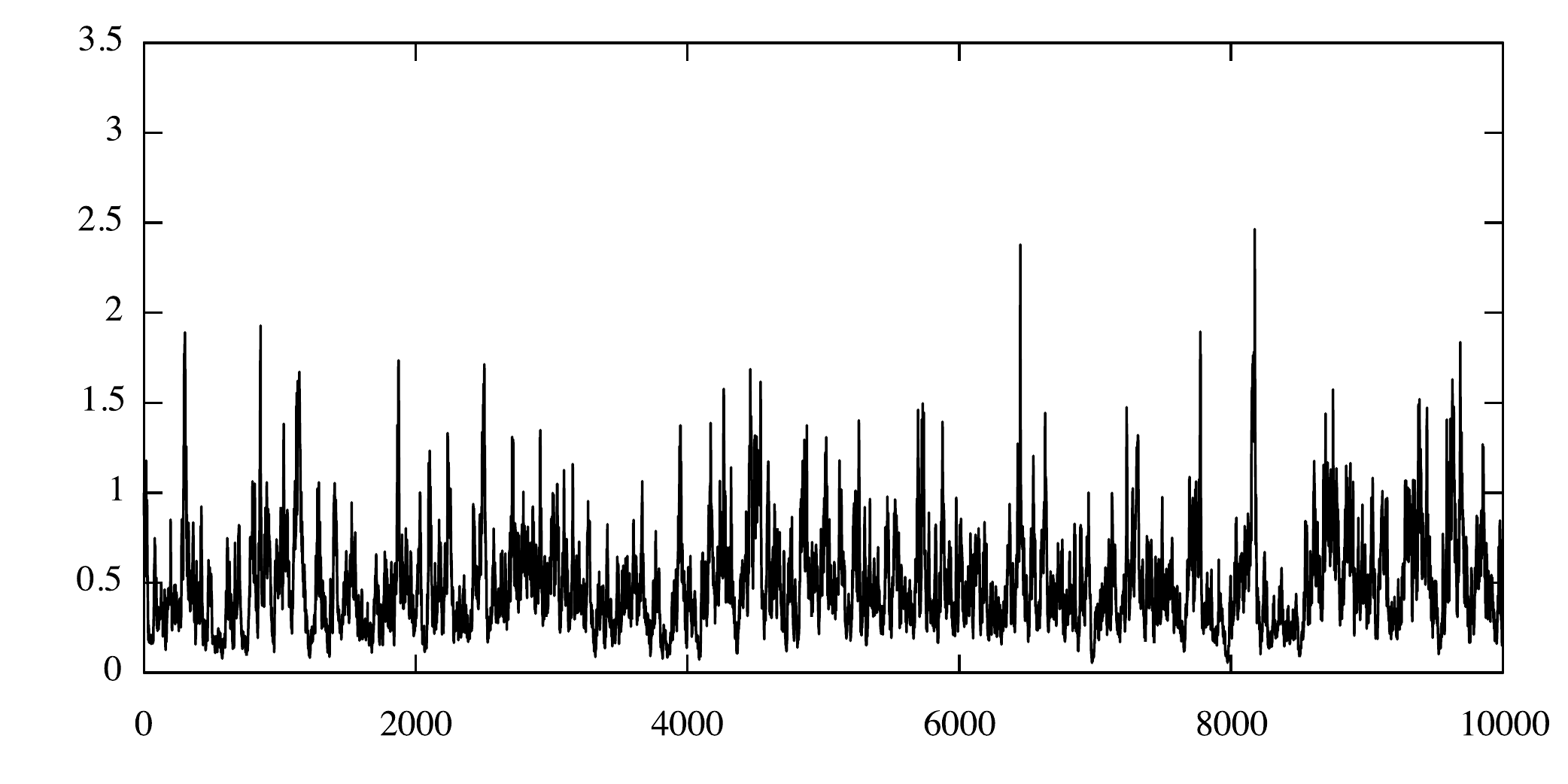}                    
            \includegraphics[type=pdf, ext=.pdf, read=.pdf, width=0.32\columnwidth]{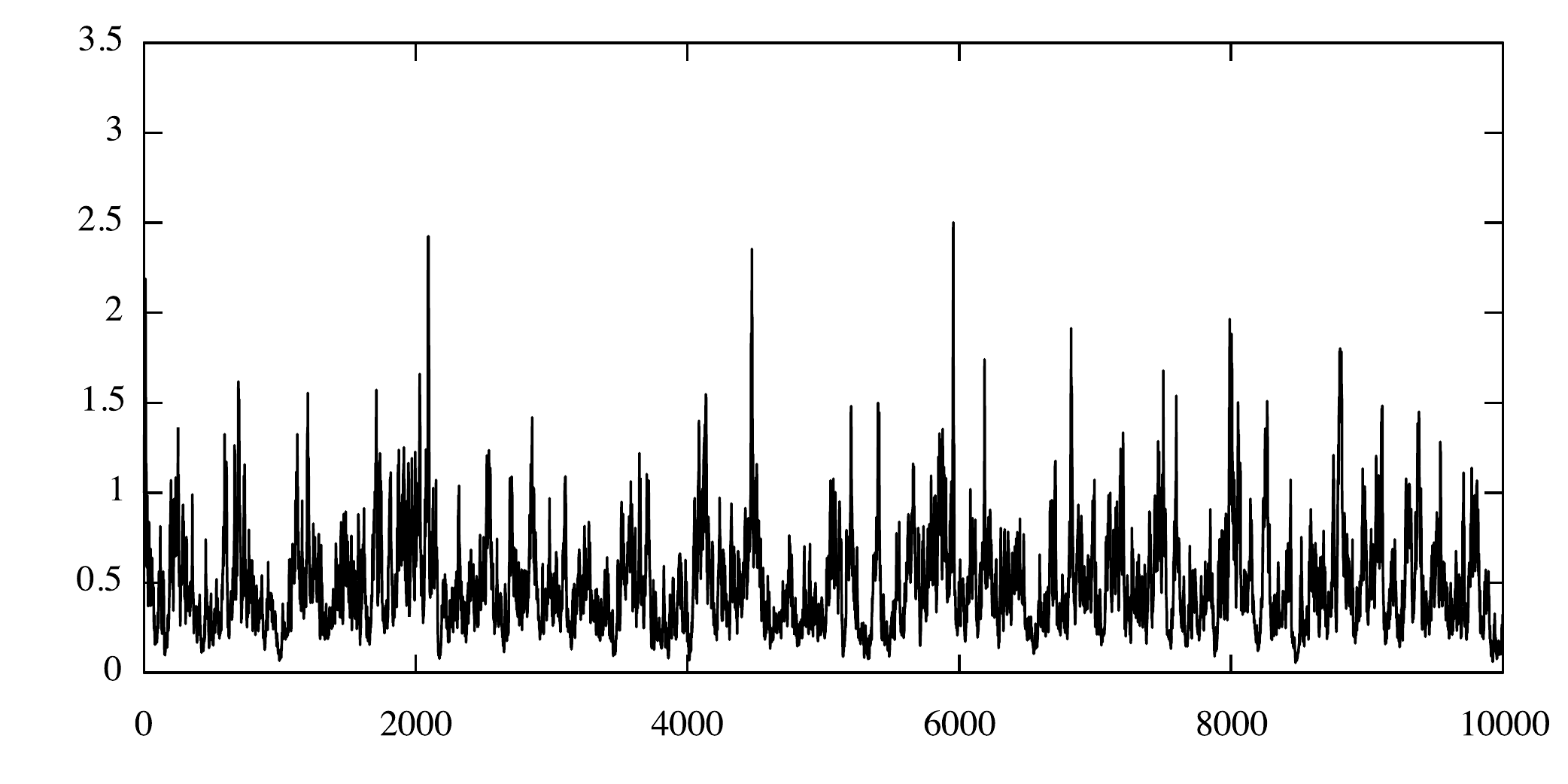}
            \includegraphics[type=pdf, ext=.pdf, read=.pdf, width=0.32\columnwidth]{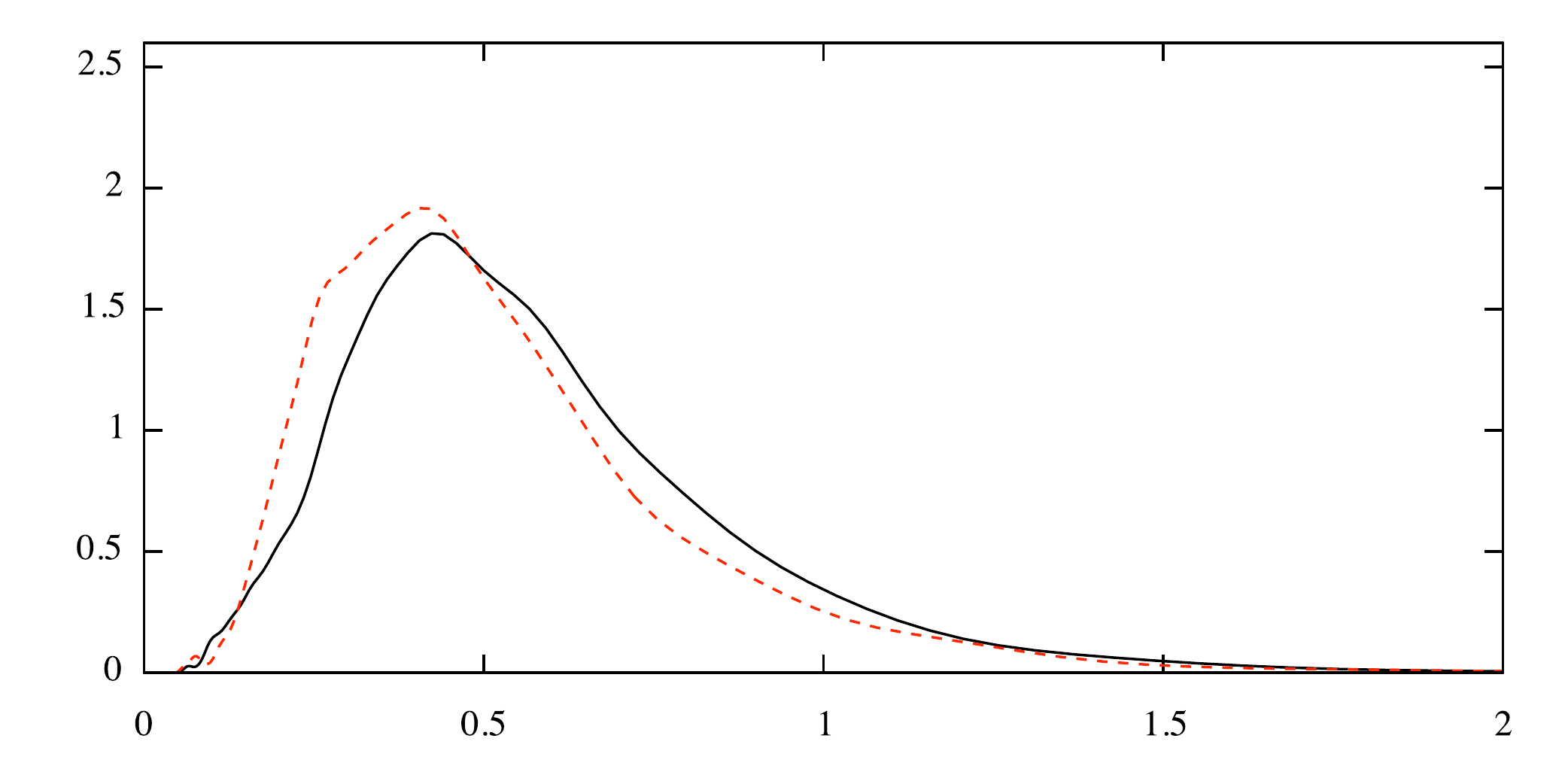}
            \includegraphics[type=pdf, ext=.pdf, read=.pdf, width=0.32\columnwidth]{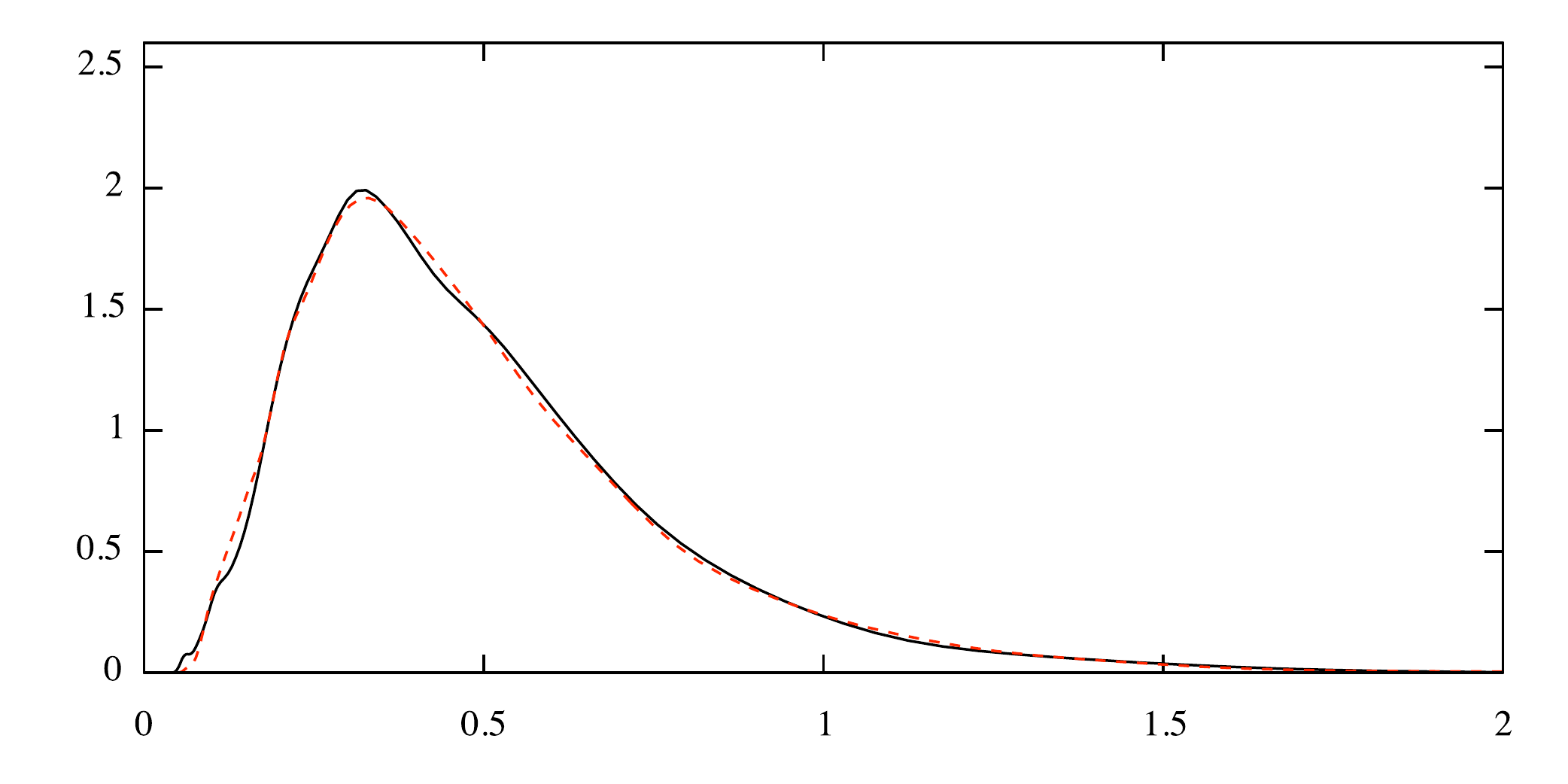}  
            \includegraphics[type=pdf, ext=.pdf, read=.pdf, width=0.32\columnwidth]{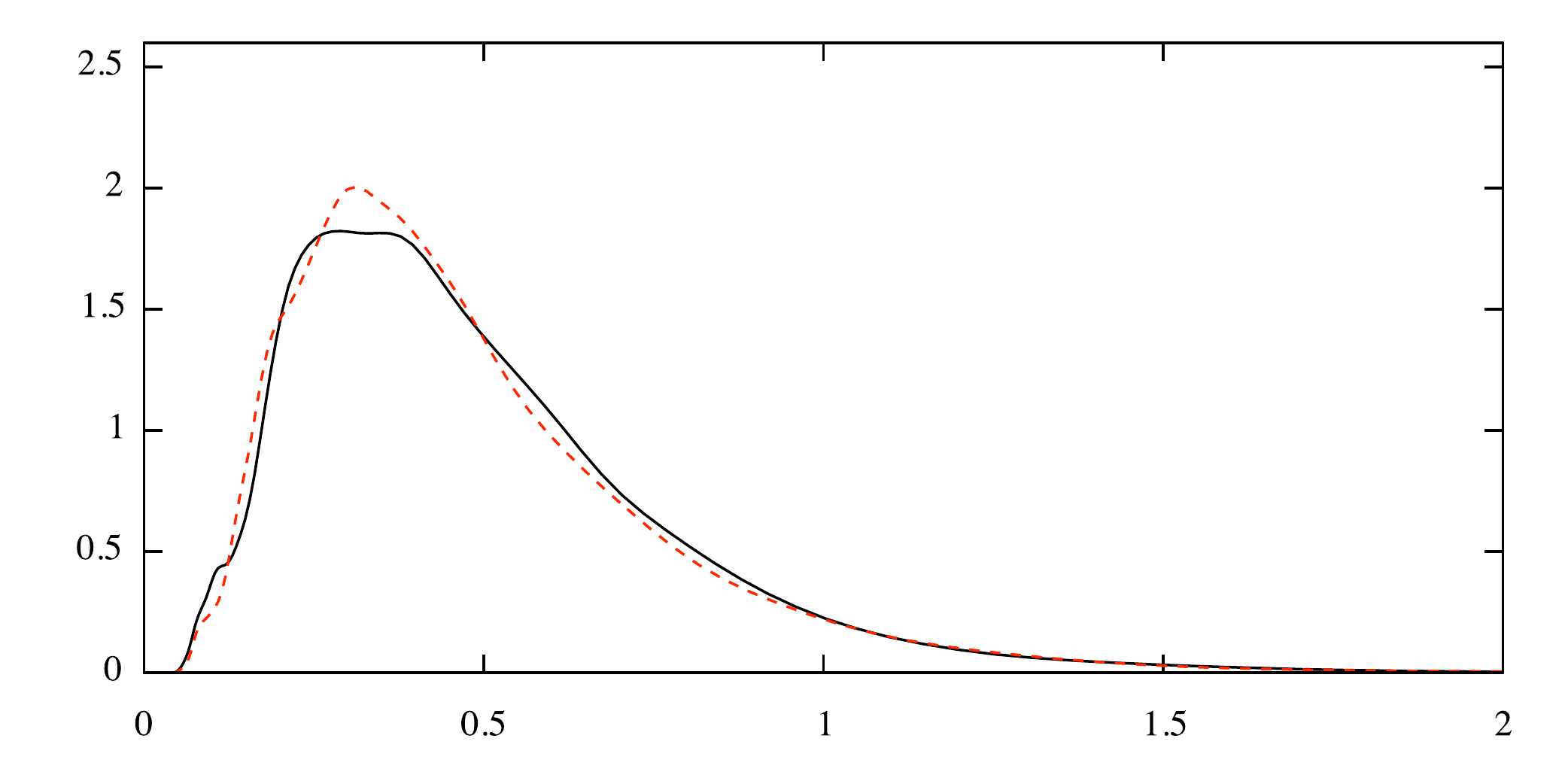}}
            \caption{NCA: $\delta$-chains (top)  and kernel density estimates of the posterior on $\delta$ (bottom) for dimensions $N=15,  127$ and $1023$ left to right. {In dashed red in the density plots is the density estimate using MA, considered as a gold standard.}} 
              \label{ch3:fig8}
\end{figure}

Our observations in Figures \ref{ch3:fig7} and \ref{ch3:fig8} are supported by the autocorrelation plots presented in Figure \ref{ch3:fig9}. The rate of decay of correlations in the $\delta$-chain in CA  appears to decrease as the dimension increases, and in particular for large $N$ the correlations seem to decay very slowly. On the contrary, the rate of decay of correlations in the $\delta$-chain in NCA appears not to be affected by the increase in dimension and is relatively close to the one in MA.

\begin{figure}[htp]
           \center{ \includegraphics[type=pdf, ext=.pdf, read=.pdf, width=0.3\columnwidth]{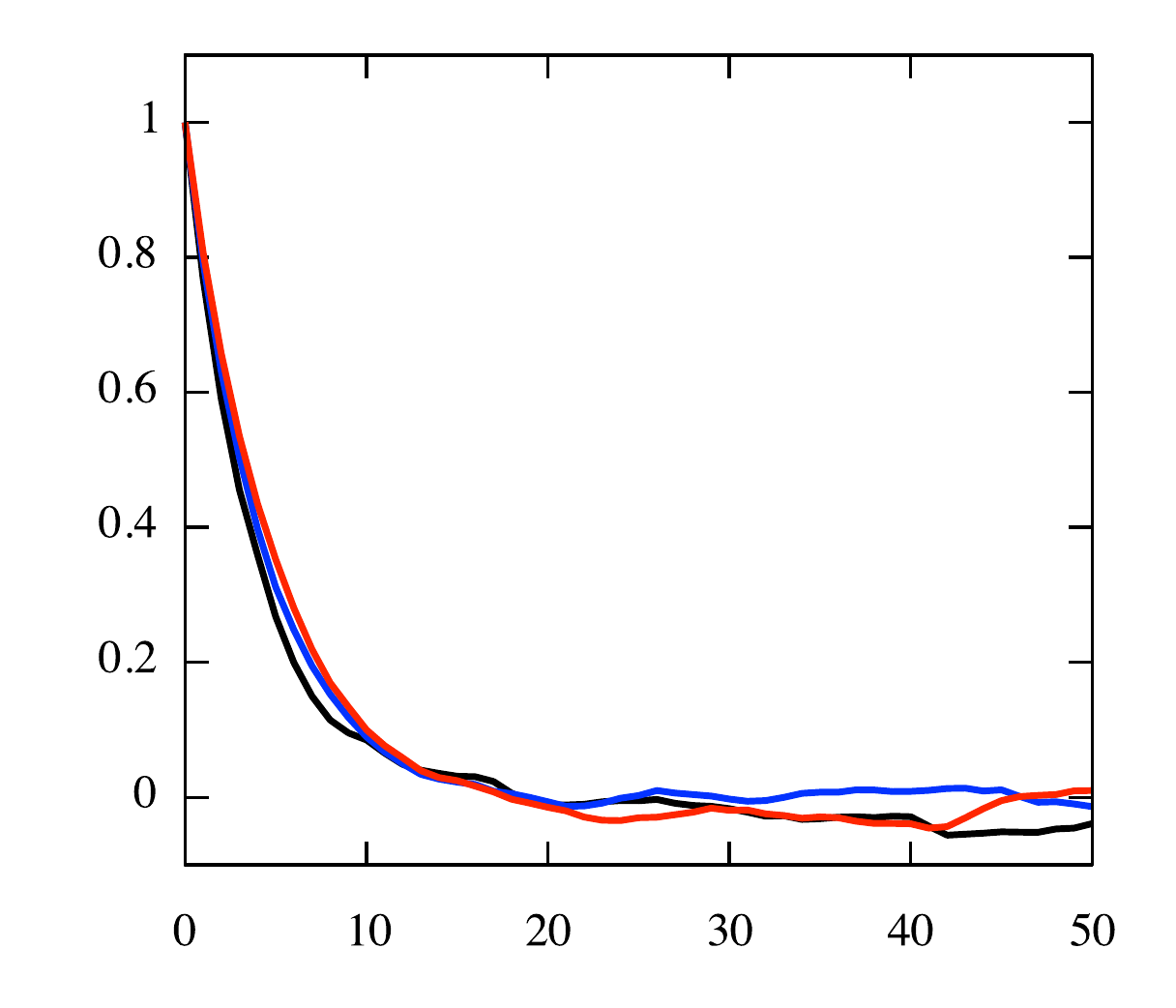}
            \includegraphics[type=pdf, ext=.pdf, read=.pdf, width=0.3\columnwidth]{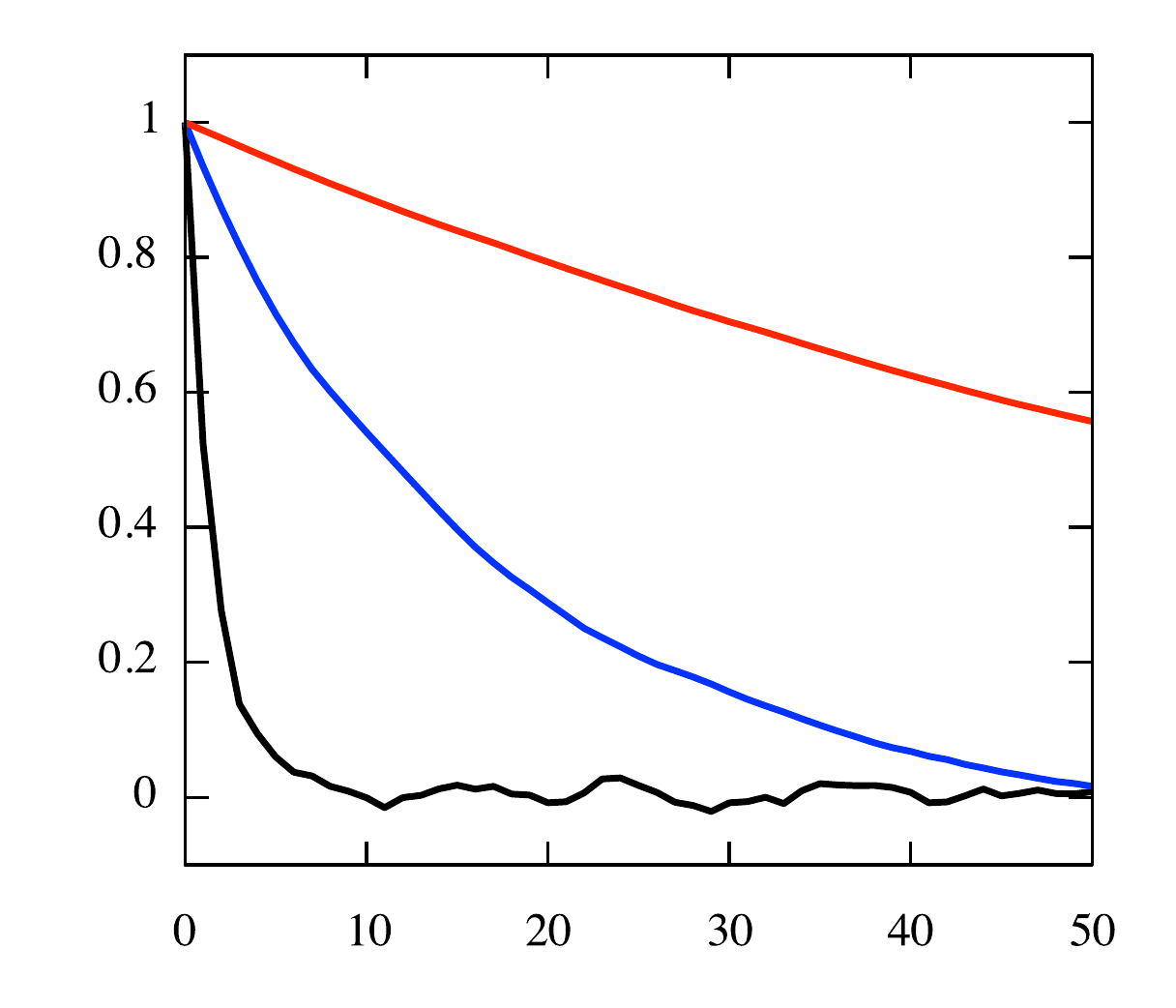}
             \includegraphics[type=pdf, ext=.pdf, read=.pdf, width=0.3\columnwidth]{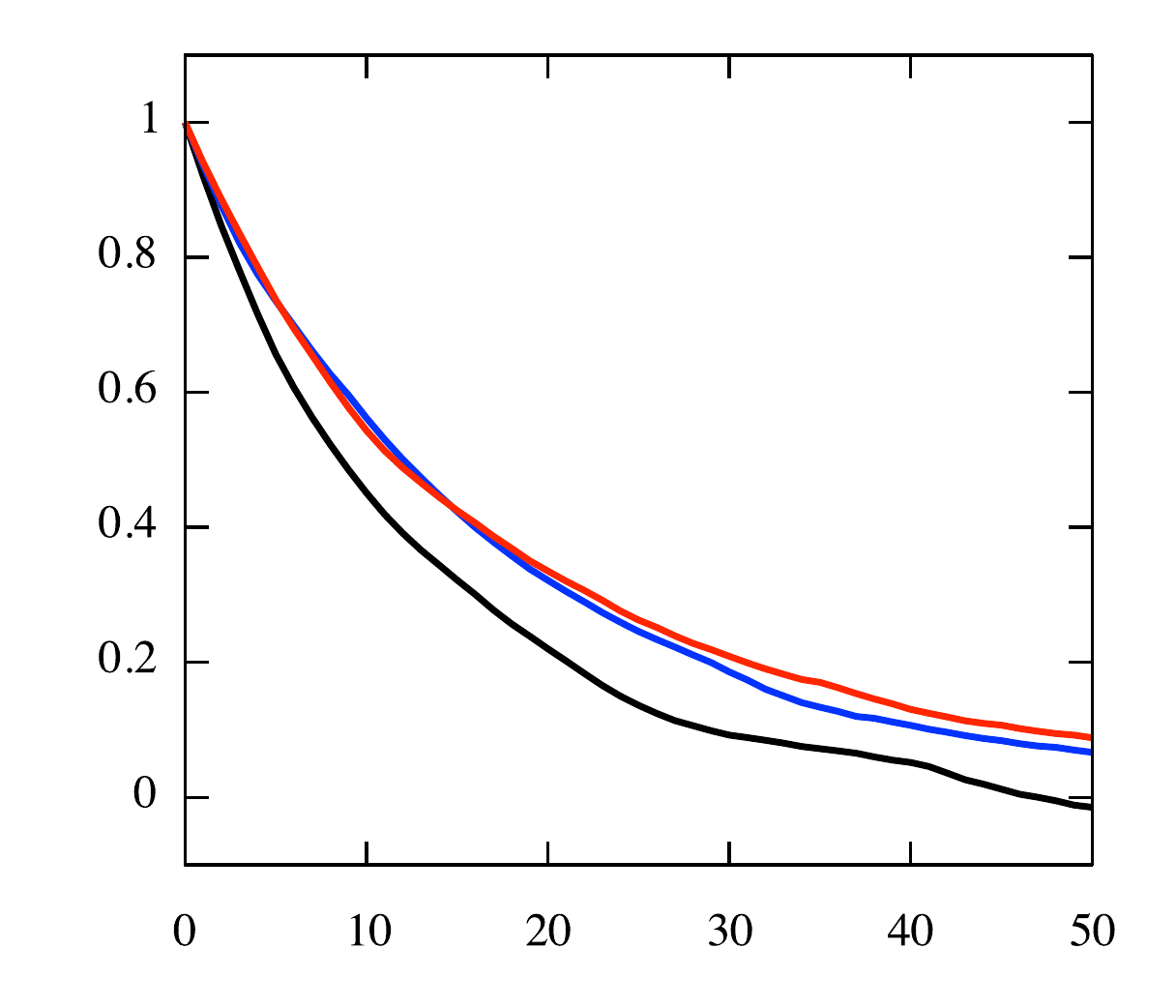}
           
           }   
            \caption{Autocorrelation functions of $\delta$-chain, dimensions 15 (black), 127 (red) and 1023 (blue); {left for MA, middle for CA, right for NCA}.}
            \label{ch3:fig9}       
\end{figure}


\section{Conclusions}\label{ch3:sec:con}

We considered a hierarchical Bayesian approach to the function-space general inverse problem (\ref{ch3:eq:1}), with Gaussian priors on the unknown function $\bu$ which depend on a variance-scaling parameter $\delta$ also endowed with a prior. We studied the finite-dimensional implementation of this setup and in particular, examined the mixing properties of MwG algorithms for sampling the posterior, as the discretization level $N$ of the unknown increases. We provided measure-theoretic intuition suggesting that under natural assumptions on the underlying function space model, as $N$ increases, CA, which is the most natural algorithm in this setting, deteriorates (see section \ref{ch3:sec:int}). We then used this intuition to propose a reparametrization of the prior for which the resultant algorithm, NCA, is expected to be robust with respect to $N$. In the linear-conjugate setting we formulated rigorous theory which quantifies the deterioration of CA in the asymptotic regime of large $N$ (see section \ref{ch3:sec:main}). 

This theory holds under assumptions on the discrete level (Assumptions \ref{ch3:ass1}) which we expect to be inherited from our assumptions on the function-space model (Assumptions \ref{ch3:infass1}) when consistent discretizations are used.
Indeed, we provided three families of linear inverse problems satisfying our assumptions on the underlying infinite-dimensional model (section \ref{ch3:sec:ex}), and for two of them, which are families of mildly and severely ill-posed problems in a simultaneously diagonal setting, we also showed that a spectral truncation method based on the common eigenbasis satisfies our discrete level assumptions (subsections \ref{ch3:ssec:diag} and \ref{ch3:ssec:sev}). It would be interesting to show that discretization via finite differences of these examples also satisfies our discrete assumptions.

Our numerical results confirmed our theory on the deterioration of CA as well as our intuition about the robustness of NCA in the large $N$ limit. However, for NCA the $\delta$-chain slows down in the small noise limit. This is
 because even though $v$ and $\delta$ are a priori independent, they both
 need to explain the data, and this creates an increasingly severer
 constraint as $\hl$ becomes large. Hence, $\delta$ and $v$ concentrate
 near a lower dimensional manifold, where $\delta^{-\frac12} K v \approx y$, and the
 Gibbs sampler mixes poorly (see Figure \ref{ch3:fig5} for a numerical illustration of this effect in the example of subsection \ref{ch3:nex1}). Although MA is robust in both the large $N$ and the small noise limit, it can be prohibitively expensive for large scale inverse problems; new work is required to produce effective hierarchical algorithms in this small noise limit, when $N$ is large.
We have considered  the interweaving method of \cite{YM11},
which combines in each iteration centered and non-centered draws of
$\delta$,   and the  partially non-centered parametrizations of \cite{PRS03}, in which the 
prior is reparametrized as $u=\delta^{-\frac{t}2}v_t$ where
$v_t\sim\G(0,\delta^{t-1}\bC_0)$, for some $t\in[0,1]$. Our numerical
experimentation did not suggest significant benefits from their use,
hence we do not report them here, but further investigation of these issues would be of interest. 

\begin{figure}[htp]
           \center{           
            \includegraphics[type=pdf, ext=.pdf, read=.pdf, width=0.32\columnwidth]{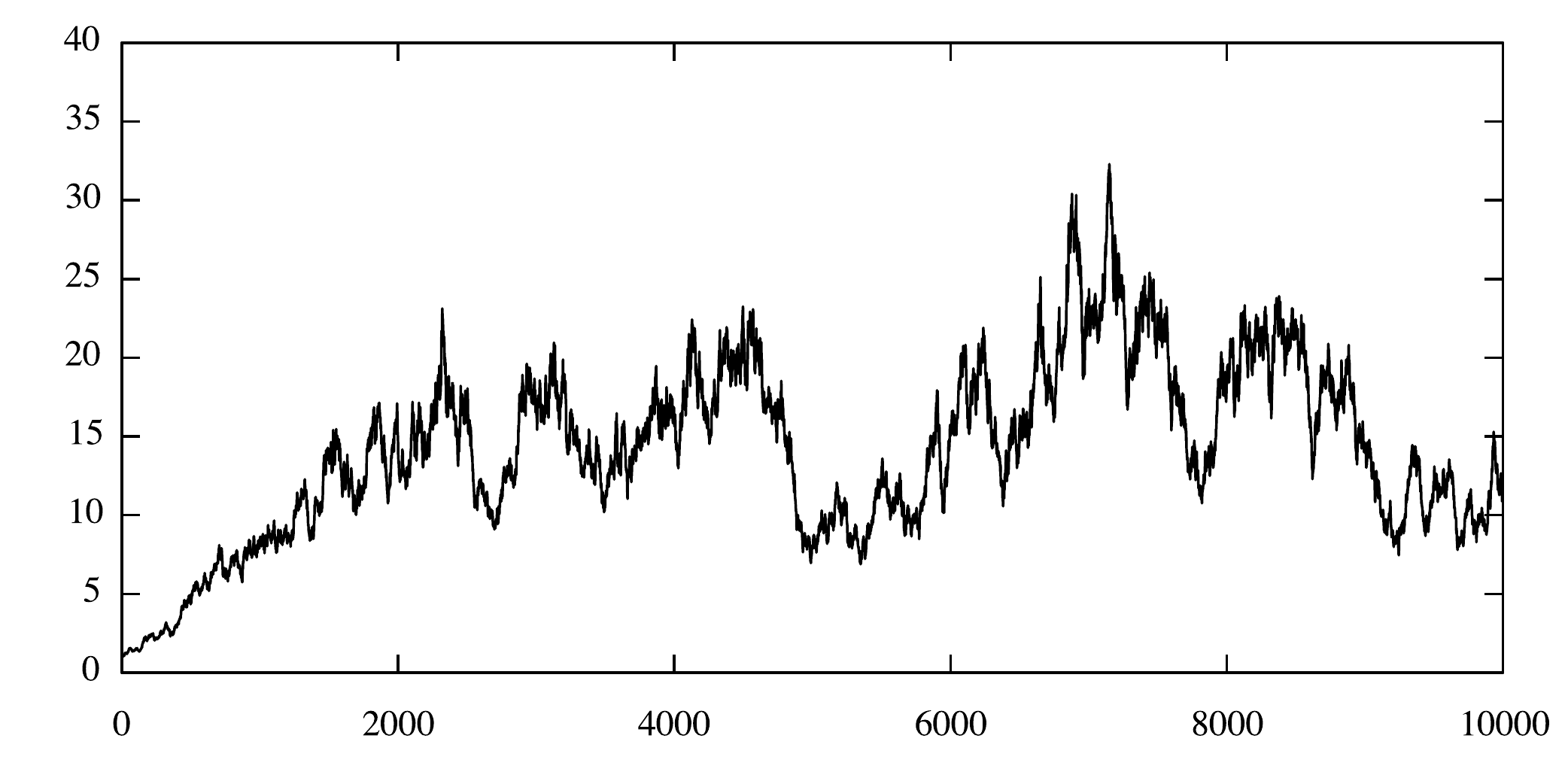}            
            \includegraphics[type=pdf, ext=.pdf, read=.pdf, width=0.32\columnwidth]{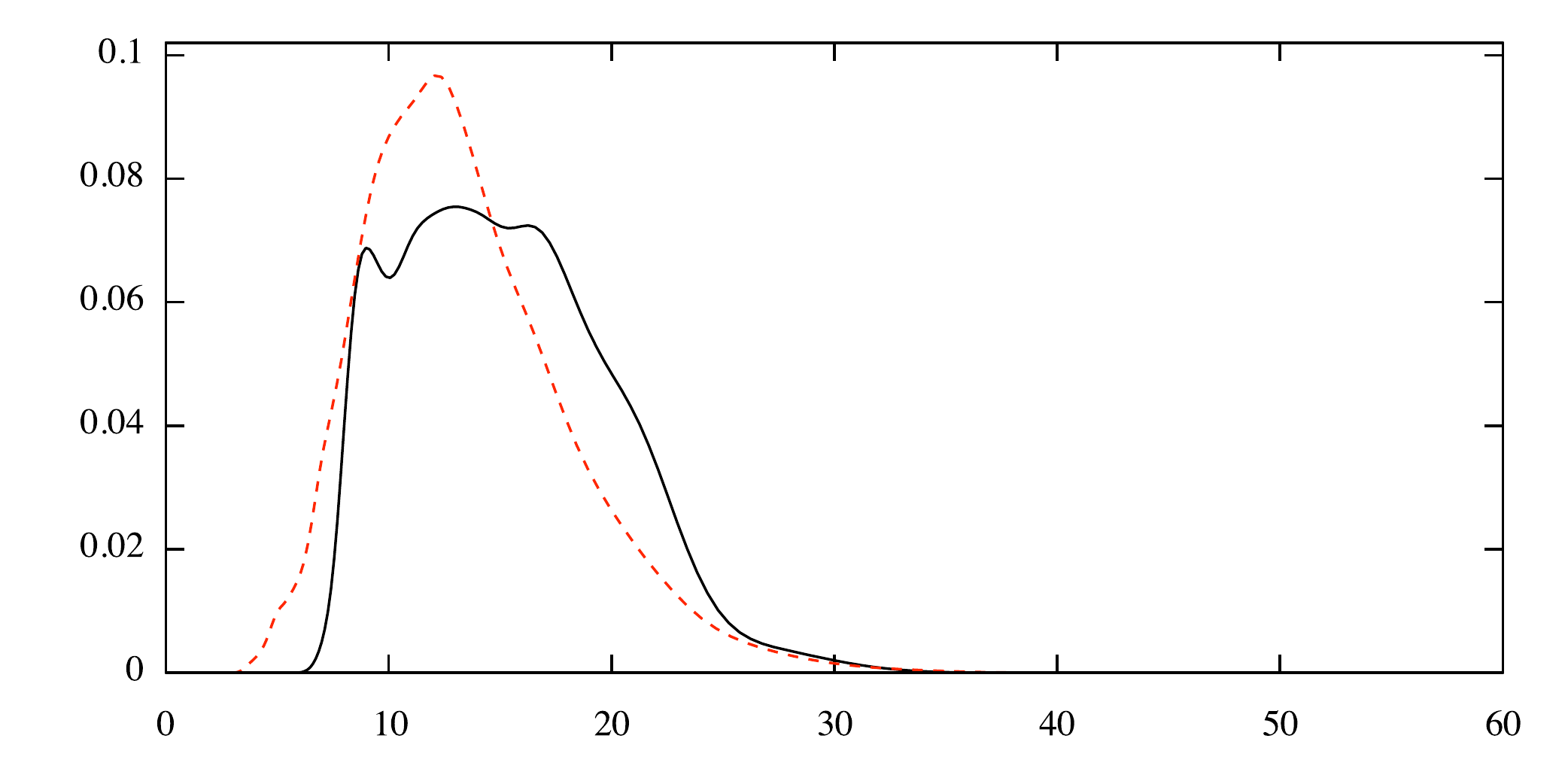}}          
            \caption{Signal in white noise model - NCA for small noise, $\hl=200^2$, and dimension $N=512$: $\delta$-chain (left) and kernel density estimate of posterior on $\delta$ (right, black). In dashed red in right plot is the density estimate using MA, considered as a gold standard.}    
            \label{ch3:fig5}
\end{figure}

In addition to \cite{JB12}, a similar hierarchical setup has been considered in \cite{SVZ13} in the \emph{signal in white noise} model (see subsection \ref{ch3:nex1}). The authors of \cite{SVZ13} study a different aspect of the problem, namely the asymptotic performance of the posterior distribution in the small noise limit. This is motivated by results on posterior consistency suggesting that the optimal rates of contraction are achieved by rescaling the prior depending on the size of the noise \cite{KVZ12, ALS13, KVZ13, ASZ12}. They also study an empirical Bayes method for estimating the value of the prior scaling from the data and show that both methods achieve optimal posterior contraction rates over a range of regularity classes of the true solution. However, we have seen in this paper that the implementation of the hierarchical Bayesian method in the large dimensional limit is problematic. On the other hand, while the empirical Bayes method is appealing because of the lack of mixing issues, it involves solving an optimization problem which in more complicated models can be computationally demanding, and it does not provide uncertainty quantification of the prior scaling which may be desirable. Again we highlight the need for more research and new ideas  in the small noise limit, when $N$ is large.

An asymptotic regime which we have not investigated yet, is the case where we have a sequence of $N$-dimensional linear inverse problems, with the relevant matrices being consistent discretizations of linear operators and where the size of the noise decreases as $N$ grows larger, that is $\lambda=\lambda(N)\to\infty$ as $N\to\infty$. This is the limit of an infinite dimensional unknown which is also identifiable from the data. Since in this regime, as $N$ grows larger the supports of both $\delta|y,u$ and $\delta|y$ shrink to zero, we expect that there will be an optimal relationship between $\lambda$ and $N$, for which CA will not deteriorate for large $N$.

Our theory on the slowing down of the $\delta$-chain can be extended
to cover nonlinear Gaussian-conjugate Bayesian inverse problems and in
particular the nonparametric drift estimation in SDE's setting
considered in \cite{ PPRS12, PSZ13, MSZ13}; see \cite[Chapter 4.5]{SA13}. Again the main result holds under assumptions on the discrete level which we expect to be inherited by consistent discretizations from natural assumptions on the underlying infinite-dimensional model which express that the posterior is absolutely continuous with respect to the prior.

Furthermore, our infinite-dimensional intuition extends to
hierarchical setups for inference on other hyper-parameters, for
instance the prior regularity parameter $\alpha$, where
$\C_0=\bA_0^{-\alpha}$, as studied in \cite{KSVZ12}. In Figure
\ref{ch3:fig11} we plot autocorrelation functions for the centered MwG
algorithm used in this setting and the corresponding version of
the non-centered algorithm;  as before we also
implemented the corresponding marginal algorithm and use it as the
gold standard. The underlying truth, the noise distribution and the
discretization method are the same as in subsection \ref{ch3:nex1} and we
use an exponential hyper-prior on $\alpha$. The idea is the same as the intuition presented in
section \ref{ch3:sec:int}, since in infinite dimensions two Gaussian
measures $\G(0,\Sigma_1)$ and $\G(0,\Sigma_2),$ where $\Sigma_1$ and
$\Sigma_2$ are simultaneously diagonalizable with eigenvalues
$\{j^{-{\alpha_1}}\}_{j\in\N}$ and $\{j^{-{\alpha_2}}\}_{j\in\N}$
respectively,  are mutually singular unless
$\alpha_1=\alpha_2$. Indeed, our numerical simulations confirm again
the deterioration of the centered algorithm and the robustness of the
non-centered algorithm, in the large $N$ limit. More generally, as
suggested in section \ref{ch3:sec:int}, our intuition holds for
inference on any prior on $\bu$ which depends on a  hyper-parameter
$\theta$, when it holds that $\bu|\by,\theta$ is absolutely continuous
with respect to $\bu|\theta$ almost surely with respect to the data,
while $\bu|\theta$ and $\bu|\theta'$ are mutually singular when
$\theta\neq\theta'$. 

\begin{figure}[htp]
           \center{ 
           \includegraphics[type=pdf, ext=.pdf, read=.pdf, width=0.3\columnwidth]{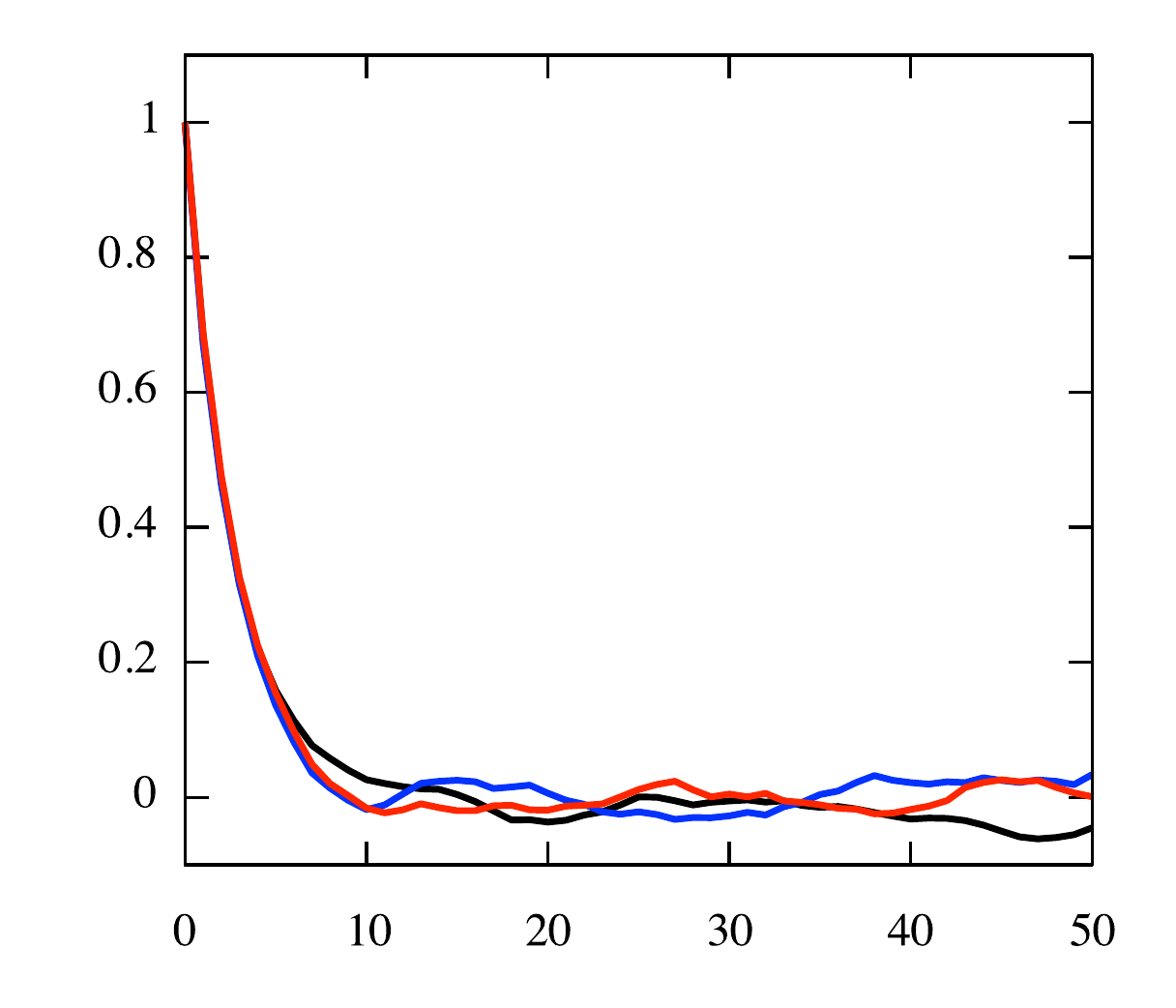}
            \includegraphics[type=pdf, ext=.pdf, read=.pdf, width=0.3\columnwidth]{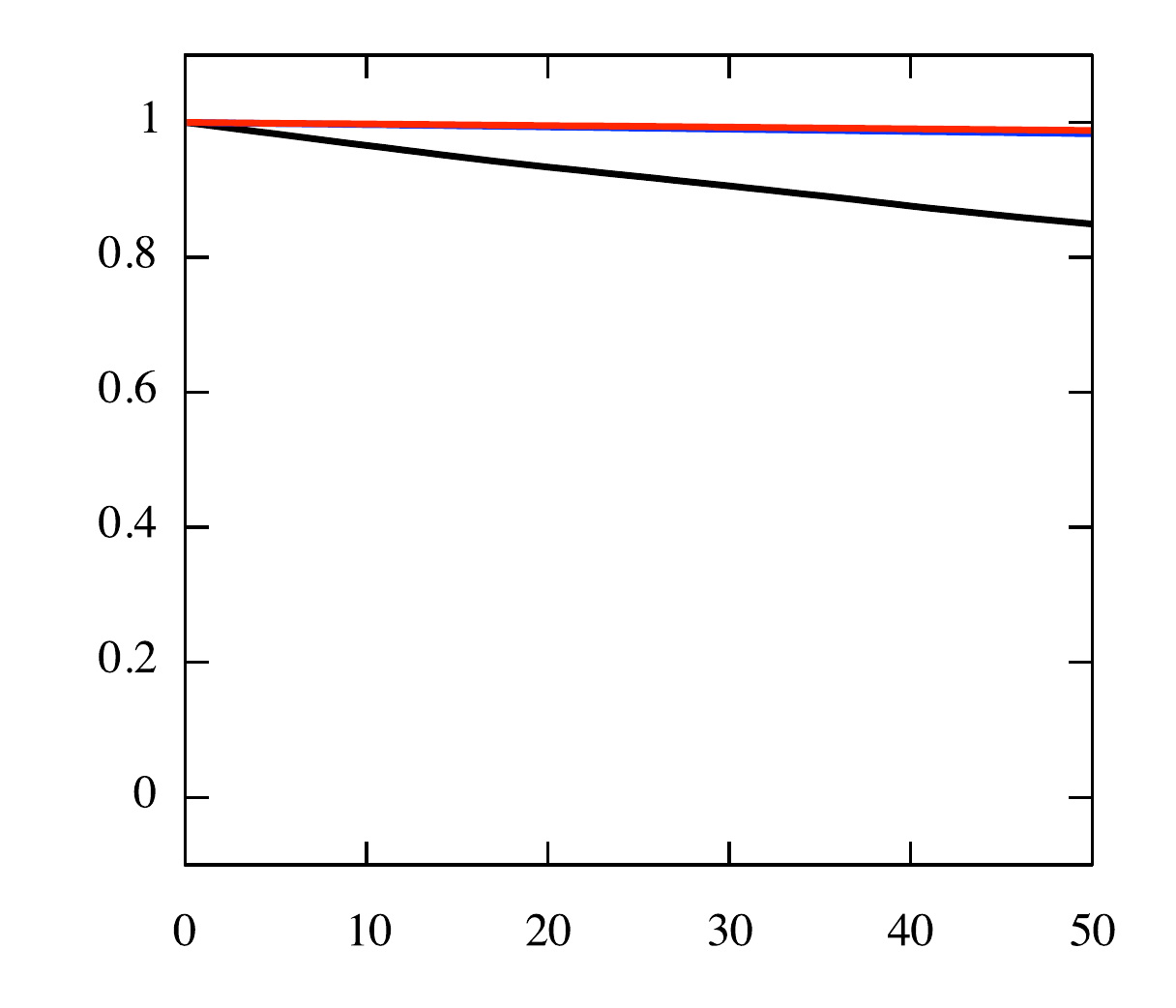}
             \includegraphics[type=pdf, ext=.pdf, read=.pdf, width=0.3\columnwidth]{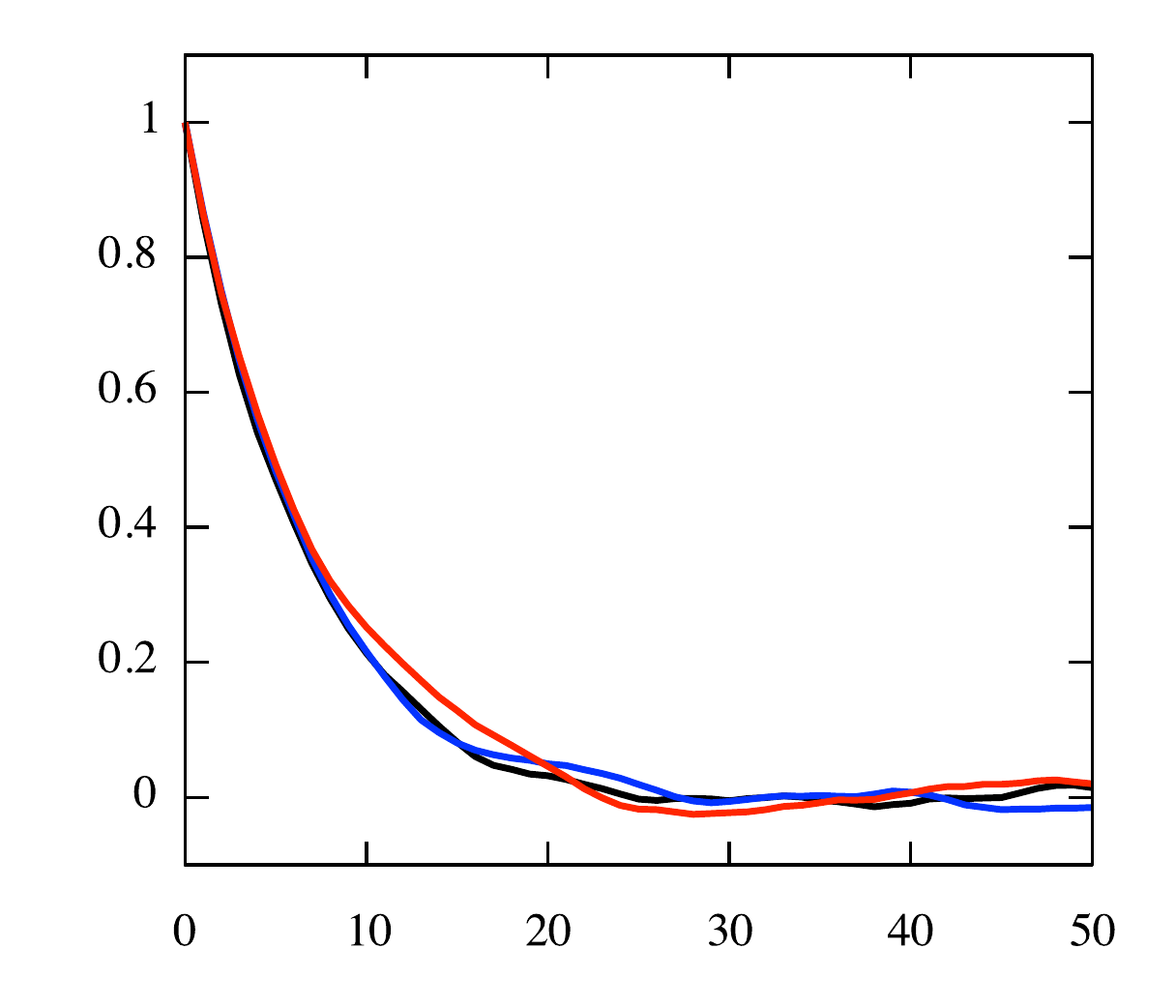}}   
            \caption{Autocorrelation functions of $\alpha$-chain, dimensions 32 (black), 512 (red) and 8192 (blue); left for marginal, middle for centered, right for non-centered.}
            \label{ch3:fig11}       
\end{figure}
Returning to the general nonlinear setting discussed in section
\ref{ch3:sec:int}, we note that both Algorithms \ref{ch3:algstd} and
\ref{ch3:algrep} are straightforward to generalize, however with a
certain loss of efficiency compared to the linear-conjugate
setting. The distribution of $\bu|\by,\delta$ no longer belongs to a known
parametric family of distributions, and thus has to be sampled
using a Metropolis-Hastings (for example one based on Langevin diffusion) step. 
Moreover, for nonlinear inverse problems there is no longer an easy way of
finding the marginal distribution  $\by|\delta$, hence MA will not be
an option. The so-called  \emph{pseudo-marginal algorithm}
\cite{AR09}, might be an alternative for non-linear problems, and has
been recently   employed to perform Bayesian inference using Gaussian
process priors in \cite{FG13}. An interesting research direction is
the comparison of the performance of the two MwG algorithms with the
pseudo-marginal algorithm in both the large $N$ and the small noise
limits. 

Finally, our research agenda includes the extension to the
hierarchical setting of the present paper, of the analysis contained
in \cite{CDS10} of the bias in the estimated posterior distribution
due to the discretization of the unknown and forward problem.

\section{Appendix}\label{ch3:sec:ap}
In this section we present the proof of Theorem \ref{ch3:thm1}, as well as several technical results and lemmas. Subsection \ref{ch3:ssec:dproof} contains the proof of Theorem \ref{ch3:thm1}, the backbone of which is Lemma \ref{ch3:lem1} proved in subsection \ref{ch3:ssec:ap1}. In subsection \ref{ch3:ssec:ap3} we state and prove a lemma on the negative moments of the rate parameter in the $\delta$ draw
(\ref{ch3:eq:dd}), which allows us to control the lower order terms arising in the proof of Theorem \ref{ch3:thm1}. Finally, in subsection \ref{ch3:ssec:ap4}, we prove several  probability and linear algebra lemmas, which are useful in our analysis.
\subsection{Proof of Theorem \ref{ch3:thm1}}\label{ch3:ssec:dproof}
We now prove Theorem \ref{ch3:thm1} under Assumptions \ref{ch3:ass1}.
Using the scaling property of the gamma distribution, $\Ga(\upalpha,\upbeta)\stackrel{\mathcal{L}}{=}
\upbeta^{-1}\Ga(\upalpha,1)$, and multiplying and dividing by $\frac{2}N\delta$, we can write the $\de{k+1}
_N$ draw in (\ref{ch3:eq:dd}) as \begin{align}\label{ch3:eq:dd1}\de{k+1}_N&\stackrel{\mathcal{L}}{=}\delta
\frac{\Gamma_{0,N}}{\frac2N\delta(\bd+\frac12\norm{\C_0^{-\frac12}u_{\delta}^{(k)}}^2_{\R^N})}
\end{align}
where $\Gamma_{0,N}\sim\Ga(\ad+\frac{N}2,\frac{N}2)$ is independent of ${\yn}$ and $u_{\delta}^{(k)}$.

Defining $
W_{2,N}:=\frac{\Gamma_{0,N}-1-\frac{2\ad}N}{\sqrt{\frac{2}N+\frac{4\ad}{N^2}}}$, we have 
\begin{align*}\Gamma_{0,N}=1+\frac{2\ad}N+\sqrt{\frac{2}N+\frac{4\ad}{N^2}}W_{2,N},\end{align*} where 
for every $N$, the random variable $W_{2,N}$ has mean zero and variance one, third and fourth moments 
bounded uniformly in $N$ (see Lemma \ref{ch3:lemgam}), and is independent of the data ${\yn}
$ and $\zeta$, the Gaussian white noise expressing the fluctuation in $u^{(k)}_{\delta}$. Concatenating we get 
\begin{align}\label{ch3:eq:sp1}\de{k+1}_N\stackrel{\mathcal{L}}{=}
\delta\frac{1+\frac{2\ad}N+\sqrt{\frac{2}N+\frac{4\ad}{N^2}}W_{2,N}}{1+\sqrt{\frac2N}W_{1,N}+\frac{2}
NF_N(\delta)\delta},\end{align} and we are now ready to prove Theorem \ref{ch3:thm1}:

\begin{proof}
By the independence of $W_{2,N}$ and $\zeta$ and since $\E[W_{2,N}]=0$, we have 
\begin{align*}\E[\de{k+1}_N-\de{k}_N|\de{k}_N=\delta]&=\delta\E\left[\frac{1+\frac{2\ad}{N}+\sqrt{\frac2N+
\frac{4\ad}{N^2}}W_{2,N}}{1+\sqrt{\frac2N}W_{1,N}+\frac{2F_N\delta}N}-1\right]\\
&=\delta\E^\zeta\left[\frac{\frac{2\ad}N-\sqrt{\frac2N}W_{1,N}-\frac{2F_N\delta}N}{1+\sqrt{\frac2N}W_{1,N}+
\frac{2F_N\delta}N}\right].
\end{align*} Using the identity $\frac{1}{1+x}=1-x+\frac{x^2}{1+x}$ we get
\begin{align*}\E&[\de{k+1}_N-\de{k}_N|\de{k}_N=\delta]\\=&\delta\E^\zeta\left[\left(\frac{2(\ad-F_N\delta)}N-
\sqrt{\frac2N}W_{1,N}\right)\left(1-\sqrt{\frac2N}W_{1,N}-\frac{2F_N\delta}N\right)\right]+\E^\zeta[e_{1,N}],
\end{align*}
where \begin{align*}e_{1,N}=\delta\frac{\left(\frac{2(\ad-F_N\delta)}N-\sqrt{\frac2N}W_{1,N}\right)
\left(\frac{2W_{1,N}^2}N+\frac{4F_N^2\delta^2}{N^2}+\frac{4\sqrt{2}F_NW_{1,N}\delta}{N^\frac32}\right)}
{1+\sqrt{\frac2N}W_{1,N}+\frac{2F_N\delta}N}.\end{align*}

Using H\"older's inequality and the fact that $F_N$ and $W_{1,N}$ have moments of all positive orders 
which are bounded uniformly in $N$, we get  
\begin{align*}\E[\de{k+1}_N-\de{k}_N|\de{k}_N=\delta]&=\frac{2}N\left((\ad+1)\delta-\E^\zeta[F_N]
\delta^2\right)+\mathcal{O}(N^{-\frac32})+\E^\zeta[e_{1,N}],
\end{align*}
almost surely with respect to $\by$.
For the residual $e_{1,N}$, by Cauchy-Schwarz inequality and (\ref{ch3:eq:denom}), we have 
\begin{align*}
\E^\zeta&[e_{1,N}]=\E^\zeta\bigg[\frac{\left(\frac{2(\ad-F_N\delta)}N-\sqrt{\frac2N}W_{1,N}\right)
\left(W_{1,N}^2+\frac2{N}F_N^2\delta^2+\frac{2\sqrt{2}}{N^\frac12}F_NW_{1,N}\delta\right)}{\frac{N}
{2\delta}(1+\sqrt{\frac2N}W_{1,N}+\frac{2F_N\delta}N)}\bigg]\\
&\leq\bigg(\E\Big[\left(\frac{2(\ad-F_N\delta)}N-\sqrt{\frac2N}W_{1,N}\right)^2\left(W_{1,N}^2+
\frac{2F_N^2\delta^2}N+\frac{2\sqrt{2}F_NW_{1,N}\delta}{N^\frac12}\right)^2\Big]\bigg)^\frac12\\
&\quad\;.\bigg(\E\Big[(\bd+\frac12\norm{\C_0^{-\frac12}u^{(k)}_\delta}^2_{\R^N})^{-2}\Big]\bigg)^\frac12.
\end{align*}
The square root of the first expectation on the right hand side of the inequality is of order $N^{-\frac12}$, 
while by Lemma \ref{ch3:lemdres} the square root of the second expectation is of order $N^{-1}$ for almost 
all $\by$. Combining we get that $\E^\zeta[e_{1,N}]=\mathcal{O}(N^{-\frac32})$, almost surely with respect to 
$\by,$ hence \begin{align*}\E[\de{k+1}_N-\de{k}_N|\de{k}_N=\delta]&=\frac2N\left((1+\ad)\delta-\E^\zeta[F_N]
\delta^2\right)+\mathcal{O}(N^{-\frac32}),
\end{align*}$\by$-almost surely.

For the variance of the step, we have 
\begin{align*}\Var\left[\de{k+1}_N-\de{k}_N|\de{k}_N=\delta\right]=&\E\left[(\de{k+1}_N-\de{k}_N)^2|\de{k}_N=\delta\right]-\E\left[\de{k+1}_N-\de{k}_N|\de{k}_N=\delta
\right]^2,
\end{align*}
where by the first part of the proof the second term is $\mathcal{O}(N^{-2})$. 
Thus, we need only consider
the first term, which will be shown to be $\mathcal{O}(N^{-1})$.
By equation (\ref{ch3:eq:sp1}) we have

\begin{align*}\E\left[(\de{k+1}_N-\de{k}_N)^2|\de{k}_N=\delta\right]&=\delta^2\E\left[\left(\frac{\frac{2\ad}
N+\sqrt{\frac2N+\frac{4\ad}{N^2}}W_{2,N}-\sqrt{\frac2N}W_{1,N}-\frac{2F_N\delta}N}{1+\sqrt{\frac2N}
W_{1,N}+\frac{2F_N\delta}N}\right)^2\right]\\
&=\delta^2\E\left[\frac{\frac{2W_{2,N}^2}N+\frac{2W_{1,N}^2}N+\frac{V_N}{N^\frac32}}{\left(1+\sqrt{\frac2N}
W_{1,N}+\frac{2F_N\delta}N\right)^2}\right],
\end{align*}
where the random variable $V_N$ depends only on $W_{1,N}$ 
and $F_N$ and has higher order moments which are bounded uniformly in $N$, $\by$-almost surely (the 
dependence on $W_{2,N}$ disappears by the independence of $W_{2,N}$ and $\zeta$ and the fact that 
both $W_{2,N}$ has mean zero and variance one). Using the identity $\frac{1}{(1+x)^2}=1-2x+
\frac{3x^2+2x^3}{(1+x)^2}$, we get 
\begin{align*}\E&\left[(\de{k+1}_N-\de{k}_N)^2|\de{k}_N=\delta\right]\\=&\delta^2\E\left[\left(\frac{2W_{2,N}
^2}N+\frac{2W_{1,N}^2}N+\frac{V_N}{N^\frac32}\right)\left(1-2\sqrt{\frac2N}W_{1,N}-\frac4NF_N\delta\right)
\right]+\E[e_{2,N}],
\end{align*}
where \begin{align*}e_{2,N}&=\delta^2\left(\frac{2W_{2,N}^2}N+\frac{2W_{1,N}^2}N+\frac{V_N}{N^
\frac32}\right)\frac{3\left(\sqrt{\frac2N}W_{1,N}+\frac{2F_N\delta}N\right)^2+2\left(\sqrt{\frac2N}W_{1,N}+
\frac{2F_N\delta}N\right)^3}{\left(1+\sqrt{\frac2N}W_{1,N}+\frac{2F_N\delta}N\right)^2}\\
&:=\frac{E_N\delta^2}{\left(1+\sqrt{\frac2N}W_{1,N}+\frac{2F_N\delta}N\right)^2}.\end{align*}
Using the fact that $\by$-almost surely $W_{1,N}$, $F_N$ and $V_N$ have moments of all positive orders 
which are bounded uniformly in $N$, by H\"older inequality (we do not need to consider higher order 
moments for $W_{2,N}$ here, because it is independent with $W_{1,N}$ and $F_N,$ hence bounding 
terms involving $W_{2,N}$ does not require the use of H\"older's inequality which needs higher moments), 
we get that \begin{align*}
\E[(\de{k+1}_N-\de{k}_N)^2|\de{k}_N=\delta]&=\frac{2\delta^2}{N}\left(\E[W_{2,N}^2]+\E[W_{1,N}^2]\right)+
\mathcal{O}(N^{-\frac32})+\E[e_{2,N}],
\end{align*}
$\by$-almost surely.
For the residual $e_{2,N}$,  as before using Cauchy-Schwarz inequality and (\ref{ch3:eq:denom}), 
\begin{align*}
\E[e_{2,N}]&\leq\frac{N^{2}}4\big(\E[E_N^2]\big)^\frac12\bigg(\E[(\bd+\frac12\norm{\C_0^{-\frac12}u^{(k)}_
\delta}^2_{\R^N})^{-4}]\bigg)^\frac12.\end{align*}
Since by Lemma \ref{ch3:lemgam} the first four moments of $W_{2,N}$ are also bounded uniformly in $N$, 
the square root of the first expectation on the right hand side is of order $N^{-2}$, while by Lemma 
\ref{ch3:lemdres} the square root of the second expectation is of order $N^{-2}$, for almost all $\by$. 
Combining we get $\E^\zeta[e_{2,N}]=\mathcal{O}(N^{-2})$, almost surely with respect to $\by$, hence since 
$\E[W_{1,N}^2]=\E[W_{2,N}^2]=1$,
\begin{align*}
\E[(\de{k+1}_N-\de{k}_N)^2|\de{k}_N=\delta]
&=\frac{4\delta^2}{N}+\mathcal{O}(N^{-\frac32}),
\end{align*}$\by$-almost surely.
Concatenating, we get the result. 
\end{proof}

\subsection{Proof of Lemma \ref{ch3:lem1}}\label{ch3:ssec:ap1} 
\begin{proof}
Let $\{e_j\}_{j=1}^N$ be any orthonormal basis of $\R^N$ (with respect to the possibly scaled norm $\smnorm{\cdot}_{\R^N}$) and for any $w\in\R^N$ write $w_j:=\pr{w}{e_j}_{\R^N}$. We then have that $\zeta=\sumn\zeta_je_j$ where $\{\zeta_j\}_{j=1}^N$  is a sequence of independent standard normal random variables.
Using (\ref{ch3:eq:uu}) we have,
\begin{align*}\norm{\C_0^{-\frac12}u^{(k)}_\delta}_{\R^N}^2&=\norm{\C_0^{-\frac12}m_{\lambda,\delta}({\yn})}_{\R^N}^2+\norm{\C_0^{-\frac12}\C_{\lambda,\delta}^\frac12\zeta}^2_{\R^N}+2\pr{\C_0^{-\frac12}m_{\lambda,\delta}({\yn})}{\C_0^{-\frac12}\C_{\lambda,\delta}^\frac12\zeta}_{\R^N}\\
&:=A_N+B_N+C_N.\end{align*}
Under Assumptions \ref{ch3:ass1}, we can analyze each term as follows:
\begin{enumerate}
\item[A)] by Assumption \ref{ch3:ass1}(i), for almost all data $\by$,  this term and all its positive integer powers are bounded uniformly in $N$. 
\item[B)] the second term can be written as\begin{align*}\norm{\C_0^{-\frac12}\C_{\lambda,\delta}^\frac12\zeta}_{\R^N}^2&=\pr{\C_0^{-\frac12}\C_{\lambda,\delta}^\frac12\zeta}{\C_0^{-\frac12}\C_{\lambda,\delta}^\frac12\zeta}_{\R^N}=\pr{\C_{\lambda,\delta}^\frac12\C_0^{-1}\C_{\lambda,\delta}^\frac12\zeta}{\zeta}_{\R^N}\\&=\delta^{-1}\pr{\C_{\lambda,\delta}^\frac12(\C_{\lambda,\delta}^{-1}-\lambda K^\ast \C_1^{-1} K)\C_{\lambda,\delta}^\frac12\zeta}{\zeta}_{\R^N}
=\delta^{-1}\norm{\zeta}_{\R^N}^2-\delta^{-1}\lambda\norm{\C_1^{-\frac12}K\C_{\lambda,\delta}^\frac12\zeta}^2_{\R^N}\\&
:=b_{1,N}-b_{2,N},
\end{align*}
where
\begin{enumerate}
\item[b1)] $b_{1,N}=\delta^{-1}\norm{\zeta}^2_{\R^N}=\frac{N}{\delta}+\frac{1}{\delta}\sumn(\zeta_j^2-1):=\frac{N}\delta+\frac{\sqrt{2N}}{\delta}W_{1,N},$ where as $N\to\infty$, $W_{1,N}=\frac1{\sqrt{2N}}\sumn(\zeta_j^2-1)$ converges weakly to a standard normal random variable by the Central Limit Theorem and by Lemma \ref{ch3:lemmom} has all positive integer moments  bounded uniformly in $N;$
\item[b2)] for $b_{2,N}$ we have by Lemma \ref{ch3:asslem1}(ii) that $\E^{\zeta}[b_{2,N}]$ is uniformly bounded in $N$. In fact using Lemma \ref{ch3:kollem} together with Lemma \ref{ch3:asslem1}(ii), we get that for any $p\in\N$, $\E^\zeta[b_{2,N}^p]$ is bounded independently of $N$.

\end{enumerate}
\item[C)] for the third term we have 
\begin{align*}
\pr{\C_0^{-\frac12}m_{\lambda,\delta}({\yn})}{\C_0^{-\frac12}\C_{\lambda,\delta}^\frac12\zeta}_{\R^N}&=\pr{(\C_0^{-\frac12}\C_{\lambda,\delta}^\frac12)^\ast\C_0^{-\frac12}m_{\lambda,\delta}({\yn})}{\zeta}_{\R^N}=\sumn((\C_0^{-\frac12}\C_{\lambda,\delta}^\frac12)^\ast\C_0^{-\frac12}m_{\lambda,\delta}({\yn}))_j\zeta_j.\end{align*}
 It holds that 
\begin{align*}
\sumn((\C_0^{-\frac12}\C_{\lambda,\delta}^\frac12)^\ast \C_0^{-\frac12}m_{\lambda,\delta}({\yn}))_j^2=\norm{(\C_0^{-\frac12}\C_{\lambda,\delta}^\frac12)^\ast \C_0^{-\frac12}m_{\lambda,\delta}({\yn})}^2_{\R^N},
\end{align*}
and we claim that the norm on the right hand side is uniformly bounded in $N$ $\by$-almost surely. Indeed, by (\ref{ch3:eq:prec}), the Cauchy-Schwarz inequality and the non-negative definiteness of the matrix $\C_{0}^{\frac12}K^\ast \C_1^{-1} K\C_{0}^\frac12$, we have\begin{align*}
\norm{(\C_{0}^{-\frac12}\C_{\lambda,\delta}^\frac12)^\ast u}^2_{\R^N}&=\pr{\C_{0}^{-\frac12}\C_{\lambda,\delta}\C_{0}^{-\frac12}u}{u}_{\R^N}=\pr{\delta^{-1}(I+\frac{\lambda}{\delta}\C_{0}^{\frac12}K^\ast \C_1^{-1} K\C_{0}^\frac12)^{-1}u}{u}_{\R^N}\\
&\leq\norm{\delta^{-1}(I+\frac{\lambda}{\delta}\C_{0}^{\frac12}K^\ast \C_1^{-1} K\C_{0}^\frac12)^{-1}u}_{\R^N}\norm{u}_{\R^N}\leq\delta^{-1}\norm{u}^2_{\R^N}.
\end{align*}
Combining with Assumption \ref{ch3:ass1}(i) we get the claim and therefore
 by Lemma \ref{ch3:sumlem} below we get that the third term has $\by$-almost surely all even moments uniformly bounded in $N$. 
\end{enumerate}
We define $F_N=\bd+\frac{A_N-b_{2,N}+C_N}2$ and observe that since all terms have bounded moments of every order uniformly in $N$ $\by$-almost surely, H\"older's inequality secures that $F_N$ also has bounded moments of every order uniformly in $N$ almost surely with respect to $\by$. \end{proof}

\subsection{Negative moments of the rate parameter in the $\delta$ draw}\label{ch3:ssec:ap3}{\ }

\begin{lemma}\label{ch3:lemdres}
Let $u^{(k)}_\delta$ be as in (\ref{ch3:eq:uu}), for any $\delta, \lambda>0$. Under Assumptions \ref{ch3:ass1}, we have  \begin{align*}
\E^\zeta\bigg[(\bd+\frac12\norm{\C_0^{-\frac12}u_\delta^{(k)}}_{\R^N}^2)^{-2i}\bigg]=\mathcal{O}(N^{-2i}),
\end{align*}
as $N\to\infty$, almost surely with respect to $\by$, for $i=1,2$.
\end{lemma}
\begin{proof}
Without loss of generality we consider the case $\delta=\lambda=1$ and drop the $\lambda$ and $\delta$ dependence in $u, m$ and $\C$. To de-clutter our notation we also drop the dependence of $m$ on the data.
Since $\bd\geq0$ it suffices to show it for $\bd=0$. Formally, the random variable $\norm{\C_0^{-\frac12}u^{(k)}}_{\R^N}^2$ behaves like a chi-squared random variable with $N$ degrees of freedom. We estimate the squared norm by a random variable $Y_N$ of known moment generating function $M_{Y_N}(t)$,  and use the following formula from \cite{CDF81} for the calculation of negative moments of nonnegative random variables
\begin{align}\label{ch3:eq:nm}
\E[Y_N^{-l}]=\Gamma(l)^{-1}\int_0^\infty t^{l-1}M_{Y_N}(-t)dt, \; l\in\N.
\end{align}

We begin by showing that there exists a constant $c>0$ independent of $N$ such that $\norm{\C^{-\frac12}\C_0^\frac12v}_{\R^N}\leq c\norm{v}_{\R^N}$ for any $v\in\R^N$. We have,
\begin{align*}
\norm{\C^{-\frac12}\C_0^\frac12v}_{\R^N}^2&=\pr{\C_0^\frac12\C^{-1}\C_0^\frac12v}{v}_{\R^N}=\pr{(I+\C_0^\frac12K^\ast \C_1^{-1} K\C_0^\frac12)v}{v}_{\R^N}\\&=\norm{v}_{\R^N}^2+\norm{\C_1^{-\frac12}K\C_0^\frac12v}^2_{\R^N}\leq(1+c_2)\norm{v}_{\R^N}^2,
\end{align*}
by Lemma \ref{ch3:asslem1}(iii). The proved claim gives the estimate  \begin{align*}
\norm{\C_0^{-\frac12}u^{(k)}}^2_{\R^N}&=\norm{\C_0^{-\frac12}(m+\C^\frac12\zeta)}^2_{\R^N}=\norm{\C_0^{-\frac12}\C^\frac12(\C^{-\frac12}m+\zeta)}_{\R^N}^2\geq c^{-1}\norm{\C^{-\frac12}m+\zeta}^2_{\R^N},
\end{align*}
hence it suffices to show that almost surely with respect to $\by$ we have $\E^\zeta[Y_N^{-2i}]=\mathcal{O}(N^{-2i})$, for $Y_N:=\norm{\C^{-\frac12}m+\zeta}^2_{\R^N}$.
Indeed, let $\{e_j\}_{j=1}^N$ be any orthonormal basis of $\R^N$ (with respect to the possibly scaled norm $\smnorm{\cdot}_{\R^N}$), and define $w_j:=\pr{w}{e_j}$ for any $w\in\R^N$. Then we have  \begin{align*}
Y_N=\sumn((\C^{-\frac12}m)_j+\zeta_j)^2,
\end{align*}
where $\zeta_j\sim\G(0,1)$ are the mutually independent components of the white noise $\zeta$ and $(\C^{-\frac12}m)_j$ are independent of $\zeta$, therefore $Y_N$ is a non-central chi-squared random variable with $N$ degrees of freedom and non-centrality parameter $p_N:=\sumn (\C^{-\frac12}m)_j^2\geq0$.
The definition and properties of the non-central chi-squared distribution can be found in \cite{JKB95}, where in particular, we find the moment generating function of $Y_N$ 
\begin{align*}
M_{Y_N}(t)=(1-2t)^{-\frac{N}2}\exp\big(\frac{p_Nt}{1-2t}\big),
\end{align*}
hence using (\ref{ch3:eq:nm}) we have for $i=1,2$, \begin{align*}
\E^\zeta[Y_N^{-2i}]&=\Gamma(2i)^{-1}\int_0^\infty t^{2i-1}(1+2t)^{-\frac{N}2}\exp\big(\frac{-p_Nt}{1+2t}\big)dt\\
&\leq c\int_0^\infty t^{2i-1}(1+2t)^{-\frac{N}2}dt=\mathcal{O}(N^{-2i}),
\end{align*}provided $N>4i$, where the last integral can by calculated by integration by parts.
\end{proof}

\subsection{Technical lemmas}\label{ch3:ssec:ap4}{\ }

\begin{lemma}\label{ch3:sumlem}
Let $\{X_j\}$ be a sequence of random variables, such that $X_j=c_jY_j$, where the $Y_j, \;j\in\N$ are independent and identically distributed random variables with finite even moments up to order $2r\in\N$ and zero odd moments, and the $c_j, \;j\in\N$ are deterministic real numbers. Then for any $N\in\N$,
\begin{align*}\E[(\sumn X_j)^{2r}]\leq \kappa(\sumn c_j^2)^r,\end{align*}where $\kappa=\E[Y_1^{2r}]>0$ is independent of $N$.
\end{lemma}
\begin{proof}
Denote by $m_n$ the $2n$-th moment of $Y_1$, $m_n=\E[Y_1^{2n}].$ Observe that since by H\"older's inequality for $0<s\leq t$, 
$\E[|Y_1|^s]^\frac1s\leq \E[|Y_1|^t]^\frac1t$, we have that for $n_1,...,n_q>0$ such that $n_1+...+n_q=r$ \begin{align*}m_{n_1}...m_{n_q}\leq\E[Y_1^{2r}]^\frac{n_1+...+n_q}{r}=\E[Y_1^{2r}].\end{align*} Combining with the fact that the random variables $Y_j$ are independent with zero odd moments,
\begin{align*}
\E[(\sumn X_j)^{2r}]&=\sumn c_j^{2r}m_r+\sum_{j_1\neq j_2}^Nc_{j_1}^{2(r-1)}m_{r-1}c_{j_2}^2m_1+\sum_{j_1\neq j_2}^Nc_{j_1}^{2(r-2)}m_{r-2}c_{j_2}^4m_2\\&+...+\sum_{j_1\neq j_2\neq...\neq j_r}^Nc_{j_1}^2c_{j_2}^2...c_{j_r}^2m_1^r\leq m_r(\sumn c_j^2)^r.
\end{align*}\end{proof}

\begin{lemma}\label{ch3:kollem}
For any $p\in\N$, there exists a constant $c=c(p)\geq0$, independent of $N$ such that for any centered Gaussian random variable $x_N$ in $\R^N$, it holds 
\begin{equation*}\E[\norm{x_N}^{2p}_{\R^N}]\leq c(p)(\E[\norm{x_N}^2_{\R^N}])^p.\end{equation*}
\end{lemma}
\begin{proof}
Direct consequence of \cite[Corollary 2.17]{DZ92}.
\end{proof}

\begin{lemma}\label{ch3:lemmom}
Let $(\gamma_j)_{j\in\N}$ be a sequence of independent standard normal random variables and define $G_N:=\frac{1}{\sqrt{2N}}\sumn (\gamma_j^2-1).$ Then all the positive integer moments of $G_N$ are bounded uniformly in $N$.
\end{lemma} 
\begin{proof}
 For $k\in\N$, we have 
$\E[G_{N}^{k}]=\frac{1}{(2N)^{\frac{k}2}}\sum_{j_1,...,j_{k}}^N\E[(\gamma_{j_1}^2-1)...(\gamma_{j_{k}}^2-1)].
$
Since $\gamma_{j}^2-1$ are independent and identically distributed with finite moments of every order, the sum on the right hand side has a dependence on $N$ determined by the total number of non zero terms in the summation.
By independence and the fact that $\E[\gamma_j^2-1]=0$, all the terms in the sum which contain a term with an index $j_i$ which occurs only once in the product are equal to zero.  
We thus have that if $k$ is even the sum on the right hand side is of order $N^{\frac{k}2}$,
 while if $k$ is odd
it is of order $N^{\frac{k-1}2}$. In both cases the $k$-th moment of $G_{N}$ is bounded uniformly in $N$.
\end{proof}

\begin{lemma}\label{ch3:lemgam}
Let $\Gamma_N\sim\Ga(\upalpha+\frac{N}2,\frac{N}2)$, for $\upalpha>0$, and define  \begin{align*}\Theta_{N}:=\frac{\Gamma_N-1-\frac{2\upalpha}N}{\sqrt{\frac{2}N+\frac{4\upalpha}{N^2}}}.\end{align*} Then the first four moments of $\Theta_N$ are bounded uniformly in $N$. 
\end{lemma}
\begin{proof}
The random variable $\Ga(a,1)$ has mean and variance $a$ and third and fourth central moments $2a$ and $3a^2+6a$ respectively, \cite{JKB94}. Hence by the scaling property of the gamma distribution, $\Gamma_N\stackrel{\mathcal{L}}{=}\frac{2}N\Ga(\upalpha+\frac{N}2,1)$ has mean $1+\frac{2\upalpha}N$, variance $\frac{2}N+\frac{4\upalpha}{N^2}$,  and third and fourth central moments which are both of order $N^{-2}$. It is thus straightforward to see that $\Theta_{N}$ has mean zero, variance equal to one, and since the denominator in $\Theta_{N}$ is of order $N^{-\frac12}$ it has third and fourth moments which are  $\mathcal{O}(N^{-\frac12})$ and $\mathcal{O}(1)$ respectively. \end{proof}

\begin{lemma}\label{ch3:asslem1}
Under Assumptions \ref{ch3:ass1}, we have that for any $\lambda,\delta>0$,
\begin{enumerate}
\item[i)] $\tr(\C_1^{-\frac12}K\C_{\lambda,\delta}K^\ast \C_1^{-\frac12})\leq c_2\delta^{-1};$
\item[ii)] $\E^{\theta}\norm{\C_1^{-\frac12}K\C_{\lambda,\delta}^\frac12\theta}_{\R^N}^2\leq c_2\delta^{-1},$ where $\theta$ is a Gaussian white noise in $\R^N$;
\item[iii)]$\norm{\C_1^{-\frac12}K\C_0^\frac12}_{2,N}\leq \sqrt{c_2};$
\end{enumerate}

where $c_2$ is defined in Assumption \ref{ch3:ass1}(ii).
\end{lemma}
\begin{proof}{\ }
\begin{enumerate}
\item[i)]By (\ref{ch3:eq:prec}), we have \begin{align*}\C_1^{-\frac12}K\C_{\lambda,\delta}K^\ast \C_1^{-\frac12}=\delta^{-1}\C_1^{-\frac12}K\C_0^\frac12(I+\frac{\lambda}{\delta}\C_0^\frac12K^\ast\C_1^{-1}K\C_0^\frac12)^{-1}\C_0^\frac12K^\ast\C_1^{-\frac12},\end{align*} hence the fact that for any matrix $A\in\R^{N\times N}$ it holds $\tr(A(I+cA^\ast A)A^\ast)\leq \tr(AA^\ast)$ for any $c>0$, together with Assumption \ref{ch3:ass1}(ii) give the claim.
\item[ii)]It is well known that for $x\sim\G(0,\Sigma)$, $\E\norm{x}^2_{\R^N}=\tr(\Sigma)$. 
Since for $\theta\sim\G(0,I)$ we have $\C_1^{-\frac12}K\C_{\lambda,\delta}^\frac12\theta\sim\G(0,\C_1^{-\frac12}K\C_{\lambda,\delta}K^\ast\C_1^{-\frac12})$, the claim follows from part (i).
\item[iii)]It is well known that for any matrix $A\in\R^{N\times N}$, the Euclidean norm satisfies $\norm{A}_{2,N}=\norm{A^\ast}_{2,N}=\sqrt{\rho(A^\ast A)}\leq \sqrt{\tr(A^\ast A)}$ where $\rho(B)$ is the spectral radius of the matrix $B$. Hence we have $\norm{\C_1^{-\frac12}K\C_0^\frac12}_{2,N}\leq \sqrt{\tr(\C_1^{-\frac12}K\C_0 K^\ast\C_1^{-\frac12})}\leq\sqrt{c_2},$ by Assumption \ref{ch3:ass1}(ii).\end{enumerate}
\end{proof}

\bibliographystyle{plain}
\bibliography{hierarchicalbibarxiv}

\begin{thebibliography}{10}

\bibitem{RA90}
R.~J. Adler.
\newblock {\em An introduction to continuity, extrema, and related topics for
  general {G}aussian processes}.
\newblock Institute of Mathematical Statistics Lecture Notes---Monograph
  Series, 12. Institute of Mathematical Statistics, Hayward, CA, 1990.

\bibitem{SA13}
S.~Agapiou.
\newblock Aspects of bayesian inverse problems.
\newblock {\em Ph.D. Thesis, University of Warwick}, 2013.

\bibitem{ALS13}
S.~Agapiou, S.~Larsson, and A.~M. Stuart.
\newblock Posterior contraction rates for the {B}ayesian approach to linear
  ill-posed inverse problems.
\newblock {\em Stochastic Processes and their Applications}, 2013.

\bibitem{ASZ12}
S.~Agapiou, A.~M. Stuart, and Y.~X. Zhang.
\newblock {B}ayesian posterior contraction rates for linear severely ill-posed
  inverse problems.
\newblock {\em arxiv:1210.1563}, 2012.

\bibitem{AR09}
C.~Andrieu and G.~O. Roberts.
\newblock The pseudo-marginal approach for efficient {M}onte {C}arlo
  computations.
\newblock {\em The Annals of Statistics}, pages 697--725, 2009.

\bibitem{JB12}
J.~M. Bardsley.
\newblock M{CMC}-based image reconstruction with uncertainty quantification.
\newblock {\em SIAM J. Sci. Comput.}, 34(3):A1316--A1332, 2012.

\bibitem{BS09}
J.~M. Bernardo and A.~F.~M. Smith.
\newblock {\em Bayesian theory}, volume 405.
\newblock Wiley, 2009.

\bibitem{BHMR07}
N.~Bissantz, T.~Hohage, A.~Munk, and F.~Ruymgaart.
\newblock Convergence rates of general regularization methods for statistical
  inverse problems and applications.
\newblock {\em SIAM Journal on Numerical Analysis}, 45(6):2610--2636, 2007.

\bibitem{BL96}
L.~D. Brown and M.~G. Low.
\newblock Asymptotic equivalence of nonparametric regression and white noise.
\newblock {\em The Annals of Statistics}, 24(6):2384--2398, 1996.

\bibitem{CHPS09}
D.~Calvetti, H.~Hakula, S.~Pursiainen, and E.~Somersalo.
\newblock Conditionally {G}aussian hypermodels for cerebral source
  localization.
\newblock {\em SIAM Journal on Imaging Sciences}, 2(3):879--909, 2009.

\bibitem{CS08}
D.~Calvetti and E.~Somersalo.
\newblock Hypermodels in the {B}ayesian imaging framework.
\newblock {\em Inverse Problems}, 24(3):034013, 2008.

\bibitem{LC08}
L~Cavalier.
\newblock Nonparametric statistical inverse problems.
\newblock {\em Inverse Problems}, 24(3):034004, 2008.

\bibitem{CDS10}
S.~L. Cotter, M.~Dashti, and A.~M. Stuart.
\newblock Approximation of {B}ayesian inverse problems for pdes.
\newblock {\em SIAM Journal on Numerical Analysis}, 48(1):322--345, 2010.

\bibitem{CRSW13}
S.~L. Cotter, G.~O. Roberts, A.~M. Stuart, and D.~White.
\newblock {MCMC} methods for functions: modifying old algorithms to make them
  faster.
\newblock {\em Statistical Science}, 28(3):424--446, 2013.

\bibitem{CDF81}
N.~Cressie, A.~S. Davis, J.~L. Folks, and G.~E. Policello~{II}.
\newblock The moment-generating function and negative integer moments.
\newblock {\em The American Statistician}, 35(3):148--150, 1981.

\bibitem{DP05}
G.~Da~Prato.
\newblock {\em An introduction to infinite-dimensional analysis}.
\newblock Springer, 2005.

\bibitem{DZ92}
G.~Da~Prato and J.~Zabczyk.
\newblock Stochastic equations in infinite dimensions, volume 44 of
  encyclopedia of mathematics and its applications, 1992.

\bibitem{FG13}
M.~Filippone and M.~Girolami.
\newblock Pseudo-{M}arginal {B}ayesian inference for {G}aussian processes.
\newblock 10 2013.

\bibitem{GGR96}
A.~Gelman, G.~O. Roberts, and W.~R. Gilks.
\newblock Efficient {M}etropolis jumping rules.
\newblock {\em Bayesian statistics}, 5:599--608, 1996.

\bibitem{JKB94}
N.~L. Johnson, S.~Kotz, and N.~Balakrishnan.
\newblock {\em Continuous univariate distributions. {V}ol. 1}.
\newblock Wiley Series in Probability and Mathematical Statistics: Applied
  Probability and Statistics. John Wiley \& Sons Inc., New York, second
  edition, 1994.
\newblock A Wiley-Interscience Publication.

\bibitem{JKB95}
N.~L. Johnson, S.~Kotz, and N.~Balakrishnan.
\newblock {\em Continuous univariate distributions. {V}ol. 2}.
\newblock Wiley Series in Probability and Mathematical Statistics: Applied
  Probability and Statistics. John Wiley \& Sons Inc., New York, second
  edition, 1995.
\newblock A Wiley-Interscience Publication.

\bibitem{JKPR04}
I.~M. Johnstone, G.~Kerkyacharian, D.~Picard, and M.~Raimondo.
\newblock Wavelet deconvolution in a periodic setting.
\newblock {\em Journal of the Royal Statistical Society: Series B (Statistical
  Methodology)}, 66(3):547--573, 2004.

\bibitem{KS05}
J.~P. Kaipio and E.~Somersalo.
\newblock {\em Statistical and computational inverse problems}, volume 160.
\newblock Springer, 2005.

\bibitem{KSVZ12}
B.~T. Knapik, B.~T. Szab{\'o}, A.~W. van~der Vaart, and J.~H. van Zanten.
\newblock {B}ayes procedures for adaptive inference in nonparametric inverse
  problems.
\newblock {\em arXiv:1209.3628}, 2012.

\bibitem{KVZ12}
B.~T. Knapik, A.~W. van Der~Vaart, and J.~H. van Zanten.
\newblock {B}ayesian inverse problems with {G}aussian priors.
\newblock {\em The Annals of Statistics}, 39(5):2626--2657, 2011.

\bibitem{KVZ13}
B.~T. Knapik, A.~W. van~der Vaart, and J.~H. van Zanten.
\newblock {B}ayesian recovery of the initial condition for the heat equation.
\newblock {\em Communications in Statistics-Theory and Methods},
  42(7):1294--1313, 2013.

\bibitem{LPS89}
M.~S. Lehtinen, L.~Paivarinta, and E.~Somersalo.
\newblock Linear inverse problems for generalised random variables.
\newblock {\em Inverse Problems}, 5(4):599, 1989.

\bibitem{LS72}
D.~V. Lindley and A.~F.~M. Smith.
\newblock Bayes estimates for the linear model.
\newblock {\em Journal of the Royal Statistical Society. Series B
  (Methodological)}, pages 1--41, 1972.

\bibitem{AM84}
A.~Mandelbaum.
\newblock Linear estimators and measurable linear transformations on a
  {H}ilbert space.
\newblock {\em Zeitschrift f{\"u}r Wahrscheinlichkeitstheorie und Verwandte
  Gebiete}, 65(3):385--397, 1984.

\bibitem{MSZ13}
F.~van~der Meulen, M.~Schauer, and J.~H. van Zanten.
\newblock Reversible jump {MCMC} for nonparametric drift estimation for
  diffusion processes.
\newblock {\em Computational Statistics \&amp; Data Analysis}, 2013.

\bibitem{PPRS12}
O.~Papaspiliopoulos, Y.~Pokern, G.~O. Roberts, and A.~M. Stuart.
\newblock Nonparametric estimation of diffusions: a differential equations
  approach.
\newblock {\em Biometrika}, 99(3):511--531, 2012.

\bibitem{stable}
O.~Papaspiliopoulos and G.~O. Roberts.
\newblock Stability of the {G}ibbs sampler for {B}ayesian hierarchical models.
\newblock {\em Ann. Statist.}, 36(1):95--117, 2008.

\bibitem{PRS03}
O.~Papaspiliopoulos, G.~O. Roberts, and M.~Sk{\"o}ld.
\newblock Non-centered parameterisations for hierarchical models and data
  augmentation.
\newblock In {\em Bayesian Statistics 7: Proceedings of the Seventh Valencia
  International Meeting}. Oxford University Press, USA, 2003.

\bibitem{PRS07}
O.~Papaspiliopoulos, G.~O. Roberts, and M.~Sk{\"o}ld.
\newblock A general framework for the parametrization of hierarchical models.
\newblock {\em Statistical Science}, pages 59--73, 2007.

\bibitem{PSZ13}
Y.~Pokern, A.~M. Stuart, and J.~H. van Zanten.
\newblock Posterior consistency via precision operators for {B}ayesian
  nonparametric drift estimation in {SDE}s.
\newblock {\em Stochastic Process. Appl.}, 123(2):603--628, 2013.

\bibitem{KR13}
K.~Ray.
\newblock Bayesian inverse problems with non-conjugate priors.
\newblock {\em Electronic Journal of Statistics}, 7:2516--2549, 2013.

\bibitem{RS01}
G.~O. Roberts and O.~Stramer.
\newblock On inference for partially observed nonlinear diffusion models using
  the {M}etropolis--{H}astings algorithm.
\newblock {\em Biometrika}, 88(3):603--621, 2001.

\bibitem{SR09}
C.~Sherlock and G.~O. Roberts.
\newblock Optimal scaling of the random walk {M}etropolis on elliptically
  symmetric unimodal targets.
\newblock {\em Bernoulli}, 15(3):774--798, 2009.

\bibitem{AS10}
A.~M. Stuart.
\newblock Inverse problems: a {B}ayesian perspective.
\newblock {\em Acta Numer.}, 19:451--559, 2010.

\bibitem{SVZ13}
B.~T. Szab{\'o}, A.~W. van~der Vaart, and J.~H. van Zanten.
\newblock Empirical {B}ayes scaling of {G}aussian priors in the white noise
  model.
\newblock {\em Electronic Journal of Statistics}, 7:991--1018, 2013.

\bibitem{YM11}
Y.~Yu and X.~L. Meng.
\newblock To center or not to center: that is not the question---an
  ancillarity--sufficiency interweaving strategy ({ASIS}) for boosting {MCMC}
  efficiency.
\newblock {\em Journal of Computational and Graphical Statistics},
  20(3):531--570, 2011.

\bibitem{LZ00}
L.~H. Zhao.
\newblock Bayesian aspects of some nonparametric problems.
\newblock {\em Annals of statistics}, pages 532--552, 2000.

\end{thebibliography}
\end{document}